\documentclass[11pt,a4paper,reqno]{amsart} 

\usepackage{latexsym}
\usepackage{verbatim}
\usepackage{epsfig}
\usepackage{rotating,qtree,dsfont}
\usepackage{amssymb}
\usepackage[T1]{fontenc}
\usepackage{enumitem}
\usepackage{algorithm,ulem}
\usepackage{algpseudocode}
\usepackage{graphics}
\usepackage{graphicx}
\graphicspath{{./pics/}}

\newcommand{\bo}[1]{{\bf#1}}
\newcommand{\ds}{\displaystyle}

\def\clap#1{\hbox to0pt{\hss#1\hss}}

\def\Stepset#1#2#3#4#5#6#7#8#9{%
  \raisebox{-9pt}{%
  \setlength{\unitlength}{1.3pt}%
  \begin{picture}(20,20)(-10,-10)
    \put(-5,-5){\clap{$\scriptstyle\ifx1#1\bullet\else\cdot\fi$}}
    \put(0,-5){\clap{$\scriptstyle\ifx1#2\bullet\else\cdot\fi$}}
    \put(5,-5){\clap{$\scriptstyle\ifx1#3\bullet\else\cdot\fi$}}
    \put(-5,0){\clap{$\scriptstyle\ifx1#4\bullet\else\cdot\fi$}}
    \put(0,0){\clap{$\scriptstyle\ifx1#5\bullet\else\cdot\fi$}}
    \put(5,0){\clap{$\scriptstyle\ifx1#6\bullet\else\cdot\fi$}}
    \put(-5,5){\clap{$\scriptstyle\ifx1#7\bullet\else\cdot\fi$}}
    \put(0,5){\clap{$\scriptstyle\ifx1#8\bullet\else\cdot\fi$}}
    \put(5,5){\clap{$\scriptstyle\ifx1#9\bullet\else\cdot\fi$}}
  \end{picture}}\StepsetB}

\def\StepsetB#1#2#3#4#5#6#7#8{%
  \raisebox{-9pt}{%
  \setlength{\unitlength}{1.3pt}%
  \begin{picture}(20,20)(-10,-10)
    \put(-5,-5){\clap{$\scriptstyle\ifx1#1\bullet\else\cdot\fi$}}
    \put(0,-5){\clap{$\scriptstyle\ifx1#2\bullet\else\cdot\fi$}}
    \put(5,-5){\clap{$\scriptstyle\ifx1#3\bullet\else\cdot\fi$}}
    \put(-5,0){\clap{$\scriptstyle\ifx1#4\bullet\else\cdot\fi$}}
    \put(5,0){\clap{$\scriptstyle\ifx1#5\bullet\else\cdot\fi$}}
    \put(-5,5){\clap{$\scriptstyle\ifx1#6\bullet\else\cdot\fi$}}
    \put(0,5){\clap{$\scriptstyle\ifx1#7\bullet\else\cdot\fi$}}
    \put(5,5){\clap{$\scriptstyle\ifx1#8\bullet\else\cdot\fi$}}
  \end{picture}}\StepsetC}

\def\StepsetC#1#2#3#4#5#6#7#8#9{%
  \raisebox{-9pt}{%
  \setlength{\unitlength}{1.3pt}%
  \begin{picture}(20,20)(-10,-10)
    \put(-5,-5){\clap{$\scriptstyle\ifx1#1\bullet\else\cdot\fi$}}
    \put(0,-5){\clap{$\scriptstyle\ifx1#2\bullet\else\cdot\fi$}}
    \put(5,-5){\clap{$\scriptstyle\ifx1#3\bullet\else\cdot\fi$}}
    \put(-5,0){\clap{$\scriptstyle\ifx1#4\bullet\else\cdot\fi$}}
    \put(0,0){\clap{$\scriptstyle\ifx1#5\bullet\else\cdot\fi$}}
    \put(5,0){\clap{$\scriptstyle\ifx1#6\bullet\else\cdot\fi$}}
    \put(-5,5){\clap{$\scriptstyle\ifx1#7\bullet\else\cdot\fi$}}
    \put(0,5){\clap{$\scriptstyle\ifx1#8\bullet\else\cdot\fi$}}
    \put(5,5){\clap{$\scriptstyle\ifx1#9\bullet\else\cdot\fi$}}
  \end{picture}}}

\def\sequenceThreeD#1#2#3#4{%
  #2\quad\hbox to.6\hsize{$#3,\dots$\hfill}\quad\textrm{(#4)}
}

\makeatletter
\def\testb#1{\testb@i#1,,\@nil}%
\def\testb@i#1,#2,#3\@nil{%
  \draw[->, thick] (O) --++(#1);
  \ifx\relax#2\relax\else\testb@i#2,#3\@nil\fi}
\makeatother   
\newcommand{\makediag}[1]{
    \coordinate (O) at (0,0); \coordinate (N) at (0,0.8);
    \coordinate (NE) at (0.8,0.8); \coordinate (E) at (0.8,0);
    \coordinate (SE) at (0.8,-0.8); \coordinate (S) at (0,-0.8);
    \coordinate (SW) at (-0.8,-0.8);\coordinate (W) at (-0.8,0);
    \coordinate (NW) at (-0.8,0.8); \coordinate (B1) at (1.2,1.2);
    \coordinate (B2) at (-1.2,-1.2);
    \testb{#1}
} 
\newcommand{\diagr}[1]{
  \begin{tikzpicture}[scale=0.8]\makediag{#1}\end{tikzpicture}
}

\def\mySTEP(#1,#2,#3){\draw[thick, red, ->] (0,0,0) -- (#1,#2,#3);}

\def\stepsetpicturescalefactor{1}

\def\StepsetThreeD#1#2#3#4#5#6#7#8#9{%
   \begin{tikzpicture}[z={(.32cm,.39cm)},scale=\stepsetpicturescalefactor]
      \draw[gray]
            (-1,-1,-1)--(-1,-1,1)--(1,-1,1) (-1,0,-1)--(-1,0,1)--(1,0,1) (-1,-1,1)--(-1,1,1)
            (0,-1,-1)--(0,-1,1)--(0,1,1) (-1,1,0)--(-1,-1,0)--(1,-1,0);
      \draw[black]
            (-1,-1,-1)--(1,-1,-1)--(1,1,-1)--(-1,1,-1)--cycle (0,-1,-1)--(0,1,-1) (-1,0,-1)--(1,0,-1)
            (-1,1,-1)--(-1,1,1)--(1,1,1)--(1,1,-1) (-1,1,0)--(1,1,0) (0,1,1)--(0,1,-1)
            (1,1,1)--(1,-1,1)--(1,-1,-1) (1,0,1)--(1,0,-1) (1,1,0)--(1,-1,0);
      \ifx1#1\mySTEP(-1,-1,-1)\fi
      \ifx1#2\mySTEP(-1,-1,0)\fi
      \ifx1#3\mySTEP(-1,-1,1)\fi
      \ifx1#4\mySTEP(-1,0,-1)\fi
      \ifx1#5\mySTEP(-1,0,0)\fi
      \ifx1#6\mySTEP(-1,0,1)\fi
      \ifx1#7\mySTEP(-1,1,-1)\fi
      \ifx1#8\mySTEP(-1,1,0)\fi
      \ifx1#9\mySTEP(-1,1,1)\fi
  \StepsetThreeDxxA}
  \def\StepsetThreeDxxA#1#2#3#4#5#6#7#8{%
      \ifx1#1\mySTEP(0,-1,-1)\fi
      \ifx1#2\mySTEP(0,-1,0)\fi
      \ifx1#3\mySTEP(0,-1,1)\fi
      \ifx1#4\mySTEP(0,0,-1)\fi
      \ifx1#5\mySTEP(0,0,1)\fi
      \ifx1#6\mySTEP(0,1,-1)\fi
      \ifx1#7\mySTEP(0,1,0)\fi
      \ifx1#8\mySTEP(0,1,1)\fi
  \StepsetThreeDxxB}
  \def\StepsetThreeDxxB#1#2#3#4#5#6#7#8#9{%
      \ifx1#1\mySTEP(1,-1,-1)\fi
      \ifx1#2\mySTEP(1,-1,0)\fi
      \ifx1#3\mySTEP(1,-1,1)\fi
      \ifx1#4\mySTEP(1,0,-1)\fi
      \ifx1#5\mySTEP(1,0,0)\fi
      \ifx1#6\mySTEP(1,0,1)\fi
      \ifx1#7\mySTEP(1,1,-1)\fi
      \ifx1#8\mySTEP(1,1,0)\fi
      \ifx1#9\mySTEP(1,1,1)\fi
    \end{tikzpicture}
  }

\def\stepsetpicturescalefactor{.6}

\renewcommand {\epsilon}{\varepsilon}

\renewcommand {\leq}{\leqslant}
\renewcommand {\geq}{\geqslant}
\newcommand {\cov}{\textnormal{cov}}

\newtheorem{Theorem}{Theorem}
\newtheorem{Lemma}[Theorem]{Lemma}
\newtheorem{Proposition}[Theorem]{Proposition}
\newtheorem{Definition}[Theorem]{Definition}
\newtheorem{Corollary}[Theorem]{Corollary}
\newtheorem{Remark}[Theorem]{Remark}
\newtheorem{Example}[Theorem]{Example}

\usepackage{multicol,ulem,tikz}
\usepackage{tikz,wasysym}
\usetikzlibrary{mindmap}
\usetikzlibrary{shapes,backgrounds}
\usepackage{tikz-3dplot}
\usepackage{tabulary}
\usetikzlibrary{shapes,calc,positioning}
\usetikzlibrary{calc,intersections}

\usepackage{afterpage}
\usepackage{color,dsfont,hyperref}
\usepackage{tikz}
\usetikzlibrary{3d,calc,fadings,decorations.pathreplacing}

\tikzset{%
  >=latex, 
  inner sep=0pt,%
  outer sep=2pt,%
  mark coordinate/.style={inner sep=0pt,outer sep=0pt,minimum size=3pt,
    fill=black,circle}%
}

\usepackage{amssymb,amsfonts,amsmath,stmaryrd}

\begin{document}

\title[3D positive lattice walks and spherical triangles]{3D positive lattice walks and spherical triangles}

\begin{abstract}
In this paper we explore the asymptotic enumeration of three-dimensional excursions confined to the positive octant. As shown in \cite{DeWa-15}, both the exponential growth and the critical exponent admit universal formulas, respectively in terms of the inventory of the step set and of the principal Dirichlet eigenvalue of a certain spherical triangle, itself being characterized by the steps of the model. We focus on the critical exponent, and our main objective is to relate combinatorial properties of the step set (structure of the so-called group of the walk, existence of a Hadamard decomposition, existence of differential equations satisfied by the generating functions)\ to geometric or analytic properties of the associated spherical triangle (remarkable angles, tiling properties, existence of an exceptional closed-form formula for the principal eigenvalue). As in general the eigenvalues of the Dirichlet problem on a spherical triangle are not known in closed form, we also develop a finite-elements method to compute approximate values, typically with ten digits of precision. 
\end{abstract}

\author{B.\ Bogosel}
\thanks{B.\ Bogosel: CMAP, \'Ecole Polytechnique, 91120 Palaiseau, France}
\email{beniamin.bogosel@cmap.polytechnique.fr}

\author{V.\ Perrollaz}
\thanks{V.\ Perrollaz: Institut Denis Poisson, Universit\'e de Tours, Parc de Grandmont, 37200 Tours, France}
\email{vincent.perrollaz@lmpt.univ-tours.fr}

\author{K.\ Raschel}
\thanks{K.\ Raschel: CNRS \& Institut Denis Poisson, Universit\'e de Tours, Parc de Grandmont, 37200 Tours, France}
\email{raschel@math.cnrs.fr}

\author{A.\ Trotignon}
\thanks{A.\ Trotignon: Institut Denis Poisson, Universit\'e de Tours, Parc de Grandmont, 37200 Tours, France \& Simon Fraser University, Burnaby, BC, Canada}
\email{amelie.trotignon@lmpt.univ-tours.fr}

\thanks{This project has received funding from the European Research Council (ERC) under the European Union's Horizon 2020 research and innovation programme under the Grant Agreement No 759702.}

\thanks{Webpage of the article: \href{https://bit.ly/2J4Vf3X}{\url{https://bit.ly/2J4Vf3X}}}

\keywords{Enumerative combinatorics; lattice paths in the octant; asymptotic analysis; spherical geometry}
\subjclass{05A15}

\date{\today}

\maketitle

\begin{figure}[ht]\includegraphics[width=0.4\textwidth]{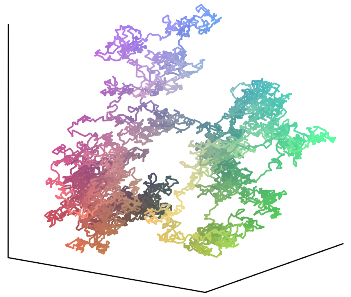}
\tdplotsetmaincoords{90}{90}
\begin{tikzpicture}[scale=3,tdplot_main_coords]

\tdplotsetthetaplanecoords{90}
\tdplotdrawarc[tdplot_rotated_coords,thick]{(0,0,0)}{0.8}{0}{360}{}{}
\tdplotsetrotatedcoords{60}{70}{0}
\tdplotdrawarc[dashed,tdplot_rotated_coords,name path=blue,color=blue]{(0,0,0)}{0.8}{0}{360}{}{}
\tdplotdrawarc[tdplot_rotated_coords]{(0,0,0)}{0.8}{0}{180}{}{}
\tdplotsetrotatedcoords{120}{110}{0}
\tdplotdrawarc[dashed,tdplot_rotated_coords,name path=green,color=green]{(0,0,0)}{0.8}{0}{360}{}{}
\tdplotdrawarc[tdplot_rotated_coords]{(0,0,0)}{0.8}{0}{180}{}{}
\tdplotsetrotatedcoords{220}{16}{0}
\tdplotdrawarc[dashed,tdplot_rotated_coords,name path=red,color=red]{(0,0,0)}{0.8}{0}{360}{}{}
\tdplotdrawarc[tdplot_rotated_coords]{(0,0,0)}{0.8}{-90}{90}{}{}


\path [name intersections={of={green and blue}, total=\n}]  
\foreach \i in {1,...,\n}{(intersection-\i) circle [radius=0.5pt] coordinate(gb\i){}};

\path [name intersections={of={green and red}, total=\n}]  
\foreach \i in {1,...,\n} {(intersection-\i) circle [radius=0.5pt]coordinate(gr\i){}};

\path[name intersections={of={red and blue}, total=\n}]  
\foreach \i in {1,...,\n}{(intersection-\i) circle [radius=0.5pt]coordinate(rb\i){}};

\shade[top color=gray,bottom color=blue,opacity=0.5]  
(rb3) to[bend left=7] (gr1) to[bend left=17] (gb2) to[bend left=15] cycle;

\draw (gb2) node[below]{$\beta$};
\draw (rb3) node[above left]{$\gamma$};
\draw (gr1) node[above right]{$\alpha$};
\end{tikzpicture}
\caption{Critical exponents in the asymptotics of 3D walks (and 3D Brownian motion as well)\ in the orthant $\mathbb N^3$ can be computed in terms of the smallest eigenvalue for the Dirichlet problem on spherical triangles}
\label{fig:3D-RandomWalk}
\end{figure}
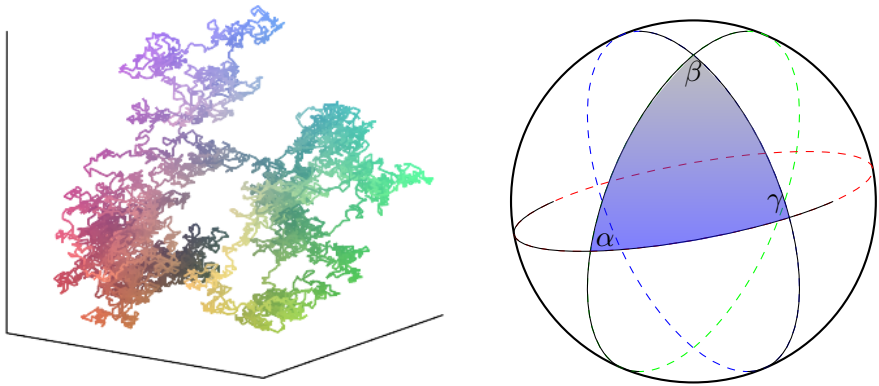

\tableofcontents

\section{Introduction}

\subsection*{Context}
The enumeration of lattice walks is an important topic in combinatorics. In addition to having various applications, it is connected to other mathematical fields such as probability theory. Recently, lots of consideration have been given to the enumeration of walks confined to cones. We will typically be considering walks on $\mathbb Z^d$ that start at the origin and consist of steps taken from $\mathcal S$, a finite subset of $\mathbb Z^d$. Most of the time we will constrain the walks in the orthant $\mathbb N^d$, with $\mathbb N$ denoting the set of non-negative integers $\{0,1,2,\ldots\}$.

\begin{figure}[bht]
\begin{tabular}{ccc}
\includegraphics[width=0.3\textwidth]{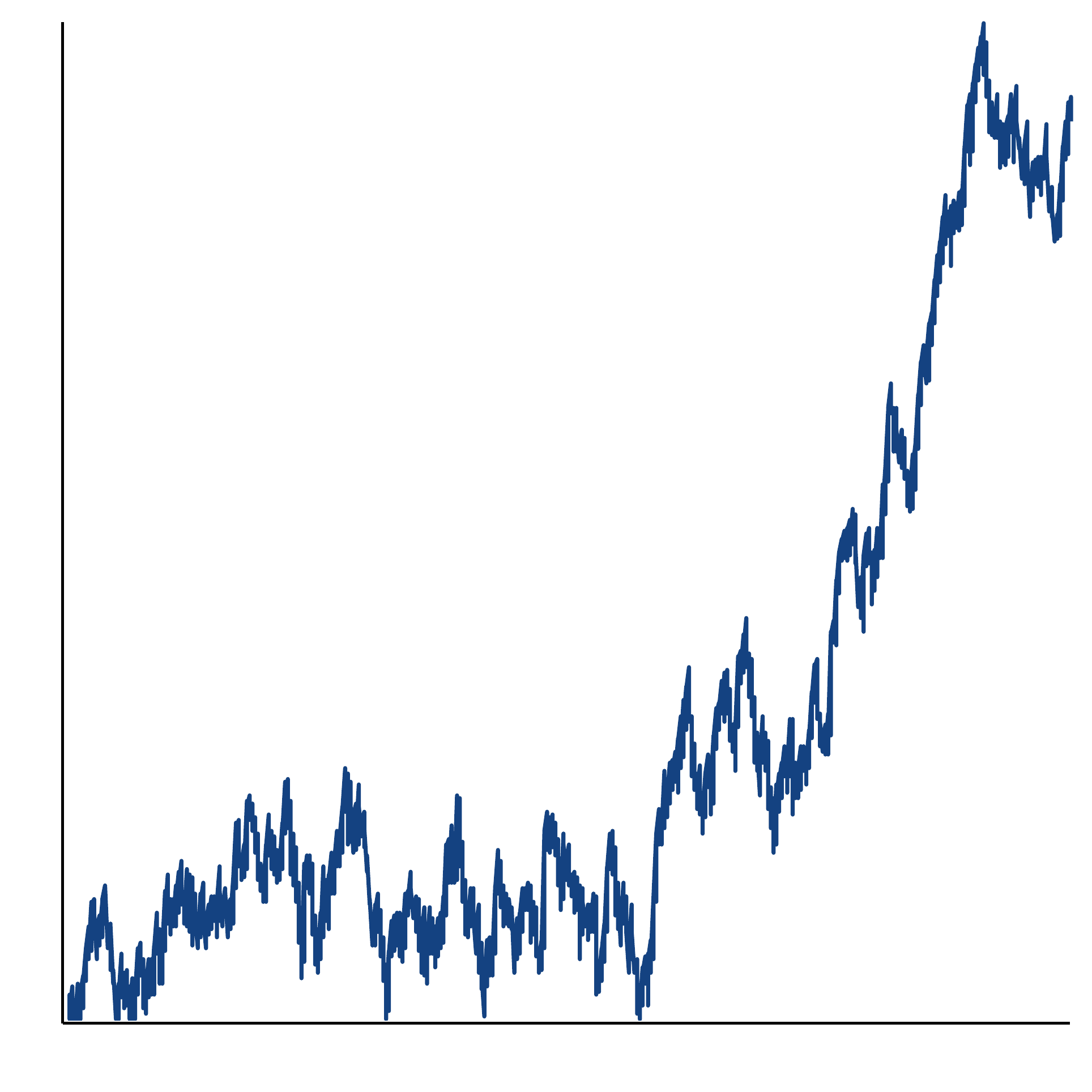}
\qquad&\qquad
\includegraphics[width=0.3\textwidth]{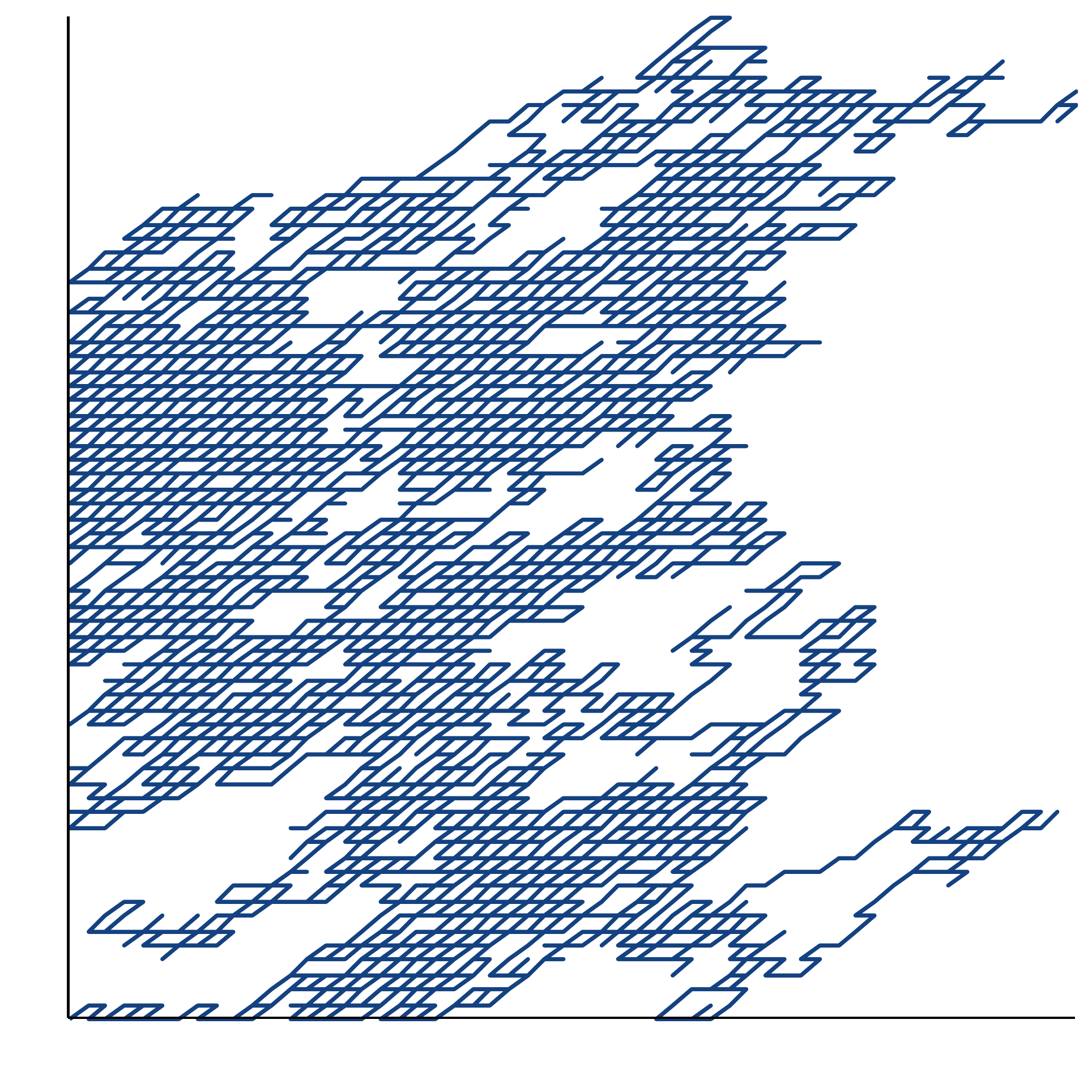}
\qquad&\qquad
\end{tabular}
\caption{Motzkin paths in $\mathbb N$ (with steps $(1,1)$, $(1,0)$ and $(1,-1)$) and Gessel's walks in $\mathbb N^2$ (with steps $(1,0)$, $(1,1)$, $(-1,0)$ and $(-1,-1)$)}
\label{fig:Motzkin-Gessel}
\end{figure}

In dimension $d=1$ (Figure \ref{fig:Motzkin-Gessel}, left) there is essentially one unique cone (the positive half-line), and positive (random) walks are very well understood, see in particular \cite{BeDo-94,BaFl-02}.

Following the seminal works \cite{FaIaMa-99,BMMi-10}, many recent papers deal with the enumeration of 2D walks with prescribed steps confined to the positive quadrant (Figure \ref{fig:Motzkin-Gessel}, right). In the case of small steps ($\mathcal S$ included in $\{0, \pm1\}^2$), various results have been obtained: exact and asymptotic expressions \cite{BMMi-10,BoKa-09,BoRaSa-14}, classification of the generating function according to the classes rational, algebraic, D-finite (that is, solution to a linear differential equation with polynomial coefficients) \cite{BMMi-10}, non-D-finite \cite{KuRa-12,BoRaSa-14}, and even non-differentially algebraic \cite{DrHaRoSi-18}. 

One of the most striking results in the quadrant walks world is the following: the generating function is D-finite if and only if a certain group associated with the step set $\mathcal S$ is finite. Remarkably this result connects an arithmetic property of the generating function to a geometric feature (the group, related to the symmetries of the step set). Non-convex cones (see \cite{BM-16} for the three quarter plane) as well as larger steps \cite{BoBMMe-18} have recently also been considered.

\begin{figure}[bht]
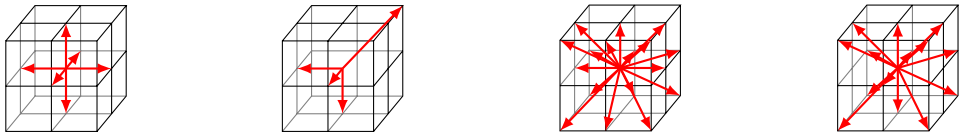

\begin{center}
\begin{tabular}{c@{\qquad\qquad\quad}c@{\qquad\qquad\quad}c@{\qquad\qquad\quad}c}
\StepsetThreeD00001000001011010000010000 &
\StepsetThreeD00001000001010000000000001 &
\StepsetThreeD10101111010101110101011110 &
\StepsetThreeD10100111001011010101001110
\end{tabular}  
\end{center}
  \caption{From left to right: the simple walk, Kreweras 3D model, a $(1,2)$-type Hadamard model and a $(2,1)$-type Hadamard model. These pictures are courtesy of Alin Bostan. As these perspective drawings might be difficult to read, we will prefer the cross-section views of the step sets as in Figure~\ref{fig:various_examples}}
\label{fig:various_examples_perpective}
\end{figure}

\begin{figure}[bht]
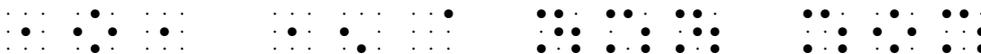

\begin{center}
\begin{tabular}{c@{\qquad}c@{\qquad}c@{\qquad}c}
\Stepset000 010 000   010 11 010   000 010 000 &
\Stepset000 010 000   010 10 000   000 000 001 &
\Stepset101 011 110   101 01 110   101 011 110 &
\Stepset101 001 110   010 11 010   101 001 110
\end{tabular}  
\end{center}
  \caption{For each model, the first diagram shows steps of the form $(i,j,-1)$, the second the steps $(i,j,0)$, and the third the steps $(i,j,1)$. The models are the same ones as in Figure \ref{fig:various_examples_perpective}. These cross-section views were first proposed in \cite{BoKa-09,BoBMKaMe-16}}
\label{fig:various_examples}
\end{figure}

In dimension three, determining whether the above equivalence between D-finiteness of the generating function and finiteness of the symmetry group holds or not remains an open problem. More generally, much less is known on 3D lattice walks confined to the non-negative octant $\mathbb N^3$. An intrinsic difficulty lies in the number of models to handle: more than 11 millions \cite{BoBMKaMe-16}. (Compare with $79$ quadrant models.) The first work is an empirical classification by Bostan and Kauers \cite{BoKa-09} of the models with at most five steps. Then in \cite{BoBMKaMe-16}, Bostan, Bousquet-M\'elou, Kauers and Melczer study models of cardinality at most six. They introduce key concepts: the dimensionality (1D, 2D or 3D) of a model, the group of the model, the Hadamard structure (roughly speaking, it is a generalization of Cartesian products of lower dimensional models). These notions will be made precise in Section~\ref{sec:preliminaries}. Furthermore, the authors of \cite{BoBMKaMe-16} classify the models with respect to these concepts and compute, in various cases (but only in presence of a finite group), the generating function
\begin{equation}
\label{eq:generating_function}
     O(x,y,z;t) = \sum_{i,j,k,n\geq0} o(i,j,k;n)x^{i}y^{j}z^{k}t^{n},
\end{equation}
where $o(i,j,k;n)$ is the number of $n$-step walks in the octant starting at the origin $(0,0,0)$ and ending at position $(i,j,k)$. In majority, the techniques used in \cite{BoBMKaMe-16} to solve finite group models are the algebraic kernel method and computer algebra (using the guessing-and-proving paradigm). 

The classification (in particular with respect to the finiteness of the group and the Hadamard structure) of the 3D small step models with arbitrary cardinality is pursued in the articles \cite{BaKaYa-16,DuHoWa-16,Ya-17,KaWa-17}. Table~\ref{fig:classification} reproduces this classification. 

\subsection*{Asymptotics of the excursion sequence}
Let us finally mention the article \cite{DeWa-15} by Denisov and Wachtel, which is fundamental to our study. It proves in a great level of generality the following asymptotics for the excursion sequence $o_{A\to B}(n)$, i.e., the number of $n$-step walks in the octant starting (resp.\ ending) at $A\in\mathbb N^3$ (resp.\ $B\in\mathbb N^3$). If $A$ and $B$ are far enough from the boundary, as $n\to\infty$,
\begin{equation}
\label{eq:asymptotics_excursions}
     o_{A\to B}(pn)= \varkappa(A,B)\cdot\rho^{p n}\cdot n^{-\lambda}\cdot(1+o(1)),
\end{equation}
where $\varkappa(A,B)>0$ is some constant, $\rho\in(0,\vert \mathcal S\vert]$ is the exponential growth, $\lambda>0$ is the critical exponent and $p\in\mathbb N$ is the period of the model, i.e., 
\begin{equation}
\label{eq:definition_period}
     p=\gcd\{n\in\mathbb N : o_{A\to B}(n)>0\}.
\end{equation}
The asymptotics \eqref{eq:asymptotics_excursions} is proved in \cite{DeWa-15} in the aperiodic case ($p=1$) and commented in \cite{DuWa-15,BoBMMe-18} for periodic models ($p>1$). For exact hypotheses and a discussion, see Theorem \ref{thm:DW_formula_exponent} in Section \ref{subsec:DeWa-15} and the comments following the statement.

Most of the time we shall assume that
\begin{enumerate}[label={\rm (H)},ref={\rm (H)}]
     \item\label{it:hypothesis_half_space}The step set $\mathcal S$ is not included in any half-space $\{y\in\mathbb R^d : \langle x,y\rangle\geq0\}$ with $x\in\mathbb R^d\setminus\{0\}$, $\langle \cdot,\cdot\rangle$ denoting the classical Euclidean inner product.
\end{enumerate}
The quantities $\rho$ and $\lambda$ in \eqref{eq:asymptotics_excursions} are computed in \cite{DeWa-15}. First, $\rho$ is the global minimum on $\mathbb R_+^d$ of the inventory (or characteristic polynomial)
\begin{equation}
\label{eq:inventory}
     \chi_\mathcal S(x,y,z)=\chi(x,y,z)=\sum_{(i,j,k)\in\mathcal S} x^{i}y^{j}z^{k}
\end{equation}
and is thus well understood and easily computed (it is an algebraic number). On the other hand, $\lambda$ is much more elaborate: applying the results of \cite{DeWa-15} (see in particular Equation~(12) there) readily shows, under the hypothesis \ref{it:hypothesis_half_space},
the following expression for the critical exponent:
\begin{equation}
\label{eq:DW_formula_exponent}
     \lambda=\sqrt{\lambda_1+\frac{1}{4}}+1,
\end{equation}
where $\lambda_1$ is the smallest eigenvalue $\Lambda$ of the Dirichlet problem for the Laplace-Beltrami operator $\Delta_{\mathbb S^{2}}$ on the sphere $\mathbb S^{2}\subset\mathbb R^3$
\begin{equation}
\label{eq:Dirichlet_problem}
     \left\{
\begin{array}{rll}
     -\Delta_{\mathbb S^{2}}m&=\ \Lambda m & \text{in } T,\\
     m&=\ 0& \text{in } \partial T,
     \end{array}
     \right.
\end{equation}
$T=T(\alpha,\beta,\gamma)$ being a spherical triangle (see Figure \ref{fig:isosceles} for an illustration), which can be computed algorithmically (and easily) in terms of the model $\mathcal S$, see Theorem \ref{thm:DW_formula_exponent} for a precise statement. 

Concerning the algebraic nature of the 3D generating function \eqref{eq:generating_function}, a few results are known: in the finite group cases solved in \cite{BoBMKaMe-16}, the generating function is always D-finite. On the other hand, the article \cite{DuHoWa-16} proves that for some degenerate (in the sense of the dimensionality) 3D models, the excursion generating function $O(0,0,0;t)$ is non-D-finite, by looking at the asymptotic behavior of the excursion sequence and showing that $\lambda$ in \eqref{eq:asymptotics_excursions} is irrational, extending the work \cite{BoRaSa-14}. Does there exist a non-degenerate 3D finite group model with a non-D-finite generating function \eqref{eq:generating_function}? The 3D Kreweras model of Figure \ref{fig:various_examples_perpective} could provide such an example. The 3D simple walk in the complement of an octant is also conjectured to admit a non-D-finite generating function, see \cite[Sec.~4]{Mu-19}.

\begin{figure}
\tdplotsetmaincoords{90}{90}

\begin{tikzpicture}[scale=3,tdplot_main_coords]

\tdplotsetthetaplanecoords{90}
\tdplotdrawarc[tdplot_rotated_coords,thick]{(0,0,0)}{0.8}{0}{360}{}{}
\tdplotsetrotatedcoords{60}{70}{0}
\tdplotdrawarc[dashed,tdplot_rotated_coords,name path=blue,color=blue]{(0,0,0)}{0.8}{0}{360}{}{}
\tdplotdrawarc[tdplot_rotated_coords]{(0,0,0)}{0.8}{0}{180}{}{}
\tdplotsetrotatedcoords{120}{110}{0}
\tdplotdrawarc[dashed,tdplot_rotated_coords,name path=green,color=green]{(0,0,0)}{0.8}{0}{360}{}{}
\tdplotdrawarc[tdplot_rotated_coords]{(0,0,0)}{0.8}{0}{180}{}{}
\tdplotsetrotatedcoords{180}{16}{0}
\tdplotdrawarc[dashed,tdplot_rotated_coords,name path=red,color=red]{(0,0,0)}{0.8}{0}{360}{}{}
\tdplotdrawarc[tdplot_rotated_coords]{(0,0,0)}{0.8}{-90}{90}{}{}


\path [name intersections={of={green and blue}, total=\n}]  
\foreach \i in {1,...,\n}{(intersection-\i) circle [radius=0.5pt] coordinate(gb\i){}};

\path [name intersections={of={green and red}, total=\n}]  
\foreach \i in {1,...,\n} {(intersection-\i) circle [radius=0.5pt]coordinate(gr\i){}};

\path[name intersections={of={red and blue}, total=\n}]  
\foreach \i in {1,...,\n}{(intersection-\i) circle [radius=0.5pt]coordinate(rb\i){}};

\shade[top color=gray,bottom color=blue,opacity=0.5]  
(rb4) to[bend left=7] (gr1) to[bend left=16] (gb2) to[bend left=16] cycle;

\draw (gb2) node[below]{$\beta$};
\draw (rb4) node[above left]{$\frac{\pi}{2}$};
\draw (gr1) node[above right]{$\frac{\pi}{2}$};
\end{tikzpicture}\qquad
\tdplotsetmaincoords{90}{90}
\begin{tikzpicture}[scale=3,tdplot_main_coords]

\tdplotsetthetaplanecoords{90}
\tdplotdrawarc[tdplot_rotated_coords,thick]{(0,0,0)}{0.8}{0}{360}{}{}
\tdplotsetrotatedcoords{60}{70}{0}
\tdplotdrawarc[dashed,tdplot_rotated_coords,name path=blue,color=blue]{(0,0,0)}{0.8}{0}{360}{}{}
\tdplotdrawarc[tdplot_rotated_coords]{(0,0,0)}{0.8}{0}{180}{}{}
\tdplotsetrotatedcoords{120}{110}{0}
\tdplotdrawarc[dashed,tdplot_rotated_coords,name path=green,color=green]{(0,0,0)}{0.8}{0}{360}{}{}
\tdplotdrawarc[tdplot_rotated_coords]{(0,0,0)}{0.8}{0}{180}{}{}
\tdplotsetrotatedcoords{220}{16}{0}
\tdplotdrawarc[dashed,tdplot_rotated_coords,name path=red,color=red]{(0,0,0)}{0.8}{0}{360}{}{}
\tdplotdrawarc[tdplot_rotated_coords]{(0,0,0)}{0.8}{-90}{90}{}{}


\path [name intersections={of={green and blue}, total=\n}]  
\foreach \i in {1,...,\n}{(intersection-\i) circle [radius=0.5pt] coordinate(gb\i){}};

\path [name intersections={of={green and red}, total=\n}]  
\foreach \i in {1,...,\n} {(intersection-\i) circle [radius=0.5pt]coordinate(gr\i){}};

\path[name intersections={of={red and blue}, total=\n}]  
\foreach \i in {1,...,\n}{(intersection-\i) circle [radius=0.5pt]coordinate(rb\i){}};

\shade[top color=gray,bottom color=blue,opacity=0.5]  
(rb3) to[bend left=7] (gr1) to[bend left=17] (gb2) to[bend left=15] cycle;

\draw (gb2) node[below]{$\beta$};
\draw (rb3) node[above left]{$\gamma$};
\draw (gr1) node[above right]{$\alpha$};
\end{tikzpicture}
\caption{On the left: a particular spherical triangle with two right angles (these triangles will eventually correspond to Hadamard models). On the right: a generic triangle with angles $\alpha,\beta,\gamma$}
\label{fig:isosceles}
\end{figure}
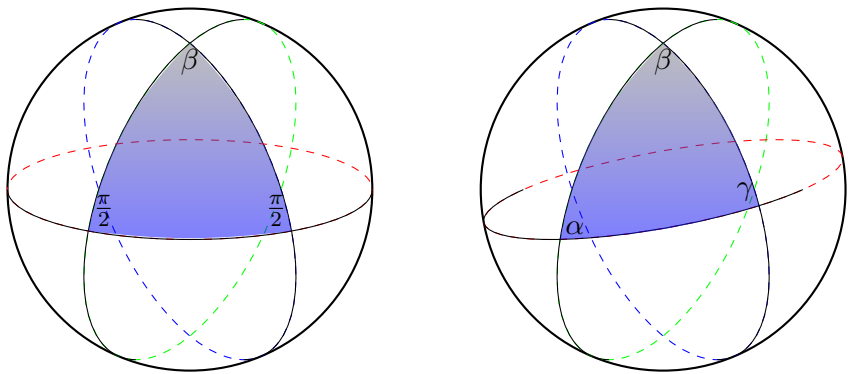

\subsection*{Contributions of the present work}

\begin{figure}
\includegraphics[width=0.35\textwidth]{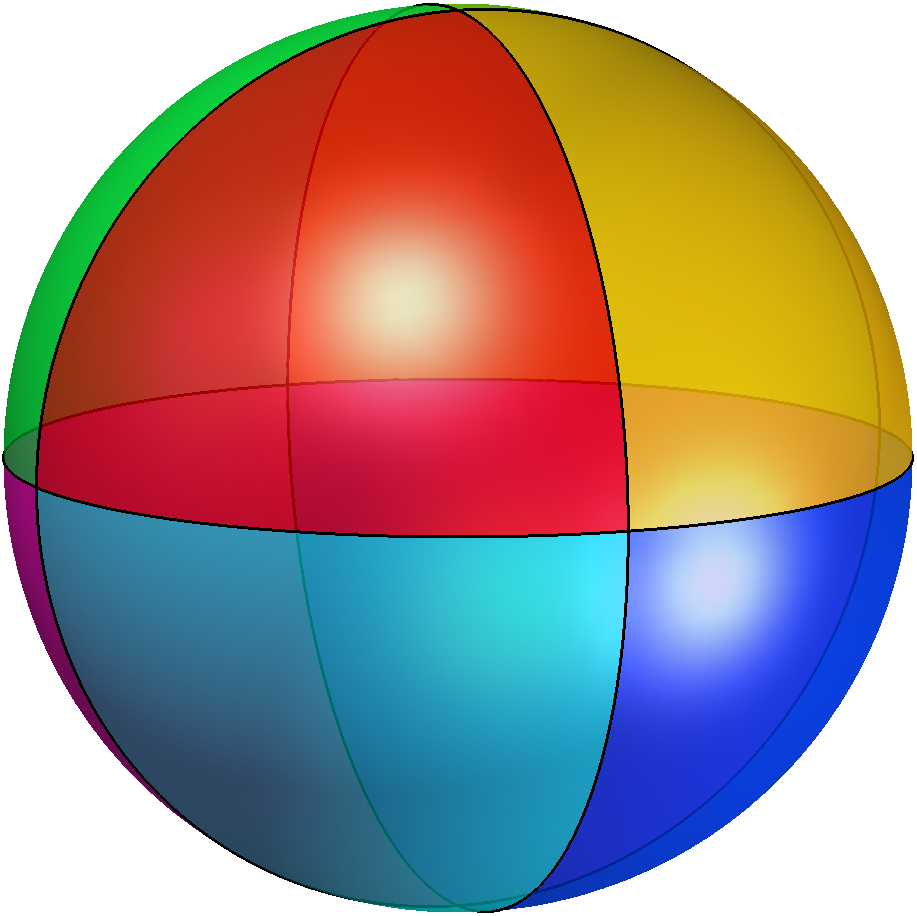}\qquad\qquad
\includegraphics[width=0.35\textwidth]{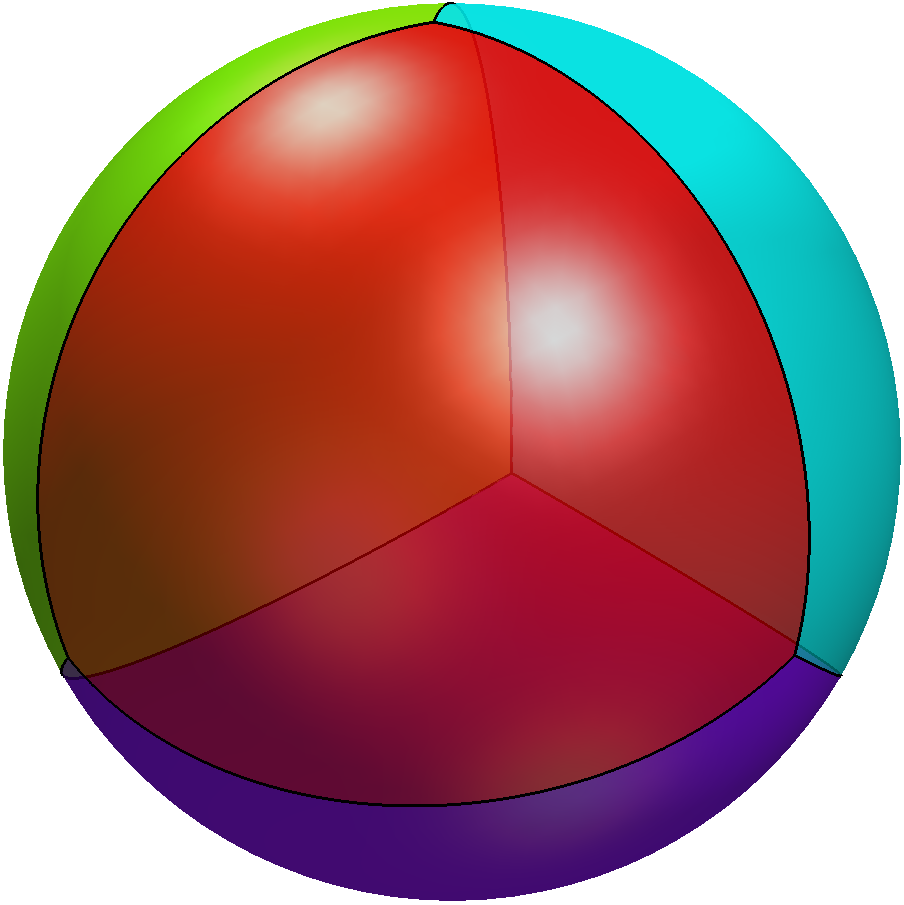}
\caption{Left: tiling of the sphere by equilateral triangles with right angles, corresponding to the simple walk. Right: the tetraedral partition of the sphere, corresponding to Kreweras 3D model. See Figure \ref{fig:some_further_tilings} for further examples of tilings}
\label{fig:tilings}
\end{figure}

In {\bf Section \ref{sec:preliminaries}} we recall all needed definitions and first properties of 3D models. In particular, we associate to each model a spherical triangle, which will capture a lot of combinatorial information. Results in that section come from \cite{DeWa-15,BoBMKaMe-16,BaKaYa-16,KaWa-17}. 
{\bf Section~\ref{sec:expression_angles}} gives the exact value of the angles.

\smallskip

 {\bf Section \ref{sec:Hadamard}:} Our next result deals with Hadamard models (mostly with infinite group, as finite group Hadamard walks are solved in \cite{BoBMKaMe-16}). They have birectangular triangles, as in Figure~\ref{fig:isosceles}, left. Finite (resp.\ infinite) group Hadamard models correspond to angles $\beta$ such that $\frac{\pi}{\beta}\in\mathbb Q$ (resp.\ $\frac{\pi}{\beta}\notin\mathbb Q$). Hadamard models are quite special for combinatorial reasons, as explained in \cite{BoBMKaMe-16}, but also for the Laplacian: to the best of our knowledge, their birectangular triangles are the only triangles (with the exception of the tiling triangles described in Lemma \ref{lem:thm6_Be-83}) for which one can compute the spectrum. We deduce the critical exponent $\lambda$ and show that (most of) infinite group Hadamard models are non-D-finite. This is the first result on the non-D-finiteness of truly 3D models. 

\smallskip

 {\bf Section \ref{sec:group_tiling}:} We classify the models with respect to their triangle and the associated principal eigenvalue, and compare our results with the classification in terms of the group and the Hadamard property obtained in \cite{BoBMKaMe-16,KaWa-17}. Finite group models correspond to triangular tilings of the sphere $\mathbb S^{2}$. The simplest example is the simple walk with steps 
          \begin{equation*}
               \mathcal S=\{(\pm1,0,0),(0,\pm1,0),(0,0,\pm1)\},
          \end{equation*}
          see Figure \ref{fig:various_examples} (leftmost). Its triangle has three right angles, namely $\alpha=\beta=\gamma=\frac{\pi}{2}$ in Figures \ref{fig:isosceles} and \ref{fig:tilings}. A second example is 3D Kreweras model, with step set 
          \begin{equation*}
               \mathcal S=\{(-1,0,0),(0,-1,0),(0,0,-1),(1,1,1)\},
          \end{equation*}
          see Figure \ref{fig:various_examples} (left). The associated triangle is also equilateral, with angles $\frac{2\pi}{3}$, this corresponds to the tetrahedral tiling of the sphere. See Figure \ref{fig:tilings}.
          
      We exhibit some exceptional models, which do not have the Hadamard property but for which, remarkably, one can compute an explicit form for the eigenvalue; this typically leads to non-D-finiteness results.    
      
      Although we won't consider these issues here, let us mention that we can also see the dimensionality on the triangle. In the case of 2D models, the triangles degenerate into a spherical digon, see Section \ref{subsec:WQP} (in particular Figure \ref{fig:digon}).

 \smallskip

{\bf Section \ref{sec:numerical_approximation}:} Our last result is about generic infinite group models. Even if no closed-form formula exists for $\lambda_1$, we may consider $\lambda_1$ as a {\it special function} of the triangle $T$ (or equivalently of its angles $\alpha,\beta,\gamma$, as in spherical geometry a triangle is completely determined by its angles), and with numerical analysis methods, obtain approximations of this function when evaluated at particular values. The techniques developed in Section \ref{sec:numerical_approximation} are completely different from the rest of the paper. Notice that for some cases, approximate values of the critical exponents have been found by Bostan and Kauers \cite{BoKa-09}, Bacher, Kauers and Yatchak \cite{BaKaYa-16}, Bogosel \cite{Bo-16}, Guttmann \cite{Gu-17}, Dahne and Salvy \cite{DaSa-19}. See Section \ref{subsec:numerical_approximation_lit} for more details.

\smallskip

{\bf Section \ref{sec:Miscellaneous}} proposes various extensions and remarks. Finally, the brief {\bf Appendix \ref{sec:spherical_geometry}} gathers some elementary facts on spherical geometry.

\subsection*{Brownian motion in orthants} 
To conclude this introduction, let us emphasize that all results that we obtain for discrete random walks admit continuous analogues and can be used to estimate exit times from cones of Brownian motion, see Section \ref{subsec:Brownian_motion}. In the literature, one can find applications of these estimates to the Brownian pursuit \cite{RaTr-09a,RaTr-09b}.

\subsection*{Acknowledgments} 
This work has benefited from discussions with many colleagues. We in particular warmly thank M.\ Kauers for interesting discussions and for sharing with us a complete and very precise classification (and many other data) about 3D walks. Many thanks also to V.\ Beck, A.\ Bostan, M.\ Bousquet-M\'elou, S.\ Cantat, M.\ Dauge, T.\ Guttmann, L.\ Hillairet, A.\ Lejay, S.\ Mustapha and B.\ Salvy. We also thank the two anonymous referees for their numerous remarks, which led us to improve the presentation of our article. This project has received funding from the European Research Council (ERC) under the European Union's Horizon 2020 research and innovation programme under the Grant Agreement No 759702.

\section{Preliminaries} 
\label{sec:preliminaries}

In this section we introduce key concepts to study 3D walks. We are largely inspired by the paper \cite{BoBMKaMe-16}, from which we borrowed Section \ref{subsec:dimension_model} (dimension of a model), Section~\ref{subsec:group} (group of a model) and Section~\ref{subsec:Hadamard_structure} (Hadamard structure). The thorough classification presented in Section~\ref{subsec:classification} is done in the papers \cite{BaKaYa-16,KaWa-17} and the fundamental asymptotic result of Section~\ref{subsec:DeWa-15} can be found  in \cite{DeWa-15}. We follow the notations of \cite{BoBMKaMe-16}.

\subsection{Dimension of a model}
\label{subsec:dimension_model}

Let $\mathcal S$ be a step set. A walk of length $n$  taking its steps in $\mathcal{S}$ can be viewed as a word $w=w_1 w_2 \ldots w_n$ made up of letters of $\mathcal S$. For $s\in
\mathcal S$, let $a_s$ be the multiplicity (i.e., the number of occurrences) of $s$ in $w$. Then $w$ ends in the positive octant if and only if the following three linear inequalities hold:
\begin{equation}
\label{eq:3ineq}
     \sum_{s\in \mathcal S} a_s s_x \geq 0, \qquad \sum_{s\in \mathcal S} a_s s_y \geq 0,
\qquad \sum_{s\in \mathcal S} a_s s_z \geq 0,  
\end{equation}
where $s=(s_x,s_y,s_z)$. Of course, the walk $w$ remains in the octant if the multiplicities observed in each of its prefixes satisfy these inequalities.

\begin{Definition}[\cite{BoBMKaMe-16}]
\label{def:dim}
Let $d\in\{0,1,2,3\}$. A model $\mathcal S$ is said to have dimension at most
$d$ if there exist $d$ inequalities in~\eqref{eq:3ineq} such that any
$\vert\mathcal S\vert$-tuple $(a_s)_{s\in \mathcal S}$ of non-negative integers satisfying
these $d$ inequalities satisfies in fact the three ones.
We define accordingly models of dimension (exactly) $d$.
\end{Definition}
See \cite[Fig.~1]{BoBMKaMe-16} for an illustration of Definition~\ref{def:dim}. In what follows we will be principally considering models of dimension $3$, and in fact only a subclass of them: most of the time we will assume the hypothesis \ref{it:hypothesis_half_space}.


\subsection{Group of the model}
\label{subsec:group}
This group was first introduced in the context of 2D walks \cite{FaIaMa-99,BMMi-10} and turns out to be very useful. Let $\chi$ be the inventory \eqref{eq:inventory}. Introduce the notation
\begin{align*}
     \chi(x,y,z) &=\overline{x}A_{-}(y,z)+A_{0}(y,z)+xA_{+}(y,z)\\
                     &=\overline{y}B_{-}(x,z)+B_{0}(x,z)+yB_{+}(x,z)\\
                     &=\overline{z}C_{-}(x,y)+C_{0}(x,y)+zC_{+}(x,y),
\end{align*}
where $\overline{x}=\frac{1}{x}$, $\overline{y}=\frac{1}{y}$ and $\overline{z}=\frac{1}{z}$. If $\mathcal S$ is $3$-dimensional then it has a positive step in each direction and $A_{+}$, $B_{+}$ and $C_{+}$ are all non-zero. The group of $\mathcal S$ is the group $G=\langle\phi,\psi,\tau\rangle$ of birational transformations of the variables $[x,y,z]$ generated by the following three involutions:
\begin{equation}
\label{eq:expression_generators}
     \left\{\begin{array}{rl}
     \phi([x,y,z]) & =\, \left[\overline{x}\frac{A_{-}(y,z)}{A_{+}(y,z)},y,z\right],\smallskip\\
     \psi([x,y,z]) & =\, \left[x,\overline{y}\frac{B_{-}(x,z)}{B_{+}(x,z)},z\right],\smallskip\\
     \tau([x,y,z]) & =\, \left[x,y,\overline{z}\frac{C_{-}(x,y)}{C_{+}(x,y)}\right].
     \end{array}\right.
\end{equation}
See \cite[Sec.~2.4]{BoBMKaMe-16} for more details on the group.
The classification of the models according to the (in)finiteness of the group is known, see Table \ref{fig:classification}. Let us also reproduce \cite[Tab.~1]{KaWa-17}:
\def\a{\mathtt{a}}\def\b{\mathtt{b}}\def\c{\mathtt{c}}\def\<#1>{\langle#1\rangle}
\begin{table}[ht!]
  \begin{tabular}{lr|lr}
    Group & \llap{Number of models} & Group &\llap{Number of models} \\\hline
    $G_1=\<\a,\b,\c\mid\a^2,\b^2,\c^2>$ &\kern-1em 10,759,449 & $G_7=\<\a,\b,\c\mid\a^2,\b^2,\c^2,(\a\b)^4>$ & 82 \\
    $G_2=\<\a,\b,\c\mid\a^2,\b^2,\c^2,(\a\b)^2>$ & 84,241 & $G_8=\<\a,\b,\c\mid\a^2,\b^2,\c^2,(\a\b)^3, (\b\c)^3>$ & 30 \\
    $G_3=\<\a,\b,\c\mid\a^2,\b^2,\c^2,(\a\c)^2, (\a\b)^2>$ & 58,642 & $G_9=\<\a,\b,\c\mid\a^2,\b^2,\c^2,\a\c\b\a\c\b\c\a\b\c>$ & 20 \\
    $G_4=\<\a,\b,\c\mid\a^2,\b^2,\c^2,(\a\c)^2, (\a\b)^3>$ & 1,483 & $G_{10}=\<\a,\b,\c\mid\a^2,\b^2,\c^2, (\a\b)^3, (\c\b\c\a)^2>$ & 8 \\
    $G_5=\<\a,\b,\c\mid\a^2,\b^2,\c^2,(\a\b)^3>$ & 1,426 & $G_{11}=\<\a,\b,\c\mid\a^2,\b^2,\c^2,(\c\a)^3,(\a\b)^4, (\b\a\b\c)^2>$ & 8\\
    $G_6=\<\a,\b,\c\mid\a^2,\b^2,\c^2,(\a\c)^2, (\a\b)^4>$ & 440 & $G_{12}=\<\a,\b,\c\mid\a^2,\b^2,\c^2,(\a\b)^4, (\a\c)^4>$ & 4
  \end{tabular}
  \bigskip
  \caption{Various infinite groups associated to 3D models. Notice that the presentations of the groups are not certified: it is not excluded \cite{KaWa-17} that further relations exist, but then involving more than $400$ generators $\a,\b,\c$. With the exception of $G_{9}$, $G_{10}$ and $G_{11}$, all groups are Coxeter groups. Most of the time, but not systematically, one can take $\a=\phi$, $\b=\psi$ and $\c=\tau$} 
  \label{tab:various_groups}
\end{table}

\subsection{Hadamard structure}
\label{subsec:Hadamard_structure}
Hadamard models are introduced in \cite{BoBMKaMe-16} (see in particular Section 5 there). These are $3$-dimensional models which can be reduced to the study of a pair of models, one in $\mathbb Z$ and one in $\mathbb Z^2$, using a Hadamard product of generating functions. 

There are two types of Hadamard models: the $(1,2)$-type and the $(2,1)$-type. More generally, in arbitrary dimension $d$ there is the notion of $(D,\delta)$-Hadamard model, with $D+\delta=d$, see \cite[Sec.~5.2]{BoBMKaMe-16}. Back to the dimension $3$, the $(1,2)$-type corresponds to models for which the inventory \eqref{eq:inventory} can be written under the form
\begin{equation}
\label{eq:inventory_(1,2)-type_Hadamard}
     \chi(x,y,z)=U(x)+V(x)T(y,z).
\end{equation}
The $(2,1)$-type corresponds to
\begin{equation}
\label{eq:inventory_(2,1)-type_Hadamard}
     \chi(x,y,z)=U(x,y)+V(x,y)T(z).
\end{equation}
The number of Hadamard models (with the additional information on the type) can be found in Table \ref{fig:classification}.

For each type, an example is presented in Figure \ref{fig:various_examples}: for the $(2,1)$-type we have taken $U(x,y)=x+\overline{x}+y+\overline{y}$ (the 2D simple walk, see Figure \ref{fig:some_2D_models}), $V(x,y)=x+x\overline{y}+\overline{x}\overline{y}+\overline{x}y+y$ (a scarecrow model, see again Figure \ref{fig:some_2D_models}) and $T(z)=z+\overline{z}$.
For the $(1,2)$-type we have $\chi(x,y,z)=U(z)+V(z)T(x,y)$ (permutation of the variables in the definition \eqref{eq:inventory_(1,2)-type_Hadamard}), with $U(z)=z+\overline{z}$, $V(z)=z+1+\overline{z}$ and $T(x,y)$ the generating function of the same scarecrow model as above.

\begin{figure}[htb]
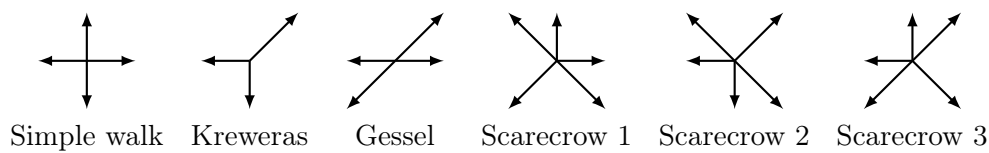

  \centering
  \begin{tabular}{cccccc}
    $\ \diagr{E,W,N,S}\ $ & $\ \diagr{NE,S,W}\ $  & $\ \diagr{NE,W,E,SW}\ $
  & $\ \diagr{SW,N,E,NW,SE}\ $ & $\ \diagr{NE,SE,S,W,NW}\ $& $\ \diagr{N,NE,SE,SW,W}\ $
\\
Simple walk&Kreweras&Gessel&Scarecrow 1&Scarecrow 2&Scarecrow 3
  \end{tabular}
   \caption{Some 2D models. The three scarecrows are named after \cite[Fig.~1]{BoRaSa-14}}
  \label{fig:some_2D_models}
\end{figure}

Hadamard models extend Cartesian products of walks: Cartesian products (or equivalently independent random walks in the probabilistic framework) correspond to taking $U(x)=0$ in \eqref{eq:inventory_(1,2)-type_Hadamard} or $U(x,y)=0$ in \eqref{eq:inventory_(2,1)-type_Hadamard}. Notice that Hadamard models in dimension $2$ are always D-finite \cite{BoBMMe-18}, even with large steps.

\subsection{Classification of models}
\label{subsec:classification}Proposition 2.5 of \cite{BoBMKaMe-16} gives the number of models having dimension $2$ or $3$, no unused step (that is, a step that is never used in a walk confined to the octant), and counted up to permutations of the coordinates, ending up with the number 11,074,225 in Table \ref{fig:classification}.

\begin{table}
\Tree[.Models\\(\textcolor{blue}{11,074,225}) [.$\vert G\vert<\infty$\\(\textcolor{blue}{165,962}) [.3D~Hadamard\\(\textcolor{magenta}{2,187}) [.both\\(\textcolor{magenta}{305}) ] 
               [.(1,2)\\(\textcolor{magenta}{84}) ]
               [.(2,1)\\(\textcolor{magenta}{1,798}) ]]
               [.non-3D~Ha. ]]
          [.$\vert G\vert=\infty$\\(\textcolor{blue}{10,908,263}) [.3D~Hadamard\\(\textcolor{magenta}{58,642}) [.both\\(\textcolor{magenta}{280}) ] 
               [.(1,2)\\(\textcolor{magenta}{672}) ]
               [.(2,1)\\(\textcolor{magenta}{57,690}) ]]
               [.non-3D~Ha. ]]]
\caption{Classification of 3D walks (of dimension $2$ and $3$) according to the finiteness of the group and the Hadamard property \cite{KaWa-17}. The numbers of (non-)Hadamard models refer exclusively to dimension $3$ models. Hence among the non-3D Hadamard models one can find models of dimensionality $2$ having a (degenerate) Hadamard decomposition. A model labeled ``both'' is simultaneously $(1,2)$-type and $(2,1)$-type Hadamard. The total number of models is computed in \cite{BoBMKaMe-16}, the number of (in)finite groups in \cite{BoBMKaMe-16,DuHoWa-16,KaWa-17} and the refined statistics on 3D Hadamard models in \cite{Ka-17}}
\label{fig:classification}
\end{table} 

\subsection{Formula for the exponent of the excursions}
\label{subsec:DeWa-15}

We now explain that the exponent $\lambda$ in \eqref{eq:asymptotics_excursions} is directly related to the smallest eigenvalue of a certain Dirichlet problem on a spherical triangle. Let us start with a simple definition:
\begin{Definition}[\cite{Be-87}]
\label{def:spherical_triangle}
A spherical triangle on $\mathbb S^2$ is a triple $(x,y,z)$ of points of $\mathbb S^2$ that are linearly independent as vectors in $\mathbb R^3$. We denote it by $\langle x,y,z\rangle$.
\end{Definition}
See examples in Figures \ref{fig:isosceles} and \ref{fig:tilings}. The points $x,y,z$ are the vertices of $\langle x,y,z\rangle$. By the sides of $\langle x,y,z\rangle$ we mean the arcs of great circle determined by $(x,y)$, $(y,z)$ and $(z,x)$.

The following result gives a formula for the critical exponent; several explanatory remarks may be found below the statement.
\begin{Theorem}[\cite{DeWa-15}]
\label{thm:DW_formula_exponent}
Let $\mathcal S$ be a step set satisfying \ref{it:hypothesis_half_space} and irreducible. Let $\chi$ be its inventory~\eqref{eq:inventory}. The system of equations
\begin{equation}
\label{eq:chi_critical_point}
     \frac{\partial \chi}{\partial x}=\frac{\partial \chi}{\partial y}=\frac{\partial \chi}{\partial z}=0
\end{equation}
admits a unique solution in $(0,\infty)^3$, denoted by $(x_0,y_0,z_0)$. Define
\begin{equation}
\label{eq:expression_covariance_matrix}
     a=\frac{\frac{\partial^2 \chi}{\partial x\partial y}}{\sqrt{\frac{\partial^2 \chi}{\partial x^2}\cdot \frac{\partial^2 \chi}{\partial y^2}}}(x_0,y_0,z_0),\quad
     b=\frac{\frac{\partial^2 \chi}{\partial x\partial z}}{\sqrt{\frac{\partial^2 \chi}{\partial x^2}\cdot \frac{\partial^2 \chi}{\partial z^2}}}(x_0,y_0,z_0),\quad
     c=\frac{\frac{\partial^2 \chi}{\partial y\partial z}}{\sqrt{\frac{\partial^2 \chi}{\partial y^2}\cdot \frac{\partial^2 \chi}{\partial z^2}}}(x_0,y_0,z_0)
\end{equation}
and introduce the covariance matrix     
\begin{equation}
\label{eq:covariance_matrix}
     \cov=\left(\begin{array}{lll}
     1 & a & b\\
     a & 1 & c\\
     b & c & 1
     \end{array}\right).
\end{equation}
Let $S$ denote a square root of the covariance matrix, namely 
\begin{equation}
\label{eq:square_root_equation}
     \cov=S S^\intercal. 
\end{equation}
Consider the spherical triangle $T=(S^{-1}\mathbb R_+^3) \cap \mathbb S^2$. Let $\lambda_1$ be the smallest eigenvalue of the Dirichlet problem \eqref{eq:Dirichlet_problem}. Then for $A$ and $B$ far enough from the boundary, the asymptotics \eqref{eq:asymptotics_excursions} of the number of excursions going from $A$ to $B$ holds, where 
\begin{equation}
\label{eq:formula_rho} 
     \rho =\min_{(0,\infty)^3} \chi
\end{equation}
and the critical exponent $\lambda$ in \eqref{eq:asymptotics_excursions} is given by \eqref{eq:DW_formula_exponent}.
\end{Theorem}

The proof of Theorem \ref{thm:DW_formula_exponent} is sketched in Section \ref{subsec:proof_Thm_DW}. Here we just comment on its hypotheses.
\begin{itemize}
     \item First, under \ref{it:hypothesis_half_space} the characteristic polynomial is strictly convex and coercive on $(0,\infty)^3$ and hence there is a unique global minimizing point $(x_0,y_0,z_0)$, which satisfies  \eqref{eq:chi_critical_point}.\smallskip
     \item The covariance matrix \eqref{eq:covariance_matrix} is positive definite, this is a direct consequence of \ref{it:hypothesis_half_space} (the rank of the covariance matrix describes the dimension of the subspace in which the random walk evolves).\smallskip 
     \item The matrix $S^{-1}$ has full rank and hence $T=(S^{-1}\mathbb R_+^3) \cap \mathbb S^2$ is a spherical triangle (see our Definition \ref{def:spherical_triangle}), bounded by the three great-circle arcs $(S^{-1}e_i) \cap \mathbb S^2$, with $e_i$ denoting the $i$th vector of the canonical basis.\smallskip
     \item The choice of the square root in \eqref{eq:square_root_equation} is not relevant: if $\cov=S_1 S_1^\intercal=S_2 S_2^\intercal$ then obviously $S_1=MS_2$, where $M$ is an orthogonal matrix, and the two associated spherical triangles are isometric (and in particular they have the same angles).\smallskip
     \item The boundary of the spherical triangle is piecewise infinitely differentiable. Under this assumption, the spectrum of the Laplacian for the Dirichlet problem \eqref{eq:Dirichlet_problem} is discrete (see \cite[p.~169]{Ch-84}), of the form $0<\lambda_1<\lambda_2\leq\lambda_3\leq \cdots$.\smallskip
     \item The irreducibility hypothesis means that for any two points in the space $\mathbb Z^3$, there exists a path connecting these points.\smallskip
     \item The asymptotics \eqref{eq:asymptotics_excursions} is proved in \cite{DeWa-15} under the assumption that the walk is strongly aperiodic (see the lattice assumption in \cite[p.~999]{DeWa-15}), i.e., irreducible and aperiodic in the sense of the Markov chains. The aperiodicity is defined by $p=1$ in \eqref{eq:definition_period}. Two remarks should be made:
     \begin{itemize}
          \item As explained in \cite{BoBMMe-18}, an extra-assumption (namely, a reachability condition) has to be made. There is indeed in \cite{BoBMMe-18} the example of a 2D walk which is strongly aperiodic but such that no excursion to the origin is possible, due to the (ad hoc) particular configuration of the steps. We could easily construct a 3D analogue such that $o(0,0,0;n)=0$ for all $n$.
          \item The second point is about periodic models ($p>1$ in \eqref{eq:definition_period}), which stricto sensu are not covered by \cite{DeWa-15}. It is briefly mentioned in \cite{DuWa-15} that the main asymptotics \eqref{eq:asymptotics_excursions} still holds true. A detailed discussion of the periodic case may be found in \cite{BoBMMe-18}.
     \end{itemize}
     As our point is not to state Theorem \ref{thm:DW_formula_exponent} at the greatest level of generality, we have stated it under rather strong hypotheses, namely that $A$ and $B$ are far enough from the boundary (this is sufficient for the reachability condition). 
\end{itemize}

\subsection{Computing the principal eigenvalue of a spherical triangle}

There are very few spherical triangles (and more generally, few domains on the sphere, see Section~\ref{subsec:other_cones} and Appendix \ref{sec:spherical_geometry}) for which we can explicitly compute the first eigenvalue $\lambda_1$ of the Dirichlet problem \eqref{eq:Dirichlet_problem}. As a matter of comparison, let us recall that (to our knowledge, see also \cite{BeBe-80}) there does not exist in general a closed-form expression for the analogous problem for flat triangles.

Back to spherical triangles, there essentially exists a unique case for which an explicit expression for $\lambda_1$ is known: the case of two angles $\frac{\pi}{2}$ as in Figure \ref{fig:isosceles} (these triangles are called birectangular). Then according to \cite[Eq.~(36)]{WaKe-77} (or \cite[Sec.~IV]{Wa-74}) the smallest eigenvalue is
\begin{equation}
\label{eq:eigenvalue_WaKe-77}
     \lambda_1 = \left(\frac{\pi}{\beta}+1\right)\left(\frac{\pi}{\beta}+2\right).
\end{equation}
Let us give three relevant cases in the range of application of formula \eqref{eq:eigenvalue_WaKe-77}:
\begin{itemize}
     \item The 3D simple random walk (Figure \ref{fig:various_examples}): then $\beta=\frac{\pi}{2}$ and $\lambda_1=12$, which with \eqref{eq:DW_formula_exponent} corresponds to $\lambda=\frac{9}{2}$ (in accordance with the intuition $3\times \frac{3}{2}$, i.e., three independent positive 1D excursions).
     \item More generally, finite group Hadamard models. They correspond to $\beta\in\pi\mathbb Q$. They represent tiling groups of the sphere. See Section \ref{sec:group_tiling} for more details.
     \item Last but not least, all Hadamard models, even with infinite group (typically $\beta\notin\pi\mathbb Q$); see Section \ref{sec:Hadamard}.
\end{itemize}     

\section{The covariance matrix}
\label{sec:expression_angles}

The angles of the spherical triangle appearing in the main Theorem~\ref{thm:DW_formula_exponent} are totally explicit in terms of the correlation coefficients $a$, $b$ and $c$ defined in \eqref{eq:expression_covariance_matrix}.

\begin{Lemma}
\label{lem:exact_value_angles_triangle}
Let $\alpha,\beta,\gamma$ be the angles of the spherical triangle $T$ defined in Theorem \ref{thm:DW_formula_exponent}, and $a,b,c$ as in \eqref{eq:expression_covariance_matrix}. One has
\begin{equation}
\label{eq:exact_value_angles_triangle}
     \alpha =  \arccos(-a),\quad
     \beta =  \arccos(-b),\quad
     \gamma =  \arccos(-c).
\end{equation}
\end{Lemma}

Three remarks should be made.
\begin{itemize}
     \item It is easily seen that the correlation coefficients $a$, $b$ and $c$ of Lemma \ref{lem:exact_value_angles_triangle} are algebraic numbers. We can use the exact same algorithmic computations as in \cite[Sec.~2.4.1]{BoRaSa-14} to deduce their minimal polynomials.
     \item The formulas given in Lemma \ref{lem:exact_value_angles_triangle} are the most natural generalization of the 2D situation, where by \cite{BoRaSa-14} the spherical triangle is replaced by a wedge of opening angle $\arccos(-c)$, see  Figure \ref{fig:wedge}.
     \item If two of the three correlation coefficients $a$, $b$ and $c$ are equal to $0$, then the spherical triangle is birectangular.  
\end{itemize}     

\begin{figure}[ht!]
\begin{center}
\vspace{35mm}
\hspace{-30mm}\begin{tikzpicture}
    \thicklines
    \put(0,0){\textcolor{black}{\vector(1,0){100}}}
    \put(0,0){\textcolor{black}{\vector(0,1){100}}}
    \thicklines
    \put(13,13){{$\frac{\pi}{2}$}}
    \textcolor{blue}{\qbezier(15,0)(13,13)(0,15)}
    \end{tikzpicture}\hspace{50mm}
\begin{tikzpicture}
    \thicklines
    \put(0,0){\textcolor{black}{\vector(1,0){100}}}
    \put(0,0){\textcolor{black}{\vector(1,2){50}}}
    \thicklines
    \put(18,13){{$\arccos(-c)$}}
    \textcolor{blue}{\qbezier(15,0)(15,8.6)(7.5,15)}
    \end{tikzpicture}
    \end{center}
    \caption{After decorrelation of a 2D random walk, the quarter plane (left) becomes a wedge of opening $\arccos(-c)$ (right), where $c$ is the correlation coefficient of the driftless model}
    \label{fig:wedge}
    \end{figure}
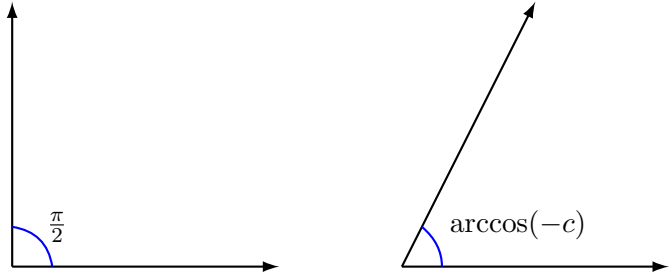

\begin{proof}[Proof of Lemma \ref{lem:exact_value_angles_triangle}]
Let $\cov$ be the matrix defined in \eqref{eq:covariance_matrix}. We easily obtain the Cholesky decomposition $\cov=L L^\intercal$, with
\begin{equation}
\label{eq:Cholesky_L}
     L = \left(\begin{array}{ccc}
     1& 0& 0\\
     a& \sqrt{1-a^2}& 0\\
     b& \frac{c-ab}{\sqrt{1-a^2}} & \frac{\sqrt{1-a^2-b^2-c^2+2abc}}{\sqrt{1-a^2}}
     \end{array}\right).
\end{equation}
 One deduces that
\begin{equation}
\label{eq:Cholesky_L^-1}
     L^{-1} = \left(\begin{array}{ccc}
     1&0&0\\
     \frac{-a}{\sqrt{1-a^2}} & \frac{1}{\sqrt{1-a^2}}  & 0\\
     \frac{ac-b}{\sqrt{1-a^2}\sqrt{1-a^2-b^2-c^2+2abc}} & \frac{ab-c}{\sqrt{1-a^2}\sqrt{1-a^2-b^2-c^2+2abc}}& \frac{\sqrt{1-a^2}}{\sqrt{1-a^2-b^2-c^2+2abc}}
     \end{array}\right).
\end{equation}
Denoting by $(e_1,e_2,e_3)$ the canonical basis of $\mathbb R^3$, the three points defining the triangle are
\begin{equation*}
     x=\frac{L^{-1}e_1}{\Vert L^{-1}e_1 \Vert},\quad y=\frac{L^{-1}e_2}{\Vert L^{-1}e_2 \Vert},\quad z=\frac{L^{-1}e_3}{\Vert L^{-1}e_3 \Vert},
\end{equation*}
see the third comment following Theorem \ref{thm:DW_formula_exponent} or Section \ref{subsec:proof_Thm_DW}. Setting 
\begin{equation*}
x_y=\frac{y-\langle x, y\rangle x}{\Vert y-\langle x, y\rangle x \Vert }
\end{equation*}
and $x_z,y_x,y_z,z_x,z_y$ similarly, we have by \cite[18.6.6]{Be-87} (giving the formulas for the angles of the triangle $\langle x,y,z\rangle$) 
\begin{equation*}
     \alpha =  \arccos\langle x_y, x_z \rangle,\quad
     \beta =  \arccos\langle y_z, y_x \rangle,\quad
     \gamma =  \arccos\langle z_x, z_y \rangle.
\end{equation*}
To conclude the proof it is enough to do the above computations in terms of $a$, $b$ and $c$.
\end{proof}
Notice that the above expression of $L$ in terms of $a$, $b$ and $c$ is particularly simple, and thus the proof of Lemma \ref{lem:exact_value_angles_triangle} is easily obtained. On the other hand, solving the equation $\cov=S S^\intercal$ with the constraint of taking a symmetric square root $S$ happens to be much more complicated and less intrinsic.

Several further aspects of the covariance matrix are provided in Section \ref{subsec:further_prop_cov}.

\section{Analysis of Hadamard models}
\label{sec:Hadamard}

This section is at the heart of the present paper. We consider Hadamard models in the sense of Section \ref{subsec:Hadamard_structure}. Let us briefly recall that these models are characterized by the existence of a decomposition of their inventory \eqref{eq:inventory} as follows:
\begin{equation*}
     \chi(x,y,z)=U(x)+V(x)T(y,z)\qquad \text{or}\qquad
     \chi(x,y,z)=U(x,y)+V(x,y)T(z).
\end{equation*} 
As will be shown in Lemmas \ref{lem:covariance_matrix_(1,2)-type_Hadamard} and \ref{lem:covariance_matrix_(2,1)-type_Hadamard}, such models admit a quite simple covariance matrix
\begin{equation*}
     \cov=\left(\begin{array}{lll}
     1 & 0 & 0\\
     0 & 1 & c\\
     0 & c & 1
     \end{array}\right),
\end{equation*}
allowing us to perform explicitly many computations. (Notice, however, that the above form for the covariance matrix does not characterize Hadamard models, we construct counterexamples in Section \ref{subsec:counterexamples}. These examples lead to the notion of exceptional models.) 

In particular, spherical triangles associated to Hadamard models are birectangular, i.e., two (or three) angles are equal to $\frac{\pi}{2}$, see Figure \ref{fig:isosceles}. These triangles are remarkable because they are the only ones (with the exception of a few sporadic cases) for which a closed-form expression for the principal eigenvalue is known. Finally, the exponent \eqref{eq:DW_formula_exponent} of the excursion sequence is computed in Propositions \ref{prop:exponent_(1,2)-type_Hadamard} and \ref{prop:exponent_(2,1)-type_Hadamard}. Using similar techniques as in \cite{BoRaSa-14}, one can rather easily study the rationality of this exponent.

In the $(1,2)$-type (Section \ref{subsec:(1,2)-Hadamard}), the 2D model associated to $T(y,z)$ dictates the exponent, see Proposition \ref{prop:exponent_(1,2)-type_Hadamard}. In particular, we will see in Corollary \ref{cor:non-D-finite_(1,2)-type_Hadamard} that if the 2D model has an irrational exponent, then the 3D model is necessarily non-D-finite. To our knowledge, this is the first proof ever of the non-D-finiteness of truly 3D models, making the Hadamard case remarkable. On the other hand, $(2,1)$-type Hadamard models (Section~\ref{subsec:(2,1)-Hadamard}) are more subtle. Their exponents can be computed from exponents of mixtures of two 2D models.

Although we won't do such considerations here, let us emphasize that most of the results in this section hold for weighted walks with arbitrarily big steps: the only crucial point is the existence of a Hadamard decomposition.

\subsection{(1,2)-Hadamard models}
\label{subsec:(1,2)-Hadamard}

\begin{Lemma}
\label{lem:covariance_matrix_(1,2)-type_Hadamard}
For any $(1,2)$-type Hadamard model, the matrix $\cov$ in \eqref{eq:covariance_matrix} takes the form
\begin{equation}
\label{eq:covariance_matrix_(1,2)-type_Hadamard}
     \cov=\left(\begin{array}{lll}
     1 & 0 & 0\\
     0 & 1 & c\\
     0 & c & 1
     \end{array}\right), \quad \text{with}\quad c=\frac{\frac{\partial^2 T}{\partial y\partial z}}{\sqrt{\frac{\partial^2 T}{\partial y^2}\cdot \frac{\partial^2 T}{\partial z^2}}}(y_0,z_0),
\end{equation}
where $y_0,z_0$ are defined in \eqref{eq:chi_critical_point}. (Notice in particular that $c$ does not depend on the horizontal components $U$ and $V$ in the Hadamard decomposition \eqref{eq:inventory_(1,2)-type_Hadamard}.)
\end{Lemma}

\begin{proof}
The proof is elementary. Using the decomposition \eqref{eq:inventory_(1,2)-type_Hadamard} in the last two equations of the system \eqref{eq:chi_critical_point} gives
\begin{equation}
\label{eq:system_eval}
     V(x)\frac{\partial T}{\partial y}(y,z)=V(x)\frac{\partial T}{\partial z}(y,z)=0.
\end{equation}
As $V(x)$ cannot be equal to $0$, we obtain the autonomous system $\frac{\partial T}{\partial y}=\frac{\partial T}{\partial z}=0$. Let $(y_0,z_0)$ be its unique solution. Moreover, the first equation in \eqref{eq:chi_critical_point} leads to
\begin{equation*}
     U'(x)+V'(x)T(y_0,z_0)=0
\end{equation*}
which (as $T(y_0,z_0)>0$) has a unique solution $x_0$. Using once again \eqref{eq:inventory_(1,2)-type_Hadamard} as well as \eqref{eq:system_eval}, we deduce that
\begin{equation*}
     V'(x_0)\frac{\partial T}{\partial y}(y_0,z_0)=0,
\end{equation*}
whence $a=0$ and similarly $b=0$. The formula \eqref{eq:covariance_matrix_(1,2)-type_Hadamard} for $c$ is a direct consequence of \eqref{eq:expression_covariance_matrix} and \eqref{eq:inventory_(1,2)-type_Hadamard}.
\end{proof}

Our aim now is to compute the spherical angles in the Hadamard case. We use Lemma~\ref{lem:exact_value_angles_triangle} to deduce Proposition \ref{prop:exponent_(1,2)-type_Hadamard} below.

\begin{Proposition}
\label{prop:exponent_(1,2)-type_Hadamard}
The spherical triangles associated to $(1,2)$-type Hadamard models have angles $\frac{\pi}{2},\frac{\pi}{2},\arccos(-c)$ (as in Figure \ref{fig:isosceles}, left), with $c$ defined in \eqref{eq:covariance_matrix_(1,2)-type_Hadamard}. The smallest eigenvalue $\lambda_1$ of the Dirichlet problem and the exponent $\lambda$ are respectively given by 
\begin{equation*}
     \lambda_1=\left(\frac{\pi}{\arccos(-c)}+1\right)\left(\frac{\pi}{\arccos(-c)}+2\right),\qquad \lambda=\frac{\pi}{\arccos(-c)}+\frac{5}{2}.
\end{equation*}
 \end{Proposition}

In order to completely characterize the excursion exponent we now compute $c$ and $\lambda$. This happens to be done in \cite{BoRaSa-14}: for the 2D unweighted models under consideration, $c$ is always algebraic (possibly rational), and minimal polynomials in the infinite group case are provided in \cite[Tab.~2]{BoRaSa-14}. 

For instance, for the first and second scarecrows in Figure \ref{fig:some_2D_models} one has $c=-\frac{1}{4}$, while $c=\frac{1}{4}$ for the last scarecrow. Moreover, by \cite[Cor.~2]{BoRaSa-14}, $\lambda$ is irrational for all infinite group models. This leads to the following corollary.

\begin{Corollary}
\label{cor:non-D-finite_(1,2)-type_Hadamard}
For any $(1,2)$-type Hadamard 3D model such that the group associated to the step set $T$ is infinite, the series $O(0,0,0;t)$ (and thus also $O(x,y,z;t)$) is non-D-finite.
\end{Corollary}

We list below important comments on Corollary \ref{cor:non-D-finite_(1,2)-type_Hadamard}.
\begin{itemize}
     \item First of all, Corollary \ref{cor:non-D-finite_(1,2)-type_Hadamard} is (to the best of our knowledge) the first non-D-finiteness result on truly 3D models (the 3D models considered in \cite{DuHoWa-16} have dimensionality $2$ in the sense of Definition \ref{def:dim}, and thus do not satisfy the main hypothesis \ref{it:hypothesis_half_space}, which guarantees the existence of a non-degenerate spherical triangle, see Section~\ref{subsec:WQP}). It answers an open question raised in \cite[Sec.~9]{BoBMKaMe-16} (concerning the possibility of extending the techniques of \cite{BoRaSa-14} to octant models).\smallskip
     \item In order to give a concrete application of Corollary \ref{cor:non-D-finite_(1,2)-type_Hadamard}, consider a model with arbitrary $U$ and $V$ (provided that the model is truly 3D), and with $T$ one scarecrow of Figure \ref{fig:some_2D_models}. This 3D model is non-D-finite since the 2D model associated with $T$ has an infinite group by \cite{BMMi-10}.\smallskip
     \item Note that Corollary \ref{cor:non-D-finite_(1,2)-type_Hadamard} can be extended to models with weights and arbitrarily big steps, provided that the hypothesis on the infiniteness of the group be replaced by the assumption that $\frac{\pi}{\arccos(-c)}$ is irrational. An algorithmic proof of the irrationality of such quantities is proposed in \cite[Sec.~2.4]{BoRaSa-14}, and further applied to some weighted models in \cite{DuHoWa-16}. 
     \item The proof of Corollary \ref{cor:non-D-finite_(1,2)-type_Hadamard} is a direct consequence of \cite[Cor.~2]{BoRaSa-14}, which states that for the $51$ unweighted non-singular step sets with infinite group in the quarter plane, the excursion exponent is irrational. By \cite[Thm~3]{BoRaSa-14} this implies that the series is non-D-finite.
\end{itemize}

\begin{Remark}[Combinatorial interpretation of the exponent]
For $(1,2)$-type models, 3D excursions may be decomposed as products of two lower dimensional excursions: a first excursion in the $(y,z)$-plane with the inventory $T$ and a second 1D excursion in $x$. This can easily be read on the formula of Proposition \ref{prop:exponent_(1,2)-type_Hadamard}: writing
\begin{equation*}
     \lambda=\left(\frac{\pi}{\arccos(-c)}+1\right)+\frac{3}{2},
\end{equation*}
the exponent is interpreted as the sum of the exponent of the 2D model (see \cite[Thm~4]{BoRaSa-14}) and of the universal exponent $\frac{3}{2}$ of a 1D excursion.
\end{Remark}

\subsection{(2,1)-Hadamard models}
\label{subsec:(2,1)-Hadamard}

\begin{Lemma}
\label{lem:covariance_matrix_(2,1)-type_Hadamard}
For any $(2,1)$-type Hadamard model, the matrix $\cov$ 
takes the form
\begin{equation}
\label{eq:covariance_matrix_(2,1)-type_Hadamard}
     \cov=\left(\begin{array}{lll}
     1 & a & 0\\
     a & 1 & 0\\
     0 & 0 & 1
     \end{array}\right), \quad \text{with}\quad a=\frac{\frac{\partial^2 \left.\chi\right\vert_{z_0}}{\partial x\partial y}}{\sqrt{\frac{\partial^2 \left.\chi\right\vert_{z_0}}{\partial x^2}\cdot \frac{\partial^2 \left.\chi\right\vert_{z_0}}{\partial y^2}}}(x_0,y_0),
\end{equation}
where $x_0,y_0,z_0$ are defined in \eqref{eq:chi_critical_point} and $\left.\chi\right\vert_{z_0}(x,y)=\chi(x,y,z_0)$.
\end{Lemma}

\begin{proof}
We solve the system \eqref{eq:chi_critical_point} in the $z$-variable first and obtain the point $z_0$ characterized by $T'(z_0)=0$. The first two equations of this system read
\begin{equation*}
     \frac{\partial U}{\partial x}(x,y)+T(z_0)\frac{\partial V}{\partial x}(x,y)=\frac{\partial U}{\partial y}(x,y)+T(z_0)\frac{\partial V}{\partial y}(x,y)=0.
\end{equation*}
The pair $(x_0,y_0)$ is the critical point associated to the mixture of models \eqref{eq:mixture_U_V}.
\end{proof}

The following result is derived similarly as Proposition \ref{prop:exponent_(1,2)-type_Hadamard}.
\begin{Proposition}
\label{prop:exponent_(2,1)-type_Hadamard}
The spherical triangles associated to $(2,1)$-type Hadamard models have angles $\frac{\pi}{2},\frac{\pi}{2},\arccos(-a)$ (as in Figure \ref{fig:isosceles}, left), with $a$ defined in \eqref{eq:covariance_matrix_(2,1)-type_Hadamard}. The smallest eigenvalue $\lambda_1$ of the Dirichlet problem and the exponent $\lambda$ are respectively given by 
\begin{equation*}
     \lambda_1=\left(\frac{\pi}{\arccos(-a)}+1\right)\left(\frac{\pi}{\arccos(-a)}+2\right),\qquad \lambda=\frac{\pi}{\arccos(-a)}+\frac{5}{2}.
\end{equation*}
\end{Proposition}

\subsubsection*{(2,1)-type Hadamard walks and mixing of 2D models}

From a probabilistic point of view, the $(2,1)$-type is slightly more interesting than the $(1,2)$-type. Many computations are indeed related to the concept of mixtures of two 2D probability laws. 

More precisely, the polynomials $U(x,y)$ and $V(x,y)$ in \eqref{eq:inventory_(2,1)-type_Hadamard} both induce a law (or a model) in 2D, which are {\it mixed} as below: 
\begin{equation}
\label{eq:mixture_U_V}
     \left.\chi\right\vert_{z_0}(x,y)=U(x,y)+T(z_0)V(x,y),
\end{equation}     
the parameter $z$ being specialized at $z_0$, the latter being defined by $T'(z_0)=0$. 

In the combinatorial case, for a 3D model we must have $T(z)=z+\overline{z}$, hence $z_0=1$ and $T(z_0)=2$. Equation \eqref{eq:mixture_U_V} becomes $U(x,y)+2V(x,y)$, which is the inventory of a 2D weighted walk (with possible weights $0,1,2,3$). Remark that it is not the first appearance of weighted 2D walks in the theory of (unweighted) 3D walks: in \cite[Sec.~7]{BoBMKaMe-16} (see in particular Figure 5), 2D projections of 3D models are analyzed, and these projections are typically weighted 2D walks; see also \cite{KaYa-15}. 

\subsubsection*{Computing $a$ in \eqref{eq:covariance_matrix_(2,1)-type_Hadamard}} 

From a technical point of view, computing $a$ and studying the rationality of $\frac{\pi}{\arccos(-a)}$ requires the same type of computations as above for $c$ and $\frac{\pi}{\arccos(-c)}$ (see Section \ref{subsec:(1,2)-Hadamard}). For an illustration see Example \ref{ex:mixing_scarecrows} below. However, some difficulties may occur from the fact that weighted steps are allowed:
\begin{itemize}
     \item It is not possible to exclude that a model with infinite group has a rational exponent $\lambda$ (this does not happen in the unweighted case \cite{BoRaSa-14}, but may happen in the weighted case, see examples in \cite{BoBMMe-18}).
     \item Knowing the critical exponents associated to $U$ and $V$ does not give much information on the exponent of the mixture \eqref{eq:mixture_U_V}.
\end{itemize}

\subsubsection*{Applications and examples}
We start with a result on non-D-finiteness, for a subclass of $(2,1)$-type Hadamard models. 

\begin{Corollary}
\label{cor:non-D-finite_(2,1)-type_Hadamard}
For any $(2,1)$-type Hadamard 3D model such that the group associated to the step set $V$ is infinite, and $U=V$ or $U=0$, the series $O(0,0,0;t)$ (and thus also $O(x,y,z;t)$) is non-D-finite.
\end{Corollary}

Corollary \ref{cor:non-D-finite_(2,1)-type_Hadamard} applies for several models, but the constraint of taking either $U=V$ or $U=0$ is quite strong. We now construct a more elaborate example. 

\begin{Example}
\label{ex:mixing_scarecrows}
Let $U,V$ be any of the first two scarecrows of Figure \ref{fig:some_2D_models} (possibly the same ones). These models have zero drift (meaning that the sum of the steps over the step set is zero), and thus critical point $(1,1)$, and an easy computation shows that they have the same covariance matrices. Then for any $T(x)=t_{1}x+t_0+t_{-1}\overline{x}$, the associated $(2,1)$-type Hadamard model defined by \eqref{eq:inventory_(2,1)-type_Hadamard} is non-D-finite.
\end{Example}

\section{Classification of the models and eigenvalues}
\label{sec:group_tiling}

\subsection{Motivations and presentation of the results}

In this section we would like to classify the $11,074,225$ models with respect to their triangle and the associated principal eigenvalue. The central idea is that there is a strong link between the group (as defined in Section \ref{subsec:group}) and the triangle. To understand this connection, we propose a novel, natural and manipulable geometric interpretation of the group, as a reflection group on the sphere. More precisely, we will interpret the three generators of the group as the three reflections with respect to the sides of the spherical triangle. We shall present three main features:
\begin{itemize}
     \item {\it Finite group case} (Section \ref{subsec:finite_group_case}): we interpret the group $G$ as a tiling group of the sphere, see Table \ref{tab:finite_groups} as well as Figures \ref{fig:tilings} and \ref{fig:some_further_tilings}. We also explain a few remarkable facts observed in the tables of \cite{BaKaYa-16}, on the number of different asymptotic behaviors. 
     \item {\it Infinite group case} (Section \ref{subsec:infinite_group_case}): the existence of a relation between the generators of the group can be read off on the angles. The simplest example is the relation $(\a\b)^m=1$, which on the triangle will correspond to an angle equal to $k\pi/m$ for some integer $k$. In particular, all triangles from the group
\begin{equation*}
     G_3=\<\a,\b,\c\mid\a^2,\b^2,\c^2,(\a\c)^2, (\a\b)^2>
\end{equation*}
of Table \ref{tab:various_groups} (Hadamard models) will have two right angles.
     \item {\it Exceptional models} (Section \ref{subsec:counterexamples}): for some infinite group models (a few hundreds of thousands), some unexpected further identities on the angles hold---{\it unexpected} means not implied by a relation between the generators, as explained above. 
     
     The most interesting case is given by some models in $G_1$ and $G_2$, which have a triangle with exactly $2$ right angles. Although these models do not have a Hadamard structure, their triangle has the same type as classical Hadamard models, and the principal eigenvalue (and hence the critical exponent) can be computed in a closed form. 
     
     There are also models with infinite group and three right angles (in this case, the exponent is $\frac{9}{2}$ and cannot be used to detect non-D-finiteness). Let us finally mention that there are two models with infinite group and having the same triangle as Kreweras 3D. See Theorem \ref{thm:exceptional_models}.
\end{itemize}

To summarize, classifying the triangles is close to, but different from classifying the groups. The latter task has already been achieved in \cite{BaKaYa-16} (finite group case; we have reproduced their results in Table \ref{tab:groups}) and \cite{KaWa-17} (infinite groups; see our Table \ref{tab:various_groups}), using a heavy computer machinery. However, the group classification is more precise, in the sense that the spherical triangle does not determine everything: infinite group models can have a tiling triangle, and Hadamard models are not the only ones to have birectangular triangles.

\begin{figure}
\includegraphics[width=0.35\textwidth]{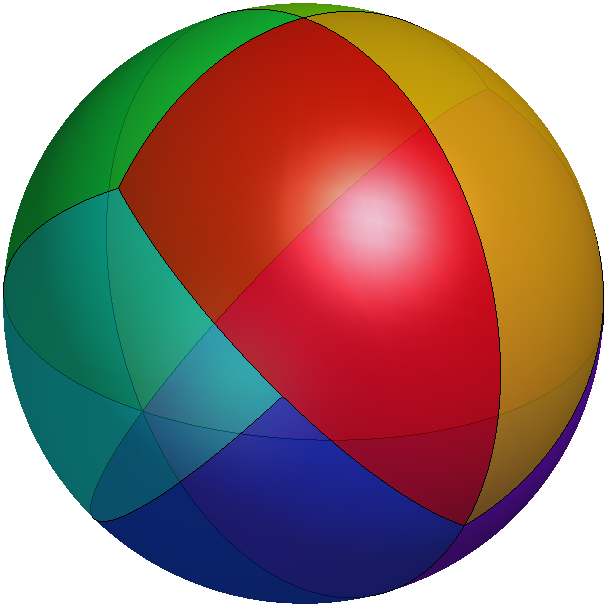}\qquad\qquad
\includegraphics[width=0.35\textwidth]{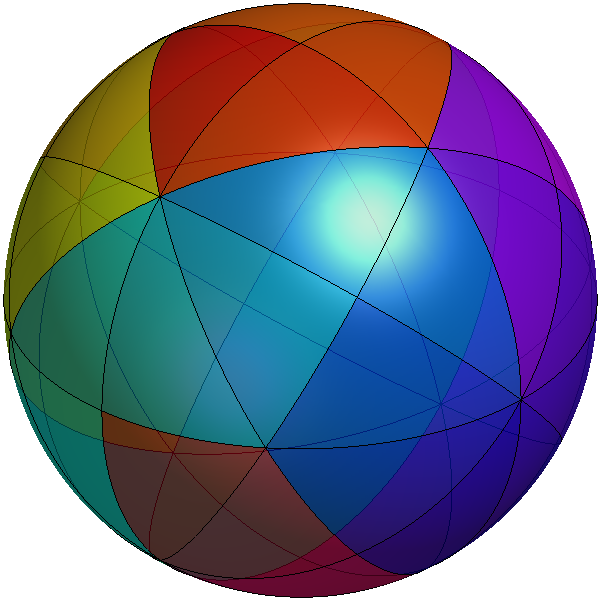}
\caption{Tilings associated to the triangles with angles $\frac{\pi}{3},\frac{\pi}{2},\frac{2\pi}{3}$ (left) and $\frac{\pi}{4},\frac{\pi}{3},\frac{\pi}{2}$ (right). These triangles correspond to the lines $9$ and $17$ in Table \ref{tab:finite_groups}}
\label{fig:some_further_tilings}
\end{figure}

\subsection{Finite group case}
\label{subsec:finite_group_case}

\subsubsection*{Some aspects of the group}
Let us recall a few applications of this concept:
\begin{itemize}
     \item When the group is finite and if in addition the orbit-sum of the monomial $x_1\times \cdots\times x_d$ under the group $G$, namely
     \begin{equation}
\label{eq:OS}
       {\rm OS}(x_1,\ldots ,x_d) = \sum_{g\in G} \text{sign}\;(g)\cdot g(x_1\times \cdots\times x_d),
\end{equation}
 is non-zero, one may obtain closed-form expressions for the generating function (as positive part extractions). See \cite{BMMi-10} for the initial application of this technique, called the orbit-sum method; it was further used in \cite{KaYa-15,BoBMMe-18}.\smallskip
     \item When the group is finite but the orbit-sum \eqref{eq:OS} is zero, it is still possible, in a restricted number of cases, to derive an expression for the generating function, see \cite{BMMi-10,KaYa-15,BoBMMe-18} for examples. The applicability of this technique is not completely clear.\smallskip
     \item Last but not least, in dimension $2$ there is an equivalence between the finiteness of the group and the D-finiteness of the generating functions (this is a consequence of the papers \cite{BMMi-10,BoKa-10,KuRa-12} altogether). 
\end{itemize}

Let us now examine each of the above applications in dimension $3$. The first item is still valid, as shown in \cite{BoBMKaMe-16,Ya-17}. As in the 2D case, the second item only works for a few cases. For instance, Figure 4 in \cite{BoBMKaMe-16} gives a list of $19$ non-Hadamard 3D models with finite group and zero orbit-sum, which are not solved at the moment. Finally, the third item is an open question. As an illustration, all $19$ previous models (including Kreweras 3D model) have a finite group, but as explained in \cite[Sec.~6.2]{BoBMKaMe-16}, these models do not seem D-finite. In this case, the equivalence in the third item would not be satisfied.

  \subsubsection*{Our contribution.} 
    The following result is summarized in Table \ref{tab:finite_groups}:
    \begin{Theorem}
    	\label{thm:finite_groups_classif}
    	Under the hypothesis \ref{it:hypothesis_half_space}, there are exactly $17$ triangles that are associated to finite groups. Each triangle corresponds to a particular eigenvalue computed in Table~\ref{tab:finite_groups}.
    \end{Theorem}
    
        \begin{proof} 
    	The proof of the above result is computational (all computations are done using symbolic tools and are exact) and is based on Theorem \ref{thm:DW_formula_exponent}. In each case, the critical point of the inventory function is found by solving \eqref{eq:chi_critical_point}. Once the critical point is found, we compute the covariance matrix \eqref{eq:covariance_matrix} and we use Lemma \ref{lem:exact_value_angles_triangle} to find the angles of the associated spherical triangle.
\end{proof}
Details concerning the symbolic validation of the results as well as Matlab codes used  are available on the webpage of the article: \href{https://bit.ly/2J4Vf3X}{\url{https://bit.ly/2J4Vf3X}}.
    
    {\small 
\begin{table}
\begin{center}
\begin{tabular}{|c|c|c|c|c|c|c|c|}
\hline
 & Eigenvalue & Exponent &  Nb tri. & Angles & Hadamard & Gr. Size & Group \\
 \hline
$1$& 4.261,734   & 3.124,084 &  $2$ &  $\ds \left[\frac{2\pi}{3},\frac{3\pi}{4},\frac{3\pi}{4}\right]$ & no & $48$ & $\Bbb{Z}_2\times S_4$ \\ 
\hline
$2$&  5.159,145   & 3.325,756 & $7$ & $\ds \left[\frac{2\pi}{3},\frac{2\pi}{3},\frac{2\pi}{3}\right]$ & no & $24$ & $S_4$\\
\hline
$3$&  6.241,748   & 3.547,890 & $2$ & $\ds \left[\frac{\pi}{2},\frac{2\pi}{3},\frac{3\pi}{4}\right]$ & no & $48$ & $\Bbb{Z}_2\times S_4$ \\
\hline
$4$&  6.777,108   & 3.650,869 & $5$ & $\ds \left[\frac{\pi}{2},\frac{2\pi}{3},\frac{2\pi}{3}\right]$ & no & $24$ & $S_4$\\
\hline
$5$&  $70/9$   & $23/6$  & $41$ & $\ds \left[\frac{\pi}{2},\frac{\pi}{2},\frac{3\pi}{4}\right]$ & yes & $16$ & $\Bbb{Z}_2 \times D_8$ \\
\hline
$6$&  $35/4$   & $4$  & $279$ & $\ds \left[\frac{\pi}{2},\frac{\pi}{2},\frac{2\pi}{3}\right]$ & yes/no & $12$ & $D_{12}$ \\
\hline
$7$&  $12$   & $9/2$ &  1,852 & $\ds \left[\frac{\pi}{2},\frac{\pi}{2},\frac{\pi}{2}\right]$ & yes & $8$ & $\Bbb{Z}_2 \times \Bbb{Z}_2 \times \Bbb{Z}_2$ \\
\hline
$8$&  12.400,051   & 4.556,691  & 2 & $\ds \left[\frac{\pi}{3},\frac{\pi}{2},\frac{3\pi}{4}\right]$ & no & $48$ & $\Bbb{Z}_2 \times S_4$ \\
\hline
$9$&  13.744,355   & 4.740,902 & 7 & $\ds \left[\frac{\pi}{3},\frac{\pi}{2},\frac{2\pi}{3}\right]$ & no & $24$ & $ S_4$ \\
\hline
$10$& $20$   & $11/2$ &  $172$ & $\ds \left[\frac{\pi}{3},\frac{\pi}{2},\frac{\pi}{2}\right]$  & yes/no & $12$ &  $D_{12}$ \\
\hline
$11$&  20.571,973   & 5.563,109 &  $2$ & $\ds \left[\frac{\pi}{4},\frac{\pi}{2},\frac{2\pi}{3}\right]$ & no & $48$ & $\Bbb{Z}_2 \times S_4$ \\
\hline
$12$&  21.309,407   & 5.643,211 &$7$ & $\ds \left[\frac{\pi}{3},\frac{\pi}{3},\frac{2\pi}{3}\right]$ & no & $24$ & $S_4$\\
\hline
$13$&  24.456,913   & 5.970,604 & $2$ & $\ds \left[\frac{\pi}{4},\frac{\pi}{3},\frac{3\pi}{4}\right]$ & no & $48$ & $\Bbb{Z}_2 \times S_4$\\
\hline
$14$&  $30$   & $13/2$ &  $41$ & $\ds \left[\frac{\pi}{4},\frac{\pi}{2},\frac{\pi}{2}\right]$ & yes & $16$ & $\Bbb{Z}_2\times D_8$\\
\hline
$15$&  $42$   & $15/2$  &  $5$ & $\ds \left[\frac{\pi}{3},\frac{\pi}{3},\frac{\pi}{2}\right]$ & no & $24$ & $S_4$ \\
\hline
$16$ &  49.109,945   & 8.025,663 &  $2$ & $\ds \left[\frac{\pi}{4},\frac{\pi}{4},\frac{2\pi}{3}\right]$ & no & $48$ & $\Bbb{Z}_2 \times S_4$ \\
\hline
$17$&  $90$   & $21/2$  & $2$ & $\ds \left[\frac{\pi}{4},\frac{\pi}{3},\frac{\pi}{2}\right]$ & no & $48$ & $\Bbb{Z}_2 \times S_4$ \\
\hline

\end{tabular}
\end{center}
\caption{Characterization of triangles and exponents associated to models with finite group. One can see some eigenvalues appearing in Lemma \ref{lem:thm6_Be-83}}
\label{tab:finite_groups}
\end{table}
}

\subsubsection*{Comments on Theorem \ref{thm:finite_groups_classif}}
We have computed the critical exponents for each one of the models corresponding to a finite group, using the fundamental eigenvalue of the associated spherical triangles. In some cases the eigenvalues are known explicitly and are written in rational form in Table~\ref{tab:finite_groups}. The computation procedure is described in Section \ref{sec:numerical_approximation}. We believe that all digits shown are accurate.

It is remarkable that among all possible $17$ exponents, each one is uniquely assigned to a particular spherical triangle. Moreover, each group can be realized as a reflection group for the associated triangles, giving a connection between combinatorial and geometric aspects. More precisely, we notice that all triangles associated to models with finite groups are {\it Schwarz triangles}, which means that they can be used to tile the sphere, possibly overlapping, through reflection in their edges. They are classified in \cite{Sc-1873} and a nice theoretical and graphical description can be seen on the associated Wikipedia page (\href{https://en.wikipedia.org/wiki/Schwarz_triangle}{\url{https://en.wikipedia.org/wiki/Schwarz_triangle}}). The classification of Schwarz triangles also includes information about their symmetry groups, which are seen to coincide with the combinatorial groups.

The triangle on the ninth line of Table~\ref{tab:finite_groups} is exactly half of Kreweras triangle. Accordingly (and this was confirmed by our numerical approximations) the principal eigenvalue of the models with half Kreweras triangle equals the second smallest eigenvalue of Kreweras model.

\begin{table}
  \begin{center}
  \begin{tabulary}{\textwidth}{|C|C||C|C|}
  \hline
    Group & Hadamard & Non-Hadamard ${\rm OS}\neq0$&Non-Hadamard ${\rm OS}=0$\\ \hline
    $\mathbb{Z}_{2}\times \mathbb{Z}_{2} \times \mathbb{Z}_{2}$ & 1,852 & 0 & 0\\ \hline
    $D_{12}$ & 253 & 66 & 132 \\ \hline
    $\mathbb{Z}_2 \times D_{8}$ & 82 & 0 & 0\\ \hline
    $S_4$ & 0 & 5 & 26\\ \hline
    $\mathbb{Z}_2\times S_4$ & 0 & 2 & 12\\
    \hline 
  \end{tabulary}
  \end{center}
  \caption{Number of models with finite group. Note that ${\rm OS}$ refers to the orbit-sum defined in \eqref{eq:OS}. The original version of this table may be found in \cite[Tab.~1]{BaKaYa-16}}
  \label{tab:groups}
\end{table}

\subsubsection*{Some remarks on the tables of \cite{BaKaYa-16}} 

In this paragraph we explain a few conjectural comments which appear in the captions of Tables 2, 3 and 4 of \cite{BaKaYa-16}.

First, Table 2 of \cite{BaKaYa-16} gives the guessed asymptotic behavior of the $12$ models with group $\mathbb{Z}_2\times S_4$ and zero orbit-sum (see our Table \ref{tab:groups}). The first remark of \cite{BaKaYa-16} is that the critical exponent $\beta$ of the generating function $O(1,1,1;t)$ seems to be related to the excursion exponent $\lambda$ by the formula
\begin{equation}
\label{eq:alpha_beta}
     \beta=\frac{\lambda}{2} - \frac{3}{4}.
\end{equation}
(Notice that the remark in \cite{BaKaYa-16} is stated with $+\frac{3}{4}$ and not $-\frac{3}{4}$; the reason is that our critical exponents are opposite to the ones in \cite{BaKaYa-16}.) Let us briefly mention that \eqref{eq:alpha_beta} is indeed true and is a consequence of Denisov and Wachtel results: by \eqref{eq:DW_formula_exponent} (resp.\ \cite[Thm~1]{DeWa-15}) one has 
\begin{equation}
\label{eq:DW_formula_exponent_000_111}
     \lambda=\sqrt{\lambda_1+\frac{1}{4}}+1 \quad \text{and}\quad \beta=\frac{1}{2}\left(\sqrt{\lambda_1+\frac{1}{4}}-\frac{1}{2}\right),
\end{equation}
for zero-drift models (which is the case of these models under consideration). This remark also applies to \cite[Tab.~3]{BaKaYa-16}, giving the excursion asymptotics for the $26$ models whose group is $S_4$ and orbit-sum zero (see again our Table \ref{tab:groups}).

The second comment of \cite{BaKaYa-16} that we can easily explain is about the number of different critical exponents. It is remarked in \cite{BaKaYa-16} that each exponent $\lambda$ seems to appear for exactly two models in their Table 2, and that in their Table~3 there are only four different exponents (namely, $-$5.643,21, $-$4.740,90, $-$3.650,86 and $-$3.325,75). This simply follows from the fact that in \cite[Tab.~2]{BaKaYa-16} (resp.\ \cite[Tab.~3]{BaKaYa-16}) there are only six (resp.\ four) types of spherical triangles, which appear twice for the second table.

\subsection{Infinite group case}
\label{subsec:infinite_group_case}

We have numerically computed for each model corresponding to an infinite group its associated spherical triangle, the eigenvalue and thus the exponent. Details about numerical computations can be found in Section \ref{sec:numerical_approximation}. 

As expected, the behavior is irregular (much more than in the finite group case) and the number of distinct eigenvalues, leading to distinct exponents, is more important. Therefore, we do not attempt to classify the models by the associated eigenvalues. In order to illustrate their repartition, we show in Figure \ref{fig:inf-groups} the distribution of the eigenvalues for triangles associated to the models in $G_1$, $G_2$ and $G_3$. Points having $y$-coordinate zero represent the cases where the hypothesis \ref{it:hypothesis_half_space} is not satisfied. 

As in the finite group case (Table \ref{tab:finite_groups}), we may wonder if there is a connection between the triangles associated to the models and their combinatorial group. The remark shown below strongly indicates that the analogous proposition holds.
In some cases, like for example when the triangle has two angles equal to $\pi/2$, the realization of the infinite group as a symmetry group for the triangle is more obvious.
\begin{figure}
\includegraphics[width=0.31\textwidth]{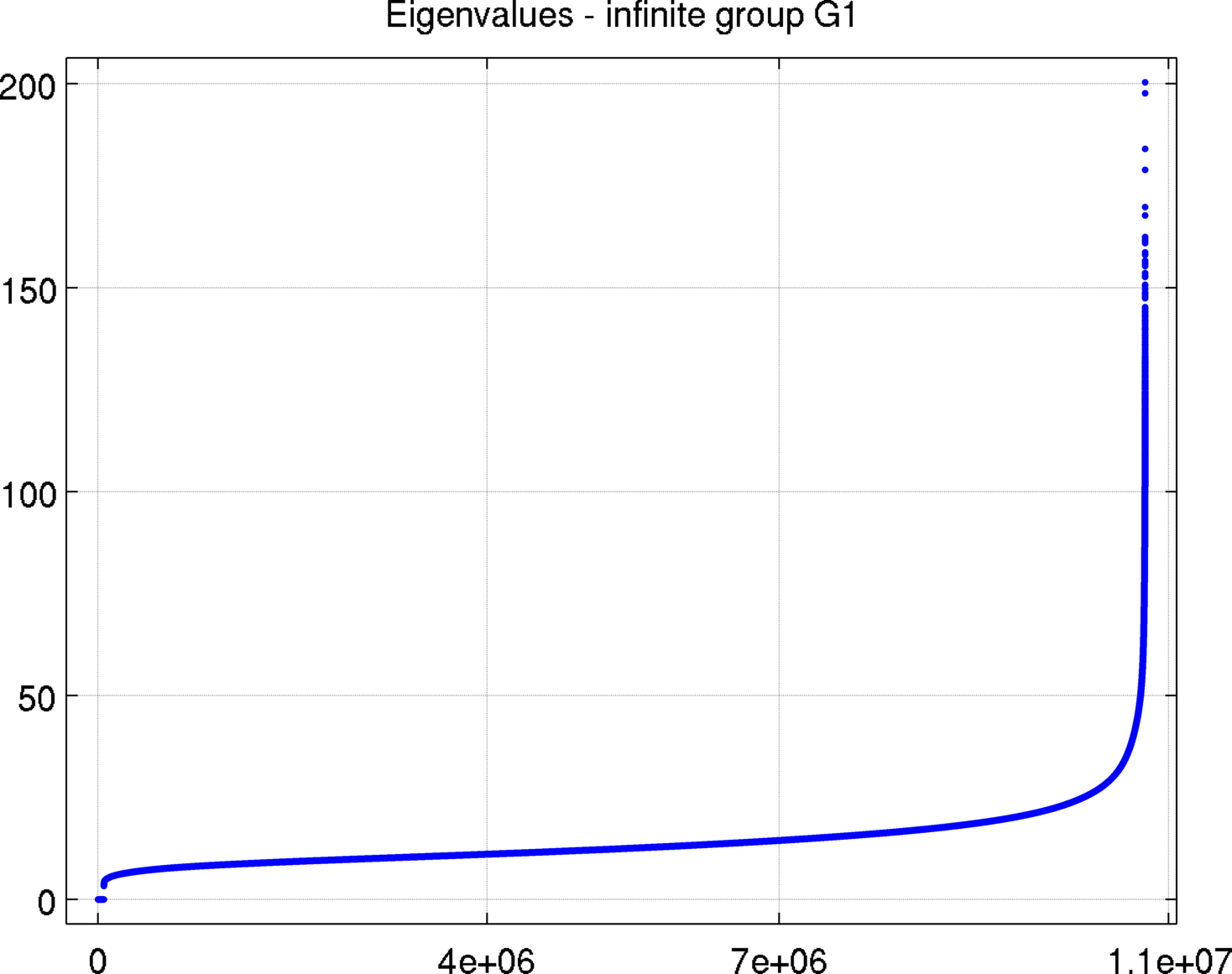}\quad
\includegraphics[width=0.295\textwidth]{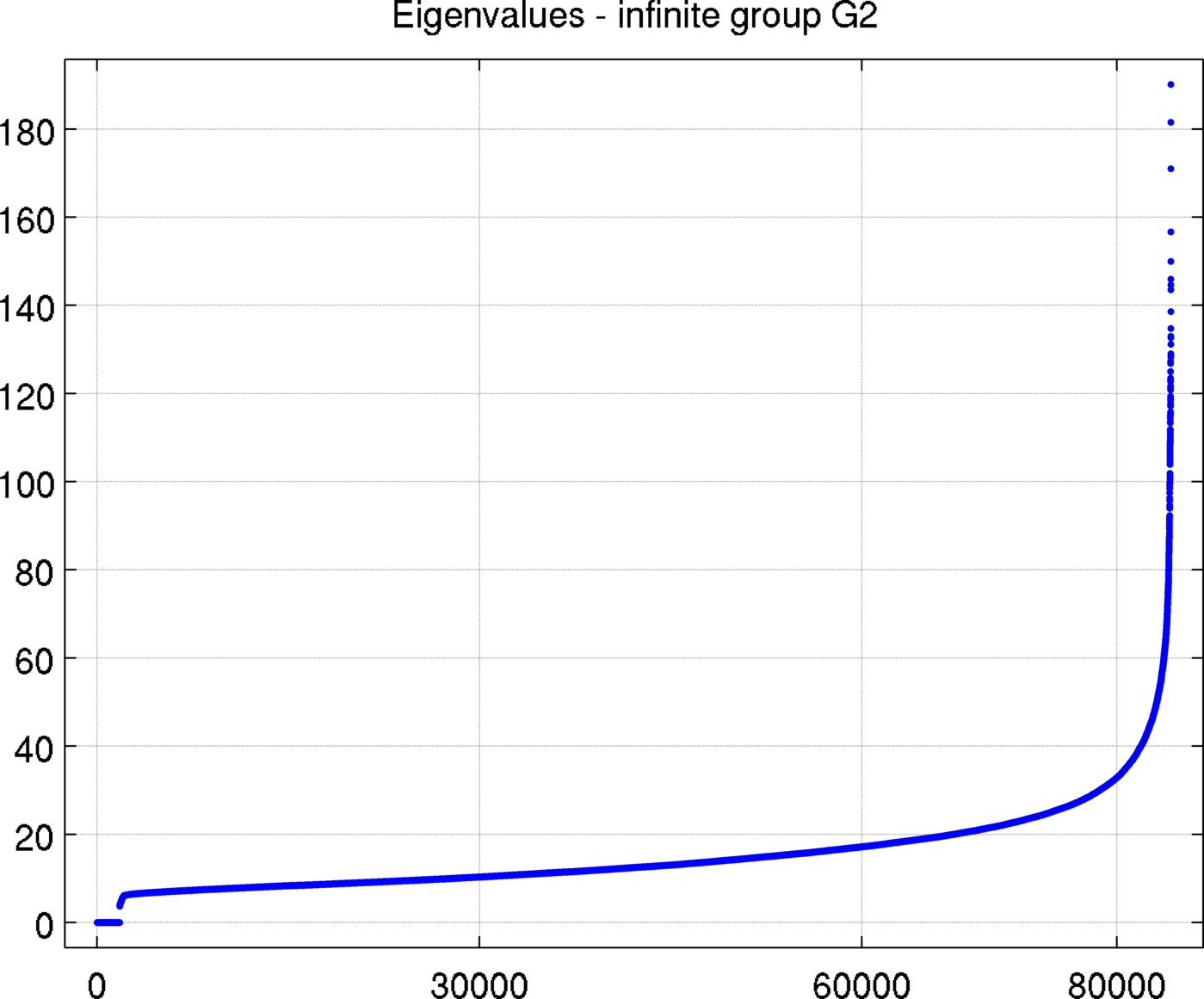}\quad
\includegraphics[width=0.3\textwidth]{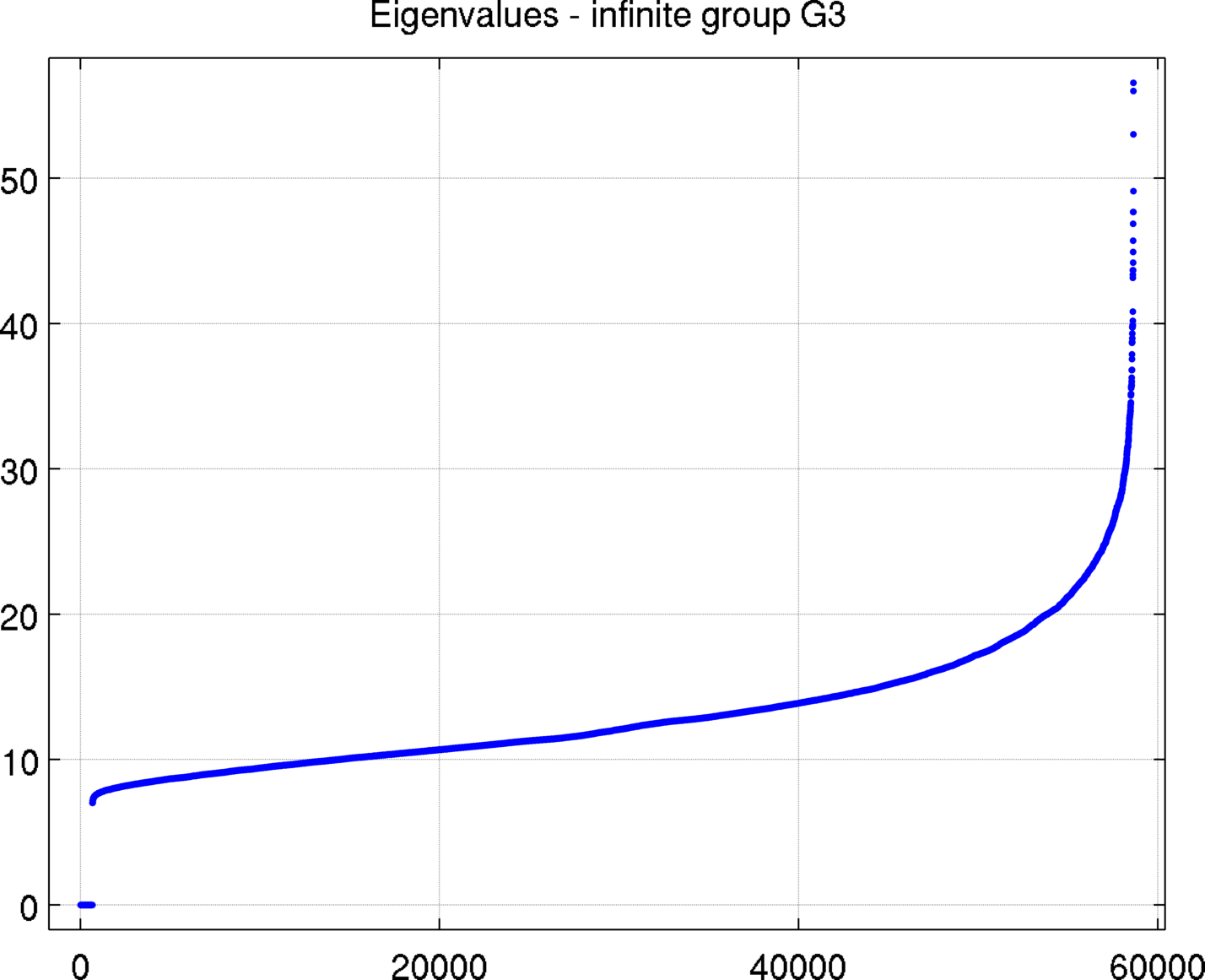}
\caption{Distribution of eigenvalues for triangles associated to the infinite groups $G_1$, $G_2$ and $G_3$ from Table \ref{tab:various_groups}}
\label{fig:inf-groups}
\end{figure}

\begin{Remark} 
\label{thm:infinite_symmetry}
All triangles associated to non-degenerate models with infinite groups satisfy the following property: the combinatorial group can be realized as a symmetry group of the triangle. We have two possibilities:
\begin{itemize}
     \item The generators $\a,\b,\c$ are the reflections with respect to the three sides of the triangle.
     \item In cases where the first possibility does not hold, it suffices to replace one of the reflections by its conjugate with respect to one of the other two (for example replace $\a$ by $\b \a \b$).
\end{itemize}
\end{Remark}

Remark \ref{thm:infinite_symmetry} is justified using numerical computations made in Section~\ref{sec:numerical_approximation}. We notice that only information on the angles is needed here (which can eventually be obtained using elementary functions and algebraic numbers). Therefore the arguments below can be justified using symbolic computations.

\begin{proof}[Summary of numerical observations justifying Remark \ref{thm:infinite_symmetry}]
For $G_1$ there is nothing to prove: we may choose $\a,\b,\c$ to be the three reflections with respect to the sides of the triangle, and no additional relation is required.

Among the 82,453 triangles associated to non-degenerate models of type $G_2$, exactly 79,219 have (at least) one right angle. Therefore, if $\a$ and $\b$ are reflections with respect to the sides adjacent to the right angle, then $(\a\b)^2=1$. The remaining 3,234 triangles have the property that for one particular labeling $\a,\b,\c$ of the symmetries associated to the sides of the triangles, the composition $\c\a\c\b$ is a rotation of angle $\pi$ and therefore $(\c\a\c\b)^2=1$. Therefore, after a transformation of the type $\a \gets \c\a\c$ described in  \cite{KaWa-17}, $G_2$ is represented as a group of symmetries of the associated triangles.

All triangles associated to non-degenerate models in $G_3$ have at least two angles equal to $\pi/2$, and $40$ among these have three right angles. Therefore, there is a labeling $\a,\b,\c$ of the reflections for which $(\a\b)^2=1$ and $(\a\c)^2=1$.

For triangles associated to groups among $G_{4},\ldots,G_{11}$ (all models in $G_{12}$ turn out to be included in a half-space) the relations are not always immediately identifiable with geometric aspects related to angles. One may find triangles with angles $\pi/k$ for groups having relations of the type $(\a\b)^k=1$, but this is not always the case. In order to validate these cases we use the following procedure:
\begin{enumerate}[label={(\roman{*})},ref={(\roman{*})}]
     \item\label{it:one}For a triangle $T$ associated to a group $G_n$, $n=4,\ldots,11$, for every one of the six permutations of the reflections $\a,\b,\c$, we construct the result of the transformations $\mathcal R(\a,\b,\c)(T)$, where $\mathcal R$ varies among the relations of the group $G_n$. We test if the resulting triangle after the above transformations coincides with the initial triangle. If this is the case for every relation $\mathcal R$ of $G_n$ then we have found a representation of $G_n$ as a group of reflections.
     \item\label{it:two}If the above step fails, then we consider transformations of the type $\mathcal R(\c\a\c,\b,\c)$ where, as before, $\a,\b,\c$ are reflections along the sides of the triangles and $\mathcal R$ varies among the relations of $G_n$.   
\end{enumerate}
For $G_5,G_7,G_8,G_9,G_{10},G_{11}$ the step \ref{it:one} of the above procedure finds a permutation of basic symmetries which satisfies the group relations. This also works partially for $G_4$ and $G_6$. For all remaining cases, the step \ref{it:two} finds a combination of reflections with one modification of the type $\a \gets \c \a \c$ such that $G_n$ is represented again as a symmetry group of the triangle.
\end{proof}

\subsection{Exceptional models}
\label{subsec:counterexamples}

In this section we are interested in a family of models, which is remarkable in the sense that the triangle has additional symmetries than those implied by the relations between the generators. We identify models which are non-Hadamard and which have two right angles, providing additional examples where we may compute exponents explicitly. Moreover, we identify triangles associated to infinite groups with three right angles or three angles equal to $2\pi/3$. Numerical investigations show that:
\begin{itemize}
	\item 200 models in $G_6$, 837 in $G_4$, 77,667 in $G_2$ and 31,005 in $G_1$ have exactly one right angle;
	\item 57,935 models in $G_3$, 1,552 in $G_2$ and 28,893 in $G_1$ have exactly two right angles;
	\item 40 models in $G_3$ and 563 models in $G_1$ have three right angles (see Figure \ref{fig:exceptional_examples});
	\item $2$ models in $G_4$ and $3$ models in $G_1$ have three $2\pi/3$ angles (see Figure \ref{fig:exceptional_examples}).
\end{itemize}
We have used numerical tools to find the numbers of models in each category: we inspect the triangles by using methods described in Section~\ref{sec:numerical_approximation} and use a tolerance of $10^{-8}$ in order to classify the angles of the triangle. Lists with steps corresponding to each one of the cases presented in the above numerical result can be accessed at the following link: \href{https://bit.ly/2J4Vf3X}{\url{https://bit.ly/2J4Vf3X}}.

Some of these results are validated using symbolic computations, as underlined below. 

\begin{Theorem} 
\label{thm:exceptional_models}
Among infinite group 3D models, there exist models for which the triangles have exactly one, two or three right angles. There also exist models having three $2\pi/3$ angles.  
\end{Theorem}

\begin{proof} 
	The cases of three right angles and three angles equal to $2\pi/3$ are completely validated using symbolic computations (using the same approach as in the proof of Theorem~\ref{thm:finite_groups_classif}).
	The existence of triangles with exactly one right angle in $G_1,G_2,G_4$ and $G_6$ and triangles with exactly two right angles in $G_1,G_2$ and $G_3$ is validated symbolically. 
\end{proof}

Figure \ref{fig:exceptional_examples} provides a few examples of Theorem \ref{thm:exceptional_models}. According to our computations, $40$ models in $G_3$ and $563$ models in $G_1$ have three right angles, and  $2$ models in $G_4$ and $3$ models in $G_1$ have three $2\pi/3$ angles. 	


\begin{figure}[bht]
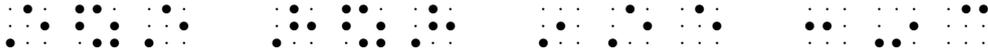

\begin{center}
\begin{tabular}{c@{\qquad}c@{\qquad}c@{\qquad}c}
\Stepset100 001 010 011 11 110 100 001 010 &
\Stepset100 011 010 011 11 110 100 011 010 &
\Stepset100 010 000 100 01 010 000 001 010 &
\Stepset000 110 000 110 01 000 000 000 011
\end{tabular}  
\end{center}
  \caption{Left: two models with a group $G_3$ and three right angles. Right: two models from $G_4$ with three angles of measure $2\pi/3$}
\label{fig:exceptional_examples}
\end{figure}

The first consequence of Theorem \ref{thm:exceptional_models} is to illustrate that the spherical triangle does not determine everything: 
\begin{itemize}
     \item infinite group models can have triangles which tile the sphere, 
     \item Hadamard models are not the only ones to admit birectangular triangles. 
\end{itemize}
Note that the first phenomenon already appears in 2D: it is indeed possible to construct two-dimensional models with infinite group and rational exponent, see, e.g., \cite{BoBMMe-18}. All known examples have either small steps and weights (not only $0$ and $1$), or admit at least one big step. However, restricted to the unweighted case there is equivalence between the infiniteness of the group and the irrationality of the exponent \cite{BoRaSa-14}. This is due to the fact that there are only $51$ (non-singular) infinite group models---and more than 11 millions of 3D models.

The second consequence of Theorem \ref{thm:exceptional_models} is the following:
\begin{Corollary}
\label{cor:exceptional_models}
For all models with exactly two right angles (say $a=b=0$), the exponent is given by
\begin{equation*}
     \lambda=\frac{\pi}{\arccos(-c)}+\frac{5}{2}.
\end{equation*}
In particular if $\frac{\pi}{\arccos(-c)}\notin \mathbb Q$ then the model is non-D-finite.
\end{Corollary}

\begin{figure}[bht]
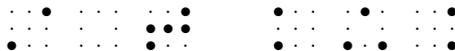

\begin{center}
\begin{tabular}{c@{\qquad}c}
\Stepset100 000 001   000 00 000   100 111 001 &
\Stepset100 000 100   101 00 010   001 000 001
\end{tabular}  
\end{center}
  \caption{Two models having a group $G_2$ (Table \ref{tab:various_groups}). Although these models do not have the Hadamard structure, they admit birectangular triangles and thus explicit eigenvalues, providing examples of application to Corollary \ref{cor:exceptional_models}}
\label{fig:application_corollary_exceptional}
\end{figure} 

As an example, we prove that the two models of Figure~\ref{fig:application_corollary_exceptional} admit irrational exponents. We present an alternative approach to the irrationally proof given in \cite[Sec.~2.4]{BoRaSa-14}.
\begin{proof}
Assume that $\arccos(-c) = \frac{p}{q}\pi$. Then obviously $\cos(q\arccos(-c)) -(-1)^p = 0$, and thus $c$ is a root of
\begin{equation}
\label{eq:Chebychev}
     f(x) = \cos(q\arccos(-x)) -(-1)^p,
\end{equation}
which is (up to an additive constant) a Chebychev polynomial.  For the first (resp.\ second) model in Figure \ref{fig:application_corollary_exceptional}, one has $c = \sqrt{7}/3$ (resp.\ $\sqrt{7/10}$), having respective minimal polynomials 
\begin{equation}
\label{eq:min_poly_examples}
     P(X)=9X^2-7 \quad \text{and} \quad P(X)=10X^2-7.
\end{equation}
Since Chebychev polynomials have leading coefficient equal to powers of $2$, this is the same for $f(x)$.

We recall that a polynomial in $\Bbb{Z}[X]$ is called {\it primitive} if its coefficients have no common factor. We also recall that the product of two primitive polynomials is again primitive, by Gauss' lemma.

Suppose that $P$ is a primitive polynomial and that $P$ divides, in $\Bbb{Q}[X]$, the polynomial $f$ defined in \eqref{eq:Chebychev}. Then there exists another polynomial $Q \in \Bbb{Q}[X]$ such that $f = PQ$. Suppose that $Q$ does not have integer coefficients. Then, let $c_Q$ be the least common multiple of the denominators of the coefficients of $Q$. In this way, the polynomial $c_QQ$ has integer coefficients and is primitive. Therefore 
\begin{equation*}
     P\cdot (c_QQ) = c_Q f,
\end{equation*}
and since $P$ and $c_QQ$ are both primitive, it follows that $c_Qf$ is also primitive. This leads to a contradiction if $c_Q>1$. Therefore $Q \in \Bbb{Z}[X]$.

We can now finish the proof and give the following general result: if $P \in \Bbb{Z}[X]$ is a primitive polynomial and the leading coefficient of $P$ is greater than $2$ and is not a power of $2$, then $P$ cannot divide $f$. Using the argument given in the previous paragraph we can conclude that $f$ admits a factorization of the type $f = PQ$ with $Q \in \Bbb{Z}[X]$. Therefore the leading coefficient of $f$ is a product of the leading coefficients of $P$ and $Q$. Since the leading coefficient of $P$ is greater than $2$ and is not a power of $2$, it cannot divide the leading coefficient of $f$, which is a power of $2$. 

In particular, both polynomials in \eqref{eq:min_poly_examples} are primitive and have leading coefficient greater than $2$, but not a power of $2$. Therefore they cannot divide $f$, and the exponent cannot be rational in these cases.
\end{proof}

\subsection{Equilateral triangles}

In spherical geometry, there exists an equilateral triangle with angles $\alpha$ for any $\alpha\in(\pi/3,\pi)$. The limit case $\alpha=\pi/3$ (resp.\ $\alpha=\pi$) is the empty triangle (resp.\ the half-sphere). 

Among the $11$ millions of models, we have numerically found $279$ different equilateral triangles. The most remarkable ones admit the angles $\pi/2$ (the simple walk), $2\pi/3$ (Kreweras), $\arccos(1/3)$ (polar triangle for Kreweras), $\arccos(\sqrt{2}-1)$ (the smallest equilateral triangle), $2\pi/5$, $3\pi/5$. It seems that only the first one admits an eigenvalue in closed-form. 

Except for the equilateral triangles with angles $\pi/2$ and $2\pi/3$, which exist in $G_3$ and $G_4$, all other equilateral triangles come from $G_1$. The list of equilateral triangles in $G_1$ and the list of all possible angles observed can be consulted on the webpage of the article: \href{https://bit.ly/2J4Vf3X
}{\url{https://bit.ly/2J4Vf3X}}.

\section{Numerical approximation of the critical exponent}
\label{sec:numerical_approximation}

\subsection{Literature}
\label{subsec:numerical_approximation_lit}
In lattice walk problems (and more generally in various enumerative combinatorics problems), it is rather standard to generate many terms of a series and to try to predict the behavior of the model, as the algebraicity or D-finiteness of the generating function, or the asymptotic behavior of the sequence. Having a large number of terms allows further to derive estimates of the exponential growth or of the critical exponent. More specifically, in the context of walks confined to cones, it is possible to make use of a functional equation to generate typically a few thousands of terms (the functional equation corresponds to a step-by-step construction of a walk, see \cite[Eq.~(4.1)]{BoBMKaMe-16} for a precise statement).

One can find in \cite{BoKa-09,BaKaYa-16,Gu-17} various estimates of critical exponents (contrary to the results presented here, the estimates of \cite{BoKa-09,BaKaYa-16,Gu-17} also concern the total numbers of walks---and not only the numbers of excursions). In \cite{BoKa-09}, Bostan and Kauers consider 3D step sets of up to five elements, and guess various asymptotic behaviors using convergence acceleration techniques. Bacher, Kauers and Yatchak go further in \cite{BaKaYa-16}, computing more terms and considering all 3D models (with no restriction on the cardinality of the step set). In \cite{Gu-17}, Guttmann analyses the coefficients of a few models by either the method of differential approximants or the ratio method. The methods of \cite{Gu-17} for generating the coefficients and for analyzing the resulting series are given in Chapters 7 and 8 of \cite{Gu-09}.

Some other techniques have the advantage of being applicable to any spherical triangle, not necessarily related to a 3D model. Using the stereographic projection, the 3D eigenvalue problem \eqref{eq:Dirichlet_problem} can be rewritten as a 2D eigenvalue problem for a different operator. Since the stereographic projection maps circles onto circles, the new domain is bounded by three arcs of circles and is thus rather simple. However, as expected, the eigenvalue problem becomes more complicated and is a priori not exactly solvable. See \cite{FiSn-74,Fi-73/74} for more details (in particular \cite[Eq.~(2.12) and Fig.~3]{Fi-73/74}). In \cite{DaSa-19}, the authors present a method for enclosing the principal eigenvalue of any triangle using validated numerical techniques.

Finally, the authors of \cite{LeMa-07} describe a Monte Carlo method for the numerical computation of the first Dirichlet eigenvalue of the Laplace operator in a bounded domain. It is based on the estimation of the speed of absorption of Brownian motion by the boundary of the domain. Theoretically this could certainly be used in our situation, but as in many probabilistic methods, it is hard to expect a precision such as ours (typically, ten digits).


\subsection{Finite element method}
\label{sec:finite-element}
 Our techniques are completely different here: we develop a finite element method and compute precise approximations of the eigenvalue (typically, $10$ digits of precision). We make available our codes at the following link: \href{https://bit.ly/2J4Vf3X}{\url{https://bit.ly/2J4Vf3X}}. The finite element computation consists in a few standard steps. For general aspects regarding finite element spaces defined on surfaces, we refer to \cite{DzEl-13}. We underline the fact that the method described below can be applied to general subsets of the sphere, not only for triangles. A method for computing eigenvalues of spherical regions using fundamental solutions was recently proposed in \cite{AlAn-18} for smooth domains on the sphere. The singular behavior generated by the corners of the triangles renders this method is not directly adapted to our needs.

\subsubsection*{(a) Triangulation of the domain.}
 In order to discretize the spherical triangle, we consider triangulations. For simplicity, we work with triangulations with flat triangles, which approximate the curved surface of the sphere better and better as the number of triangles increases. In order to construct such triangulations, we use the classical midpoint refinement procedure. Starting from a triangle, we construct the midpoints projected on the sphere, and we replace the initial triangle with four smaller triangles. We iterate this procedure a few times until we reach the desired precision. The triangulation procedure is described in Algorithm \ref{algo:triangulation}. Details concerning the number of refinements and the precision will be discussed below. An illustration  can be seen in Figure \ref{fig:triangulation}.

\begin{algorithm}
\caption{Constructing a triangulation of a spherical triangle}
\label{algo:triangulation}
\begin{algorithmic}[1]
\Require 
\begin{itemize}
\item $L$: Three distinct points $A,B,C$ on the sphere
\item $k$: number of refinements
\end{itemize}
\State Initialize the set of vertices $\mathcal{P}$
\State Initialize the set of triangles $\mathcal{T}$
\For {$iter=1:k$}
  \For {$T_i = XYZ \in \mathcal{T}$}
    \State Construct $M_1,M_2,M_3$ the projections on the sphere of the midpoints of $T_i$
    \State Add $M_1,M_2,M_3$ to $\mathcal{P}$
    \State Remove $T_i$ from $\mathcal{T}$
    \State Add the four triangles determined by $X,Y,Z,M_1,M_2,M_3$ to $\mathcal{T}$
  \EndFor 
\EndFor

\Return $\mathcal{P}$, $\mathcal{T}$
\end{algorithmic}
\end{algorithm}

\begin{figure}
\includegraphics[width=0.15\textwidth]{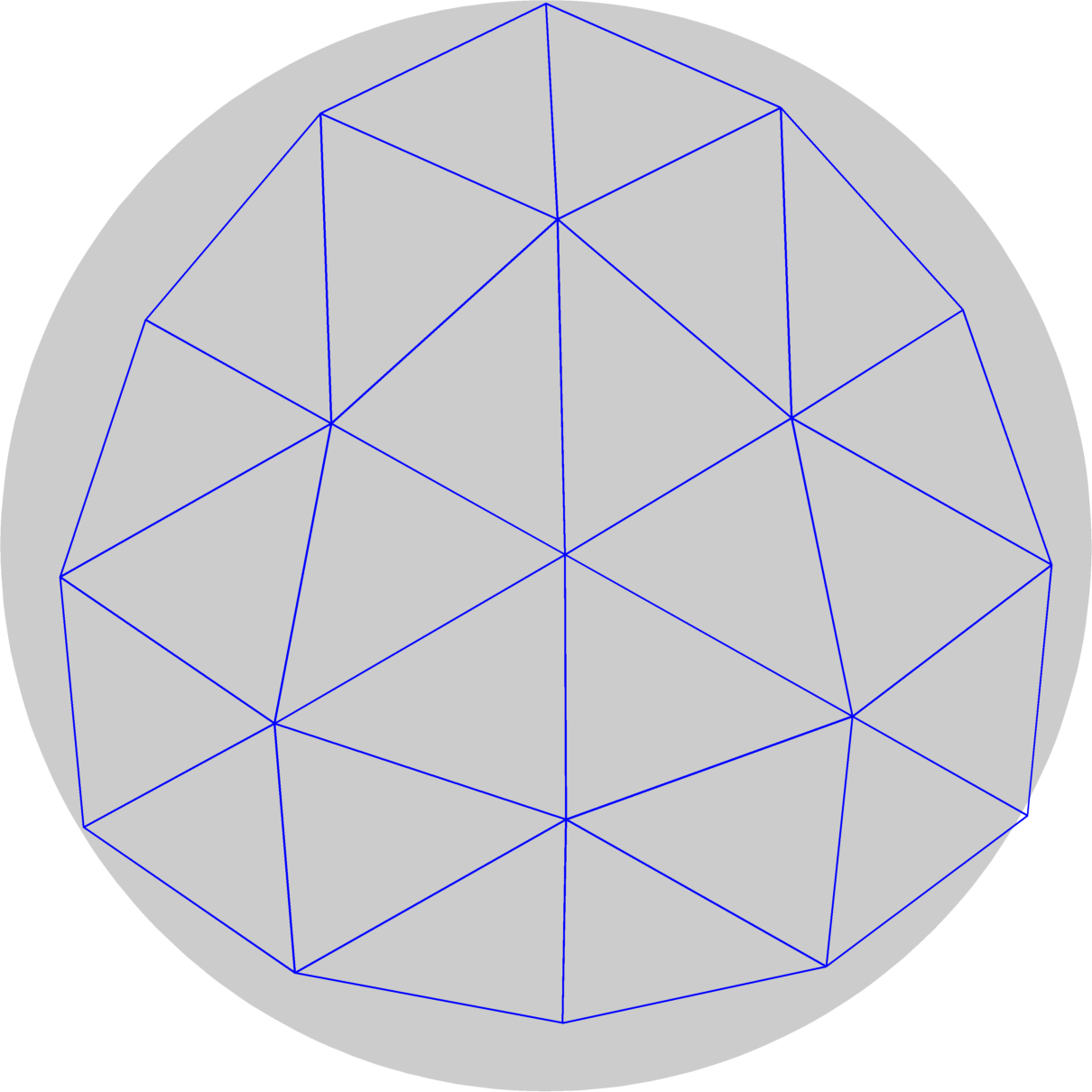}~
\includegraphics[width=0.15\textwidth]{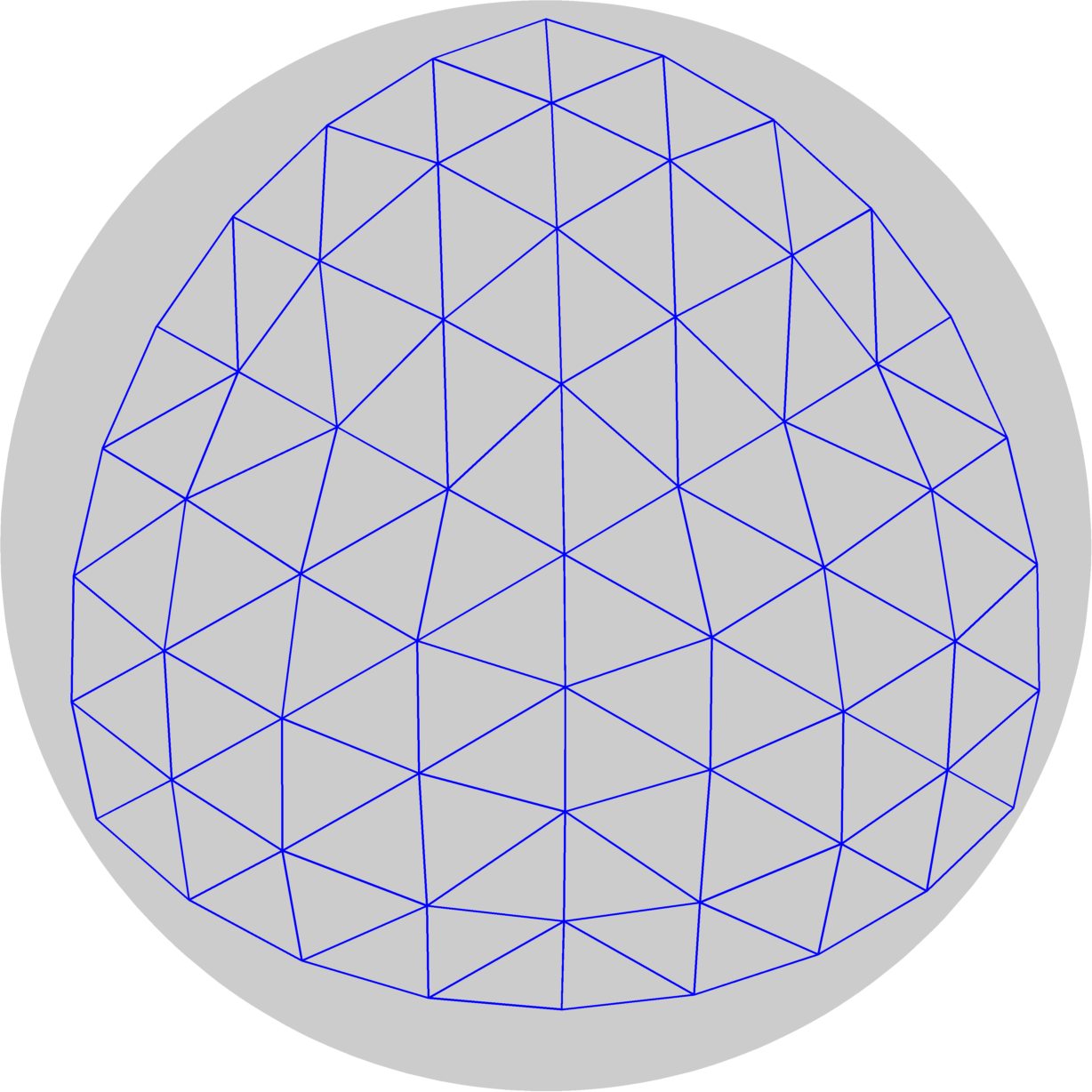}~
\includegraphics[width=0.15\textwidth]{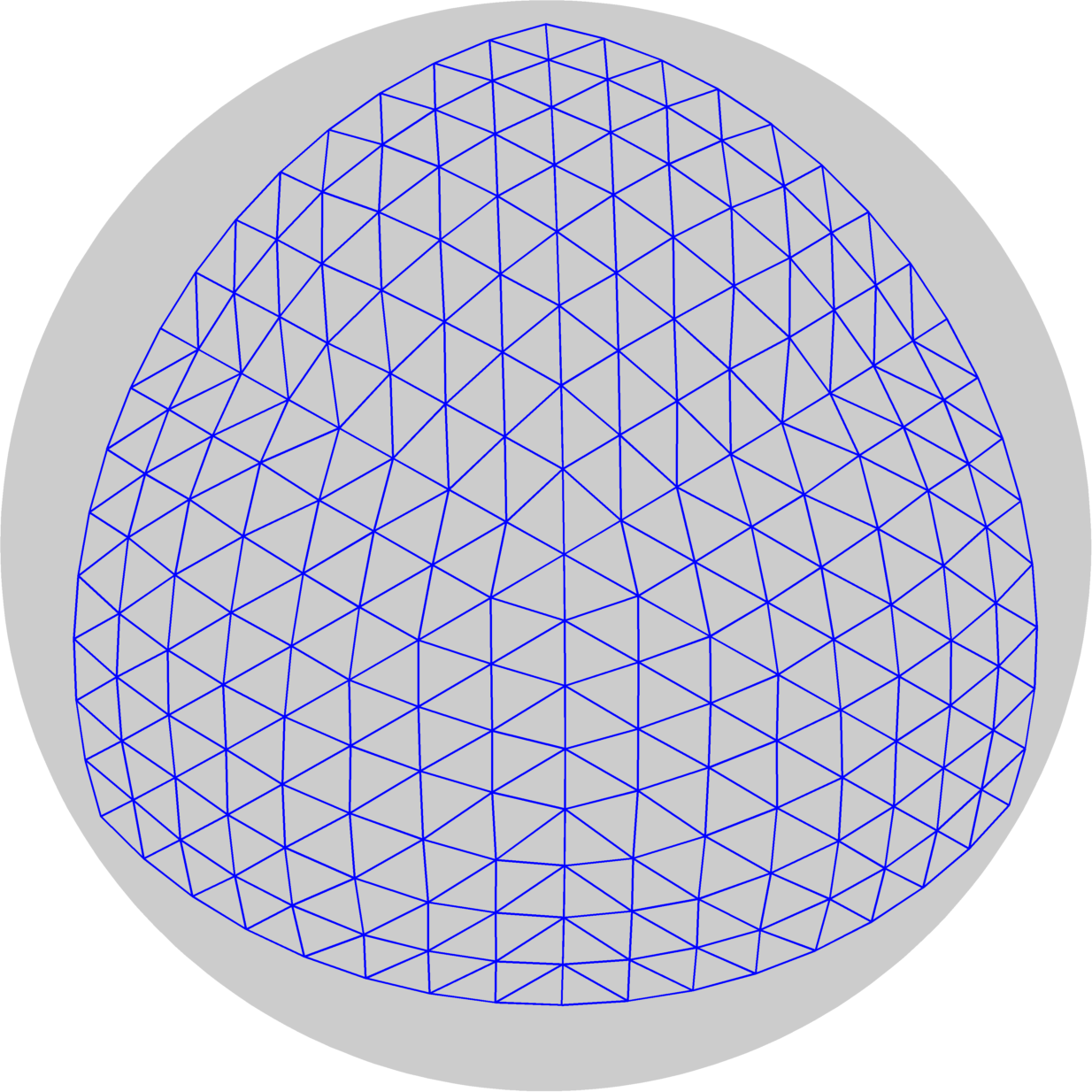}~
\includegraphics[width=0.15\textwidth]{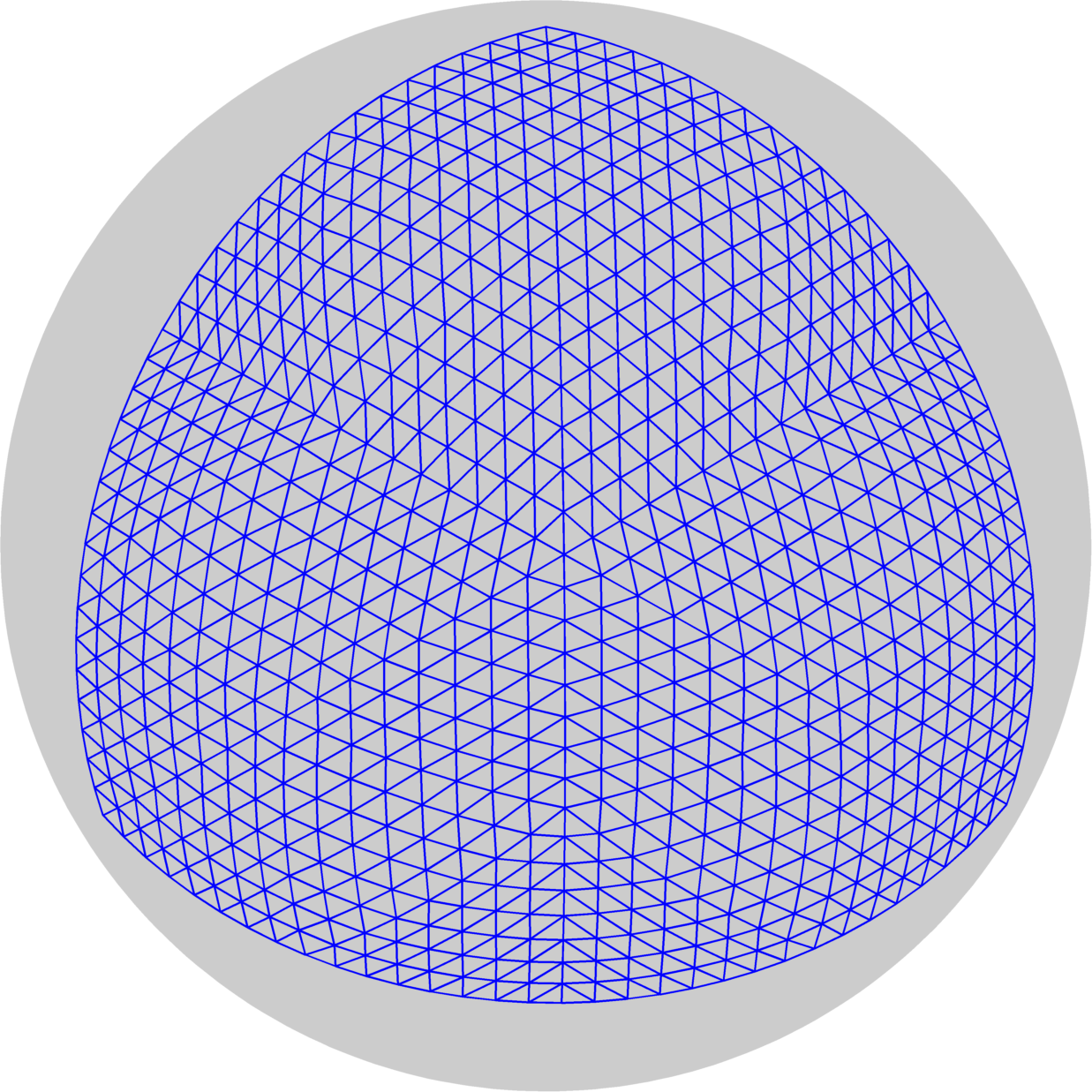}~
\includegraphics[width=0.15\textwidth]{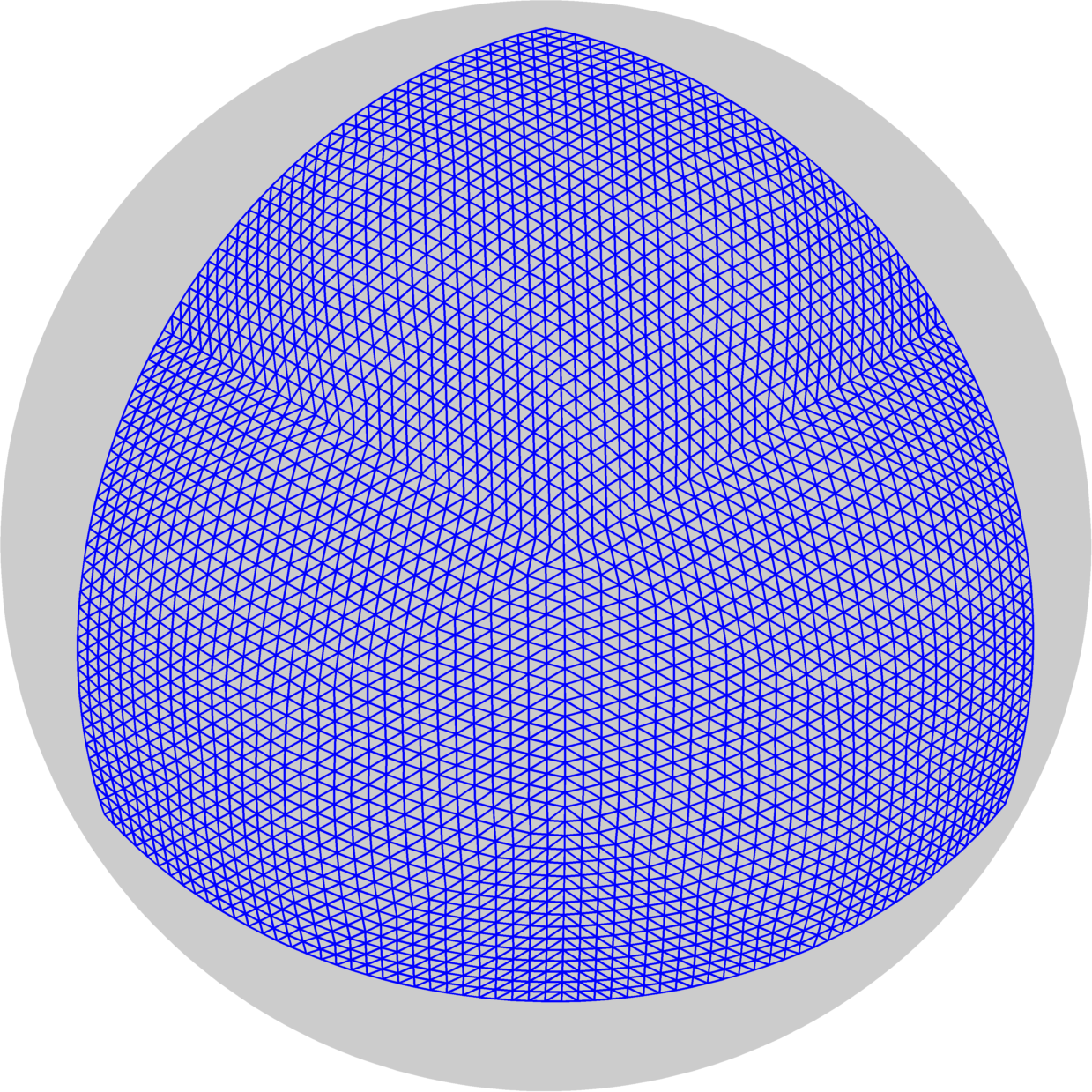}
\caption{Triangulation of a spherical triangle using successive refinements}
\label{fig:triangulation}
\end{figure}

\subsubsection*{(b) Assembly} Given a triangulation $\mathcal{T}$ of the spherical triangle, we denote by $(n_j)_{j=1}^N$ an enumeration of the nodes and by $(T_i)_{i=1}^M$ an enumeration of the triangles. Each $T_i$ contains the associated nodes to its three vertices. On the triangulation $\mathcal{T}$ we consider the $P_1$-Lagrange finite element space. This consists of associating to each node $n_j$ in the discretization a finite element function $\varphi_j$ which is piecewise affine on each of the triangles $T_i$ such that $\varphi_j(n_k)= \delta_{jk}$. A function $u \in H^1(\mathcal{T})$ is approximated by a linear combination of the finite element functions
\[ u \approx \sum_{j=1}^N a_j \varphi_j. \]
A standard approach in numerical computations is to use the weak formulation of the Laplace-Beltrami eigenvalue problem \[\int_{\mathcal T}\nabla_\tau u \nabla_\tau v = \lambda\int_{\mathcal T} uv,\quad  \forall v \in H^1(\mathcal{T}),\]
where $\nabla_\tau$ represents the tangential gradient to the surface of the sphere. When replacing $u$ and $v$ by their finite element approximations $u \approx \sum_{j=1}^N a_j \varphi_j$ and $v \approx \sum_{j=1}^N b_j \varphi_j$, we obtain the discrete version
\begin{equation}   \bo v^\intercal K \bo u = \lambda \bo v^\intercal M \bo u, \quad \forall \bo v \in \Bbb{R}^N,
\label{eq:discrete-weak-form}
\end{equation}
where $\bo u = (a_1,\ldots,a_N)$ and $\bo v = (b_1,\ldots,b_N)$. Here we have denoted by $K$ the rigidity matrix and by $M$ the mass matrix:
\begin{equation*}
      K=\left(\int_{\mathcal{T}} \nabla_\tau \varphi_i \cdot \nabla_\tau \varphi_j\right)_{1\leq i,j \leq n}\quad\text{and}\quad M = \left(\int_{\mathcal T} \varphi_i \varphi_j \right)_{1\leq i,j \leq n}.
\end{equation*} 
The matrices $K$ and $M$ are computed in an explicit way for every triangulation.

\subsubsection*{(c) Solving the discretized problem} We notice that the problem \eqref{eq:discrete-weak-form} is equivalent to the generalized eigenvalue problem
\[ K \bo u = \lambda M \bo u.\]
We are interested in the smallest eigenvalue associated to this problem. We solve this problem using the \texttt{eigs} function in Matlab.

\subsection{Improving the precision using extrapolation} 
\label{sec:extrapolation} We start by testing our algorithm for the spherical triangle having three right angles, for which the first eigenvalue is known and is equal to $12$. After $11$ refinements we arrive at the value 12.000,001,608 by using 12,589,057 discretization points. This is at the limit of what we can do using the finite element method without parallelization. The computation took $12$ minutes and used over $80$GB of RAM memory.

It is possible to improve the precision by using extrapolation procedures. Various techniques for improving the convergence of a sequence based on a finite number of terms can be found in \cite{SIAM-100}. We choose to use Wynn's epsilon algorithm, which starting from $2n+1$ terms can deliver the exact limit of a sequence, whenever this sequence can be written as a sum of $n$ geometric sequences. For any discretization parameter $h$ small enough the discrete eigenvalue approximation $\lambda_h$ has a Taylor-like expansion
\begin{equation*}
     \lambda_h = \lambda + C_1h^{k_1}+C_2h^{k_2}+\ldots,
\end{equation*}
where $k_i$ is an increasing sequence of positive real numbers. Applying Wynn's epsilon algorithm to a sequence of approximation corresponding to $h,h/2,\ldots,h/2^k$ will cancel the first terms in the above expansion, giving a better convergence rate. Applying the extrapolation procedure for the triple right angle triangle with $11$ refinement steps gives the value 11.999,999,999,999,46, which is close to machine precision. 

Wynn's epsilon algorithm is described in \cite[p.~247]{SIAM-100}. An illustration of the improvement of the convergence rate in the case of the triple right angle triangle is given in Figure \ref{fig:conv-rate}. One may note that the initial finite element approximation has convergence of order $2$, which is expected \cite{Fi-73,Kan-17}. On the other hand, the extrapolation procedure seems to have order of convergence at least $6$. Examples of applications of Wynn's algorithm and other extrapolation procedures can be found in \cite{SIAM-100}, together with Matlab codes. 

\begin{figure}
\includegraphics[width=0.5\textwidth]{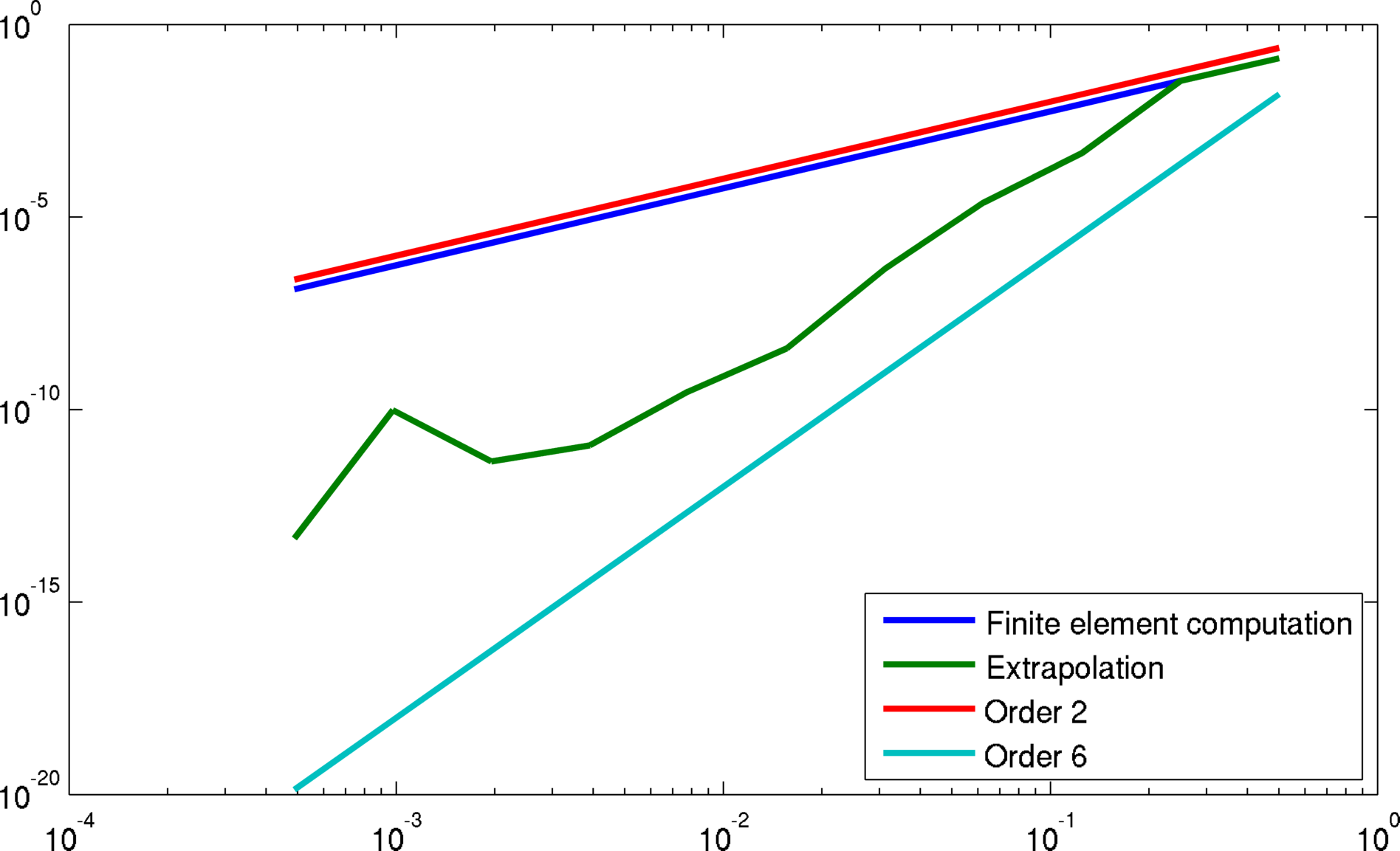}
\caption{Improvement of the convergence rate when applying the extrapolation procedure using Wynn's epsilon algorithm}
\label{fig:conv-rate}
\end{figure}

\subsection{Computing exponents}

When given a sequence of steps corresponding to a 3D walk, the first step is to test if all points belong to the same half-space, determined by a plane passing through the origin (see our assumption \ref{it:hypothesis_half_space}). We choose to loop over all pairs of steps and test if all remaining points are on the same side of the plane determined by the current pair and the origin.  

Once we confirm that the current sequence of steps is not contained in a half-space, we know that the inventory $\chi$ has a unique minimum point (which is obviously a critical point) in the positive octant. We use a numerical optimization procedure in order to find this minimizer. It is straightforward to compute the gradient and the Hessian of the inventory $\chi$, therefore a Newton algorithm is applicable. We use the function \texttt{fmincon} from the Matlab Optimization Toolbox to find the minimizers. In all our computations the numerical solution satisfies the critical equations \eqref{eq:chi_critical_point} with a numerical precision between $10^{-12}$ and $10^{-16}$. We mention that for cases of interest, exact solutions can be found (for example using Maple, see \cite[Sec.~2.4]{BoRaSa-14}). We choose to work with numerical approximations in view of the large number of computations involved in our study. 

Once the critical point is found, we may compute the coefficients $a,b,c$ of the covariance matrix and find the associated spherical triangle as described in Theorem \ref{thm:DW_formula_exponent}. Next we apply the procedure of Sections \ref{sec:finite-element} and \ref{sec:extrapolation} in order to compute the eigenvalue of the triangle. The exponent is then computed using the formula \eqref{eq:DW_formula_exponent}.

We make available at \href{https://bit.ly/2J4Vf3X
}{\url{https://bit.ly/2J4Vf3X}} our codes for constructing the triangulation, matrix assembly, eigenvalue computation and extrapolation procedure.

\subsection{Discussion of the computations}
We performed our computations in Matlab using floating point arithmetic ($16$ significant digits of precision). This leads to a significant acceleration of the computations. The computation of the eigenvalues is done using $7$ refinement steps for finite groups (typically $8$ digits of precision) and $5$ refinement steps for the infinite groups ($6$ digits of precision). We underline that the computation of the critical points of the inventory function can be computed using symbolic computations. We have a Matlab code which can do this for the majority of the cases we tested and it can be consulted on the webpage associated to this article.

Computations for the models associated to finite groups took a few hours on a laptop with an i7 processor and $16$GB of RAM memory. Computations for the infinite groups $G_2,G_3,\ldots,G_{12}$ were performed in a few hours on a $12$ core machine clocked at $3.5$Ghz and $256$GB of RAM. The computations for $G_1$ took $52$ hours on the same machine.

The precision for the computation of the elements of the triangle (points and angles) is always close to machine precision  (between $10^{-12}$ and $10^{-16}$). This is due to the fact that when minimizing a well-conditioned convex function, it is possible to obtain an upper bound for the distance between the numerical and exact minimizers in terms of the norm of the final gradient. Once the minimizer for the inventory function is found, all computations made in order to compute the triangles and the corresponding angles are explicit. For eigenvalues, the precision depends on the size of the triangulation.

If we want to have more precision for a particular model, it is possible to compute explicitly the components of the triangle and find the fundamental eigenvalue and the exponent close to machine precision. For example, we found that the eigenvalue of the triangle associated to the Kreweras model is 
\begin{equation*}
     \lambda_1 = \text{5.159,145,642,470}, 
\end{equation*}
where we believe that all digits present are correct. This is very close to the result of Guttmann \cite{Gu-17}.
 
\section{Miscellaneous}   
\label{sec:Miscellaneous}

\subsection{Walks avoiding an octant and complements of spherical triangles}
\label{subsec:walks_avoiding}

Rather than counting walks confined to an octant, one could aim at counting walks \textit{avoiding} an octant (or equivalently, walks confined to the union of seven octants). This model is briefly presented in \cite[Sec.~4]{Mu-19}. It is inspired by the dimension two case, where the model of walks avoiding a quadrant has started to be studied \cite{BM-16,Mu-19,RaTr-19}. At first sight, the (geometric) difference between quarter plane and three-quarter plane is anecdotal. However, the combinatorial complexity is much higher in the three quadrants; this is well illustrated by the fact that the simple walk model in the three-quarter plane has the same level of difficulty as quadrant Gessel walks, as shown by Bousquet-M\'elou in \cite{BM-16}.

Going back to walks avoiding an orthant in dimension three, it is clear from our construction that the critical exponent $\lambda$ of the excursion sequence is given by the same formula   \eqref{eq:DW_formula_exponent}, where $\lambda_1$ is now the principal eigenvalue of the Dirichlet problem \eqref{eq:Dirichlet_problem} on the \textit{complement of a spherical triangle}.

We were not able to identify any non-degenerate spherical triangle for which the principal eigenvalue of its complement is known to admit a closed form. A fortiori, we did not find any model for whose exponent of the excursion sequence in the seven octants has an explicit form. From that point of view, one notices the same complexification phenomenon as in dimension $2$.

Take the example of the simple walk, for which one should compute the principal eigenvalue of the complement of the equilateral right triangle. Even for this very simple case, no closed-form expression for $\lambda_1$ seems to exist, and the exponent $\lambda$ is conjectured to be irrational, see Conjecture~4.1 in \cite{Mu-19}. Numerical computations show that $\lambda$ is approximatively equal to 0.660,44.

\subsection{Walks avoiding a wedge}
Let us now mention the combinatorial model of 3D walks avoiding a wedge, which is a higher dimensional analogue of walks in the slit plane \cite{BMSc-02}. 3D walks avoiding a wedge also appear as a degenerate case of the previous model of Section~\ref{subsec:walks_avoiding}, when the triangle collapses into a single great arc of circle.

From a spherical geometry viewpoint, the problem becomes that of computing the first eigenvalue for the Dirichlet problem on the complement of a portion of great arc of circle of some given length in $[0,\pi]$. Such a problem is analyzed in \cite[Sec.~6]{Wa-74}. The extremal cases $\pi$ and $0$ are solved in \cite[Sec.~4]{Wa-74}, they correspond to $\lambda_1=\frac{3}{4}$ (exponent $2$) and $\lambda_1=0$ (exponent $\frac{3}{2}$), respectively. Tables 7 and 8 of \cite{Wa-74} provide approximate values of the fundamental eigenvalue for other values of the arc length.

\subsection{Other cones}
\label{subsec:other_cones}

As we have seen throughout the article, computing critical exponents for walks in $\mathbb N^3$ (or in any cone formed by an intersection of three half-spaces, by a linear transform) requires the computation of the principal eigenvalue of a spherical triangle. More generally, we could consider walks confined to an arbitrary cone $K$ in dimension $3$ or more (even so the natural combinatorial interpretation of positive walks is lost), and ask whether there exists a closed-form expression for the principal eigenvalue. However, only very few domains seem to admit such closed-form eigenvalues. Besides spherical digons and birectangle triangles, there are for instance the revolution cones, see Figure \ref{fig:spherical_cap}. The first eigenvalue (and in fact the whole spectrum) is described in Lemma \ref{lem:spectrum_revolution} of Section  \ref{subsec:properties_eigenvalue} in the appendix. From an analytic viewpoint, the domains leading to explicit eigenvalues have typically the property of separation of the variables, see \cite{RaTr-09a,RaTr-09b} for more details.

\begin{figure}
\begin{tikzpicture}
    \draw (2,0) arc (-0:-180:4cm and 2cm)coordinate[pos=0.75] (a);
    \draw[dashed] (2,0) arc (0:180:4cm and 2cm);
    \draw (2,0) arc (0:180:4cm);
    \draw (2,0) arc (0:-180:4cm);
    \draw[thick,red] ([shift=(90:2.5cm)]-2,0) arc (90:71.5:2cm) node[anchor=south]{\hspace{-2mm}$\zeta$};
    \begin{scope}[rotate around={90:(-2,0)}]
    \draw[top color=blue,bottom color=gray,opacity=0.5] (1.8,10mm) -- (-2,0) -- (1.8,-10mm);
    \draw[fill=blue!40] (1.8,0) circle(2mm and 10mm);
    \end{scope}
    \draw[thick,-stealth] (-2,0) -- (-2,5);
    \draw[thick,-stealth] (-2,0) -- (3,0);
    \draw[thick,-stealth] (-2,0) -- (-6,-3);
  \end{tikzpicture}
\caption{The revolution cone (or spherical cap) $K(\zeta)$ of apex angle $\zeta$}
\label{fig:spherical_cap}
\end{figure}
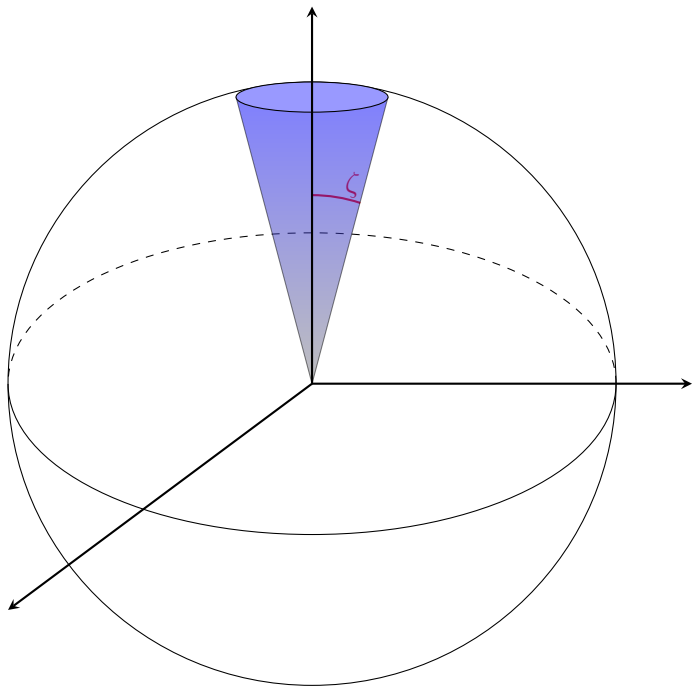

\subsection{Total number of walks}
Throughout the article we have considered the asymptotics of the number of excursions (essentially, the coefficients of $O(0,0,0;t)$, see \eqref{eq:generating_function} and \eqref{eq:asymptotics_excursions}), but other questions are relevant from an enumerative combinatorics viewpoint, as the asymptotics of the total number of walks (regardless of the ending position), or equivalently the coefficients of the series $O(1,1,1;t)$.

Let us recall that it is still an open problem to determine, in general, the asymptotics as $n\to\infty$ of the coefficients of $O(1,1,1;t)$. Assume that it has the form
\begin{equation}
\label{eq:asymptotics_total_number_walks}
     [t^n]O(1,1,1;t)= \varkappa\cdot\rho^{n}\cdot n^{-\beta}\cdot(1+o(1)).
\end{equation}
Recall from \cite{GaRa-16} that under the hypothesis \ref{it:hypothesis_half_space}, there exists $(x^*,y^*,z^*)\in[1,\infty)^3$ such that  
\begin{equation*}
     \min_{[1,\infty)^3} \chi=\chi(x^*,y^*,z^*),
\end{equation*}
and then the exponential growth $\rho$ in \eqref{eq:asymptotics_total_number_walks} is given by $\rho=\chi(x^*,y^*,z^*)$; compare with \eqref{eq:formula_rho}. There are essentially three cases for which the critical exponent $\beta$ in \eqref{eq:asymptotics_total_number_walks} is known:
\begin{itemize}
     \item Case of a drift in the interior of $\mathbb N^3$ ($\beta=0$);
     \item Zero drift (then $\beta=\frac{\lambda}{2} - \frac{3}{4}$, $\lambda$ being the critical exponent of the excursions \eqref{eq:asymptotics_excursions});
     \item Case when the point $(x^*,y^*,z^*)$ is in the interior of the domain $[1,\infty)^3$, i.e., $x^*>1$, $y^*>1$ and $z^*>1$ (in that case $\beta=\lambda$).
\end{itemize}

In the first case (drift with positive entries), the exponent is obviously $0$ by the law of large numbers. In the second case the exponent $\beta$ is given by the formula \eqref{eq:DW_formula_exponent_000_111} proved in \cite{DeWa-15}. As recalled in \eqref{eq:alpha_beta}, $\beta$ is a simple affine combination of $\lambda$, namely $\beta=\frac{\lambda}{2} - \frac{3}{4}$. The last case is proved by Duraj in \cite{Du-14}. The original statement of Duraj is in terms of the minimum of the Laplace transform of the step set on the dual cone, but it is equivalent to the one presented above, after an exponential change of variables and using the self-duality of the octant $\mathbb N^3$. The hypothesis that the point $(x^*,y^*,z^*)$ is an interior point cannot be easily translated in terms of the drift; note, however, that it contains the case of a drift with three negative coordinates. The intuition of the formula $\beta=\lambda$ is that the drift being directed towards the vertex of the cone, a typical walk will end at a point close to the vertex, and thus asymptotically the total number of walks is comparable to the number of excursions.

Among the more than $11$ millions of models, there are of course many examples corresponding to each of the above cases.

\subsection{Walks in the quarter plane and spherical digons}
\label{subsec:WQP}

In this paragraph we briefly explain how the more classical model of walks in the quarter plane enters into the framework of spherical geometry. In one sentence, spherical triangles become degenerate and should be replaced by spherical digons, see Figure \ref{fig:digon}, for which the principal eigenvalue (and in fact the whole spectrum) is known.

Indeed, given a 2D positive random walk $\{(X(n),Y(n))\}$, we can choose an arbitrary 1D random walk $\{Z(n)\}$ and embed the 2D model as a 3D walk $\{(X(n),Y(n),Z(n))\}$, with no positivity constraint on the last coordinate. The natural cone is therefore $\mathbb N^2\times \mathbb Z$, or after the decorrelation of the coordinates, the cartesian product of a wedge of opening $\alpha$ and $\mathbb Z$. On the sphere $\mathbb S^2$, the section of the latter domain is precisely a spherical digon of angle $\alpha$. 

The smallest eigenvalue $\lambda_1$ of a spherical digon is easily computed, see, e.g., \cite[Sec.~5]{Wa-74}:
\begin{equation*}
     \lambda_1=\frac{\pi}{\alpha}\left(\frac{\pi}{\alpha}+1\right).
\end{equation*}
The formula \eqref{eq:DW_formula_exponent} relating the smallest eigenvalue to the critical exponent gives an exponent equal to $\frac{\pi}{\alpha}+\frac{3}{2}$. To find the exponent of the initial planar random walk we have to subtract $\frac{1}{2}$ (exponent of an unconstrained excursion in the $z$-coordinate), which by \cite{DeWa-15,BoRaSa-14} is the correct result.

\begin{figure}[ht]
\tdplotsetmaincoords{90}{90}
\begin{tikzpicture}[scale=3,tdplot_main_coords]

\tdplotsetthetaplanecoords{90}
\tdplotdrawarc[tdplot_rotated_coords,thick]{(0,0,0)}{0.8}{0}{360}{}{}
\tdplotsetrotatedcoords{80}{85}{0}
\tdplotdrawarc[dashed,tdplot_rotated_coords,name path=blue,color=blue]{(0,0,0)}{0.8}{0}{360}{}{}
\tdplotdrawarc[tdplot_rotated_coords]{(0,0,0)}{0.8}{0}{180}{}{}
\tdplotsetrotatedcoords{100}{95}{0}
\tdplotdrawarc[dashed,tdplot_rotated_coords,name path=green,color=green]{(0,0,0)}{0.8}{0}{360}{}{}
\tdplotdrawarc[tdplot_rotated_coords]{(0,0,0)}{0.8}{0}{180}{}{}
%


\path [name intersections={of={green and blue}, total=\n}]  
\foreach \i in {1,...,\n}{(intersection-\i) circle [radius=0.5pt] coordinate(gb\i){}};




\end{tikzpicture}
\caption{A spherical digon is the domain bounded by two great arcs of circles}
\label{fig:digon}
\end{figure}
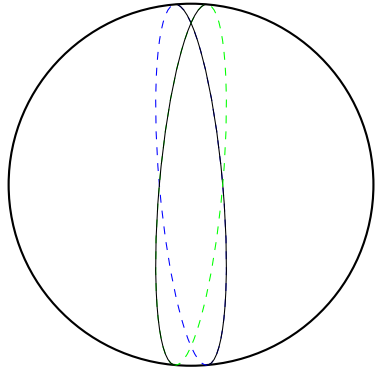

\subsection{Exit time from cones for Brownian motion}  
\label{subsec:Brownian_motion}

As shown in \cite{De-87,BaSm-97} (see in particular \cite[Cor.~1]{BaSm-97}), the exit time of a standard $d$-dimensional Brownian motion from a cone $K$ behaves when $t\to\infty$ as
\begin{equation}
\label{eq:asymptotic_Brownian_d}
     \mathbb P_x[\tau>t] =B_1\cdot m_1\left(\frac{x}{\vert x\vert}\right)\cdot\left(\frac{\vert x\vert^2}{2}\right)^{\lambda_1(K)/2}\cdot t^{-\lambda_1(K)/2},
\end{equation}
where $\lambda_1(K)$ is equal to
\begin{equation*}
     \lambda_1(K)=\sqrt{\lambda_1(C)+\left(1-\frac{d}{2}\right)^2}+\left(1-\frac{d}{2}\right)
\end{equation*}
and $\lambda_1(C)$ is the principal eigenvalue of the Dirichlet problem on the section $C=K\cap\mathbb S^{d-1}$:
\begin{equation}
\label{eq:Dirichlet_problem_general}
     \left\{
\begin{array}{rll}
     -\Delta_{\mathbb S^{d-1}}m&=\ \Lambda m & \text{in } C,\\
     m&=\ 0& \text{in } \partial C.
     \end{array}
     \right.
\end{equation}
In the asymptotics \eqref{eq:asymptotic_Brownian_d}, $m_1$ is the (suitably normalized) eigenfunction associated to $\lambda_1$. 

In the particular case of 3D Brownian motion, if the cone $K$ is an intersection of three half-spaces, the section $C$ becomes a spherical triangle and the exponent in \eqref{eq:asymptotic_Brownian_d} is directly related to the principal eigenvalue of a spherical triangle, which is the main object of investigation studied in this paper.

Let us finally comment on the case of non-standard Brownian motion (in arbitrary dimension $d\geq 2$). First, the case of non-identity covariance matrices is easily reduced to the standard case, by applying a simple linear transform (notice, however, that this implies changing the initial cone, and therefore the domain of the Dirichlet problem). The situation is more subtle in the case of drifted Brownian motion: various asymptotic regimes exist, depending on the position of the drift with respect to the cone and the polar cone \cite{GaRa-14}. In some regimes the exponent in \eqref{eq:asymptotic_Brownian_d} involves the principal eigenvalue $\lambda_1$; in some other cases (e.g., a drift which belongs to the interior of the cone) the exponent is independent of the geometry of the cone.

\subsection{Sketch of the proof of Theorem \ref{thm:DW_formula_exponent}}
\label{subsec:proof_Thm_DW}

This proof follows a certain number of steps that we now briefly recall. For more details we refer to the presentation of \cite{BoRaSa-14} (see Section~2.3 there, see also \cite{DeWa-15}). 

\medskip

$\bullet$ {\it Probabilistic interpretation:} Following Denisov and Wachtel \cite[Sec.~1.5]{DeWa-15}, the main idea is to write the number of excursions (see \eqref{eq:generating_function}) as a local probability for a random walk, namely,
\begin{equation}
\label{eq:comb_prob}
     o(i,j,k;n) = \vert \mathcal S\vert^n \mathbb P\left[\sum_{\ell=1}^{n}(X(\ell),Y(\ell),Z(\ell))= (i,j,k),\tau >n\right],
\end{equation}
where $\{(X(\ell),Y(\ell),Z(\ell))\}$ are i.i.d copies of a random variable $(X,Y,Z)$ having uniform law on the step set $\mathcal S$, i.e., for each $s\in\mathcal S$, $\mathbb P[(X,Y,Z)=s]={1}/{\vert\mathcal S\vert}$, and where $\tau$ is the first hitting time of the translated cone $(\mathbb N\cup\{-1\})^3$. At the end we shall apply the local limit theorem \cite[Thm~6]{DeWa-15} for random walks in cones. The latter theorem gives the asymptotics of \eqref{eq:comb_prob} for normalized random walks, in the sense that the increments of the random walks should have no drift, i.e., $\sum_{s\in\mathcal S}\mathbb P[(X,Y,Z)=s]\cdot s=0$, and a covariance matrix \eqref{eq:covariance_matrix} equal to the identity.

\medskip

$\bullet$ {\it Removing the drift:} It is rather standard to perform an exponential change of measure so as to remove the drift of a random variable (this is known as the Cram\'er transform). Define the triplet $(X_1,Y_1,Z_1)$ by (with $s=(s_1,s_2,s_3)\in\mathcal S$)
\begin{equation*}
     \mathbb P[(X_1,Y_1,Z_1)=s]= \frac{x_0^{s_1} y_0^{s_2} z_0^{s_3}}{\chi(x_0,y_0,z_0)}.
\end{equation*}
Under our hypothesis \ref{it:hypothesis_half_space}, the drift of $(X_1,Y_1,Z_1)$ is zero if and only if $(x_0,y_0,z_0)$ is solution to \eqref{eq:chi_critical_point}, which we now assume.

\medskip

$\bullet$ {\it Covariance identity:} We first normalize the variables by
\begin{equation*}
     (X_2,Y_2,Z_2) = \left(\frac{X_1}{\sqrt{\mathbb E[X_1^2]}},\frac{Y_1}{\sqrt{\mathbb E[Y_1^2]}},\frac{Z_1}{\sqrt{\mathbb E[Z_1^2]}}\right),
\end{equation*}
so that the variances of the coordinates are $1$, and more generally the covariance matrix of $(X_2,Y_2,Z_2)$ is given by \eqref{eq:covariance_matrix}. Writing $\cov=S S^\intercal$ as in \eqref{eq:square_root_equation} and
\begin{equation*}
     \left(\begin{array}{l} X_3\\Y_3\\Z_3\end{array}\right) = S^{-1}\left(\begin{array}{l} X_2\\Y_2\\Z_2\end{array}\right),
\end{equation*}
we obtain that $(X_3,Y_3,Z_3)$ has an identity covariance matrix, since $S^{-1}\cdot\cov\cdot(S^{-1})^\intercal$ is the identity. If $(X,Y,Z)$ is defined in the octant $\mathbb R_+^3$, then $(X_3,Y_3,Z_3)$ takes its values in the cone $S^{-1}\mathbb R_+^3$. 

\medskip

$\bullet$ {\it Conclusion:} Remarkably, the probability on the right-hand side of \eqref{eq:comb_prob} can be expressed in terms of the random walk with increments $(X_3,Y_3,Z_3)$. For instance, for $(i,j,k)$ equal to the origin,
\begin{multline*}
     \quad\mathbb P\left[\sum_{\ell=1}^{n}(X(\ell),Y(\ell),Z(\ell))= (0,0,0),\tau >n\right]=\\\left(\frac{\chi(x_0,y_0,z_0)}{\vert\mathcal S\vert}\right)^n
     \mathbb P\left[\sum_{\ell=1}^{n}(X_3(\ell),Y_3(\ell),Z_3(\ell))= (0,0,0),\tau_3 >n\right],\quad
\end{multline*}
with $\tau_3$ denoting the exit time from the cone $S^{-1}\mathbb R_+^3$. Using \eqref{eq:comb_prob} and applying \cite[Thm~6]{DeWa-15} finally gives the result stated in Theorem \ref{thm:DW_formula_exponent}.

\subsection{Further properties of the covariance matrix}
\label{subsec:further_prop_cov}

We establish a strong relationship between the cosine matrix of the angles and the Coxeter matrix of the group, we then interpret the covariance matrix as a Gram matrix, and finally we show that it is possible to realize any spherical triangle as a walk triangle.

\subsubsection*{Relation with the Coxeter matrix} 
Assume that there exists a representation of the group $G$ of Section \ref{subsec:group} as
\begin{equation*}
     G=\<\a,\b,\c\mid\a^2,\b^2,\c^2,(\a\b)^{m_{\a\b}}, (\a\c)^{m_{\a\c}},(\b\c)^{m_{\b\c}}>,
\end{equation*}
with $m_{\a\b}=\infty$ if there is no relation between $\a$ and $\b$, and similarly for $m_{\a\c}$ and $m_{\b\c}$.
(It is not always possible to represent the group $G$ as above, see Table \ref{tab:various_groups}.)
Following Bourbaki \cite{Bo-68} we introduce the two matrices
\begin{equation}
\label{eq:matrices_Bourbaki}
\left(\begin{array}{ccc}
     1&m_{\a\b}&m_{\a\c}\\
     m_{\a\b}&1&m_{\b\c}\\
     m_{\a\c}&m_{\b\c}&1
     \end{array}\right)\qquad \text{and}\qquad
     \left(\begin{array}{ccc}
     1&-\cos\left(\frac{\pi}{m_{\a\b}}\right)&-\cos\left(\frac{\pi}{m_{\a\c}}\right)\smallskip\\
     -\cos\left(\frac{\pi}{m_{\a\b}}\right)&1&-\cos\left(\frac{\pi}{m_{\b\c}}\right)\smallskip\\
     -\cos\left(\frac{\pi}{m_{\a\c}}\right)&-\cos\left(\frac{\pi}{m_{\b\c}}\right)&1
     \end{array}\right).
\end{equation}
The first one is called the Coxeter matrix, see Def.~4 in \cite[Ch.~IV]{Bo-68}. The second one is used in \cite{Bo-68} to define a quadratic form, whose property of being non-degenerate eventually characterizes the finiteness of the group $G$, see Thm~2 in \cite[Ch.~V]{Bo-68}.

Our point here is to remark the strong link between the matrix on the right-hand side of \eqref{eq:matrices_Bourbaki} and the covariance matrix, which by Lemma \ref{lem:exact_value_angles_triangle} may be rewritten as the cosine matrix
\begin{equation}
\label{eq:covariance_matrix_ANGLES}
\left(\begin{array}{ccc}
     1&-\cos(\gamma)&-\cos(\beta)\\
     -\cos(\gamma)&1&-\cos(\alpha)\\
     -\cos(\beta)&-\cos(\alpha)&1
     \end{array}\right).
\end{equation}
There are, however, two differences between the matrices \eqref{eq:covariance_matrix_ANGLES} and \eqref{eq:matrices_Bourbaki}. The first one is that in the infinite group case, all non-diagonal coefficients of the matrix \eqref{eq:covariance_matrix_ANGLES} are in the open interval $(-1,1)$, while if there is no relation between $\a$ and $\b$ (say), then $m_{\a\b}=\infty$ and $-\cos(\frac{\pi}{m_{\a\b}})=-1$. See \cite{DaDe-98} for a rather general study of cosine matrices \eqref{eq:covariance_matrix_ANGLES}.

The second difference is about the finite group case. Take any two step sets which are obtained the one from the other by a reflection (see Figure \ref{fig:3D_tandem} for an example). Then the group has the exact same structure and thus the matrix of \cite{Bo-68} is unchanged. On the other hand, the matrix \eqref{eq:covariance_matrix_ANGLES} changes after a reflection (Kreweras on the left, reflected Kreweras on the right):
\begin{equation*}
     \left(\begin{array}{rrr}
     1&\frac{1}{2}&\frac{1}{2}\smallskip\\
     \frac{1}{2}&1&\frac{1}{2}\smallskip\\
     \frac{1}{2}&\frac{1}{2}&1
     \end{array}\right)\qquad\text{and}\qquad
     \left(\begin{array}{rrr}
     1&-\frac{1}{2}&-\frac{1}{2}\smallskip\\
     -\frac{1}{2}&1&\frac{1}{2}\smallskip\\
     -\frac{1}{2}&\frac{1}{2}&1
     \end{array}\right).
\end{equation*}

\begin{figure}[bht]
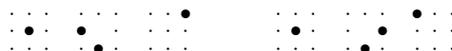

\begin{center}
\begin{tabular}{c@{\qquad}c}
\Stepset000 010 000   010 10 000   000 000 001 &
\Stepset000 010 000   010 01 000   000 000 100
\end{tabular}  
\end{center}
  \caption{On the left, Kreweras 3D model. On the right, the reflection of Kreweras 3D with respect to the $x$-axis, which can be thought of as a 3D tandem model}
\label{fig:3D_tandem}
\end{figure}

\subsubsection*{Polar angles and Gram matrix}

It is possible to compute the angles between the three segments connecting the origin to the vertices of the triangle $\langle x,y,z\rangle$. These angles may also be interpreted as the lengths $A=\overline{yz}$, $B=\overline{xz}$ and $C=\overline{xy}$ of the sides of the triangle, see \cite[18.6.6]{Be-87}. By \cite[18.6.12.2]{Be-87} they are the complements to $\pi$ of the polar angles (see Definition \ref{Def:polar_triangle}).

\begin{Lemma}
\label{lem:exact_value_angles_vectors}
Let $O$ denote the origin $(0,0,0)$. The angles between the vectors $\vec{Ox}$, $\vec{Oy}$ and $\vec{Oz}$ are given by
\begin{equation*}
     A =  \arccos\left(\frac{bc -a}{\sqrt{1-b^2}\sqrt{1-c^2}}\right),\quad
     B =  \arccos\left(\frac{ac -b}{\sqrt{1-a^2}\sqrt{1-c^2}}\right),\quad
     C =  \arccos\left(\frac{ab -c}{\sqrt{1-a^2}\sqrt{1-b^2}}\right).
\end{equation*}
\end{Lemma}

As it should be, the quantity $\frac{bc -a}{\sqrt{1-b^2}\sqrt{1-c^2}}$ (and its cyclic permutations as well) in Lemma~\ref{lem:exact_value_angles_vectors} belongs to $(-1,1)$. Indeed if $bc\geq a$ then 
\begin{equation*}
     \frac{bc -a}{\sqrt{1-b^2}\sqrt{1-c^2}}<1 \text{ iff } (bc -a)^2<(1-b^2)(1-c^2) \text{ iff } 1-a^2-b^2-c^2+2abc>0.
\end{equation*}
The quantity $1-a^2-b^2-c^2+2abc$ is positive because it is the determinant of the covariance matrix \eqref{eq:covariance_matrix}, which is assumed positive definite. In the case $bc\leq a$ we would prove similarly that $\frac{bc -a}{\sqrt{1-b^2}\sqrt{1-c^2}}>-1$.

\begin{proof}[Proof of Lemma \ref{lem:exact_value_angles_vectors}]
The angles are easily computed: if $e_1$, $e_2$ and $e_3$ are the vectors of the canonical basis and $L^{-1}$ is as in \eqref{eq:Cholesky_L^-1},
\begin{equation}
\label{eq:value_angles}
     \langle L^{-1}e_{1},L^{-1}e_{2}\rangle = \Vert L^{-1}e_{1}\Vert\cdot\Vert L^{-1}e_{j}\Vert\cdot \cos C,
\end{equation}
and cyclic permutations of the above identities hold. The formulas stated in Lemma \ref{lem:exact_value_angles_vectors} follow from \eqref{eq:value_angles}, after having computed the norms and the dot products of the columns of $L^{-1}$.

An alternative proof is to invert the covariance matrix \eqref{eq:covariance_matrix} and to use the orthogonality relations between the angles and their polar angles, see Definition \ref{Def:polar_triangle}.
\end{proof}

Finally, we stress that the covariance matrix may be interpreted as the Gram matrix
\begin{equation*}
     \left(\begin{array}{ccc}
     \langle u,u\rangle & \langle u,v\rangle & \langle u,w\rangle\\
     \langle u,v\rangle & \langle v,v\rangle & \langle v,w\rangle\\
     \langle u,w\rangle&\langle v,w\rangle&\langle w,w\rangle\end{array}\right),
\end{equation*}
where $u,v,w$ are the three vectors on the sphere which are the columns of the matrix
\begin{equation*}
     \left(\begin{array}{ccc}
     \frac{\sqrt{1-a^2-b^2-c^2+2abc}}{\sqrt{1-c^2}} & 0 & 0\\
     \frac{bc-a}{\sqrt{1-c^2}} & -\sqrt{1-c^2} & 0\\
     b&c&1\end{array}\right).
\end{equation*}

\subsubsection*{The reverse construction}

\label{rem:reverse_construction}
Our general construction consists in associating to every model of walks the covariance matrix \eqref{eq:covariance_matrix}, and thereby a spherical triangle with angles $\alpha,\beta,\gamma$ as in Lemma~\ref{lem:exact_value_angles_triangle}. It is natural to ask about the converse: is it possible to realize any spherical triangle as a walk triangle? The answer turns out to be positive, if we allow weighted walks. 
 
More specifically, let $\langle x,y,z\rangle$ be an arbitrary spherical triangle, having angles $\alpha,\beta,\gamma\in(0,\pi)$. Introduce $a,b,c\in(-1,1)$ such that \eqref{eq:exact_value_angles_triangle} holds. Let finally $(U,V,W)$ be a triplet of independent random variables (actually, having non-correlated variables is enough) with unit variances. Introduce the random variables 
\begin{equation}
\label{eq:UVWXYZ}
     \left(\begin{array}{c} Z\\Y\\X\end{array}\right)=L \left(\begin{array}{c} U\\V\\W\end{array}\right),
\end{equation}
where $L$ is the matrix \eqref{eq:Cholesky_L} appearing in the Cholesky decomposition of the matrix $\cov$. Then by construction the covariance matrix of $(X,Y,Z)$ is \eqref{eq:covariance_matrix} and its spherical triangle has angles $\alpha,\beta,\gamma$. In conclusion, the random walk model whose increment distribution is the same as $(X,Y,Z)$ given by \eqref{eq:UVWXYZ} has a spherical triangle with generic angles $\alpha,\beta,\gamma$.

\subsection{Open problems}
Besides the open problems listed in \cite[Sec.~9]{BoBMKaMe-16}, let us mention the following:
 
\subsubsection*{Singularity analysis}
Is it possible to obtain similar results on non-D-finiteness of Hadamard models using the Hadamard product of generating functions? This would mean to prove Corollaries \ref{cor:non-D-finite_(1,2)-type_Hadamard} and \ref{cor:non-D-finite_(2,1)-type_Hadamard} directly, at the level of generating functions.

\subsubsection*{3D Kreweras model}
This is clearly the model for which we can find the greatest number of estimations in the literature; its triangle is equilateral with angle $2\pi/3$. Let us quickly give a chronological list (probably non-exhaustive): 
\begin{itemize}
     \item $[5.15,5.16]$ by Costabel (2008, \cite{Da-17})
     \item $5.159$ by Ratzkin and Treibergs (2009, \cite{RaTr-09a,RaTr-09b})
     \item $5.1589$ by Bostan, Raschel and Salvy (2012, \cite{BoRaSa-12})
     \item $5.162$ by Balakrishna (2013, \cite{Ba-13a})
     \item $5.1606$ by Balakrishna (2013, \cite{Ba-13b})
     \item $5.1591452$ by Bacher, Kauers and Yatchak (2016, \cite{BaKaYa-16})
     \item $5.159145642466$ Guttmann (2017, \cite{Gu-17})
     \item $5.159145642470$ by our result 
\end{itemize}
What is the exact value? Is it a rational number? The triangle associated to Kreweras model, which corresponds to the tetrahedral partition of the sphere, is also related to minimal $4$-partitions of $\mathbb S^2$, see \cite{HeHOTe-10}.

\appendix

\section{Elementary spherical geometry and Dirichlet eigenvalues of spherical triangles}
\label{sec:spherical_geometry}

\subsection{Elementary spherical geometry}
Our main source is the book \cite{Be-87} by Berger. Spherical triangles have been introduced in Definition \ref{def:spherical_triangle}. A spherical digon is a domain bounded by two great arcs of circles, see Figure~\ref{fig:digon} and \cite[18.3.8.2]{Be-87}. 

A natural operation in spherical geometry is to take the polar spherical triangle; see \cite[18.3.8.2]{Be-87} and \cite[18.6.12]{Be-87} for more details.

\begin{Definition}[polar triangle]
\label{Def:polar_triangle}
Let $\langle x,y,z\rangle$ be a spherical triangle in the sense of Definition \ref{def:spherical_triangle}. Define the triplet $(x',y',z')$ by the conditions
\begin{equation*}
     \left\{\begin{array}{ll}
     \langle x',y\rangle = \langle x',z\rangle=0,&\quad \langle x',x\rangle>0,\\
     \langle y',z\rangle = \langle y',x\rangle=0,&\quad \langle y',y\rangle>0,\\
     \langle z',x\rangle = \langle z',y\rangle=0,&\quad \langle z',z\rangle>0.
     \end{array}\right.
\end{equation*}
Then $\langle x',y',z'\rangle$ is a spherical triangle, called the polar triangle of $\langle x,y,z\rangle$.
\end{Definition} 
The polar transformation is involutive, and the equilateral right triangle is invariant. There is no simple formula relating the eigenvalues of a spherical triangle to that of its polar triangle. See Figure \ref{fig:polar_triangles} for examples.

\begin{figure}
\includegraphics[width=0.20\textwidth]{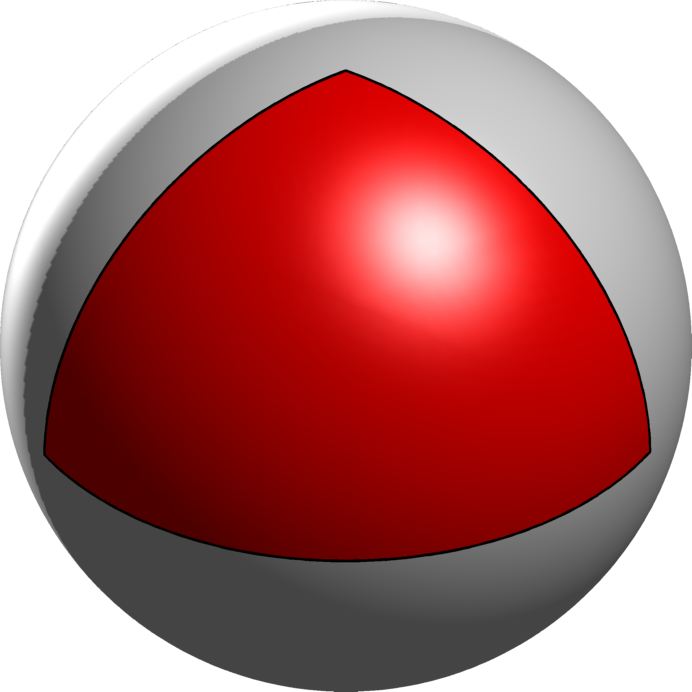}\qquad
\includegraphics[width=0.20\textwidth]{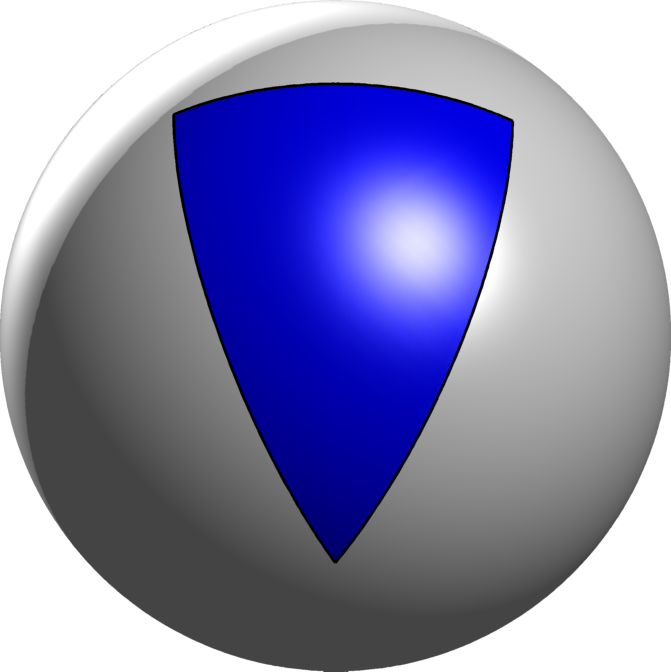}\qquad
\includegraphics[width=0.20\textwidth]{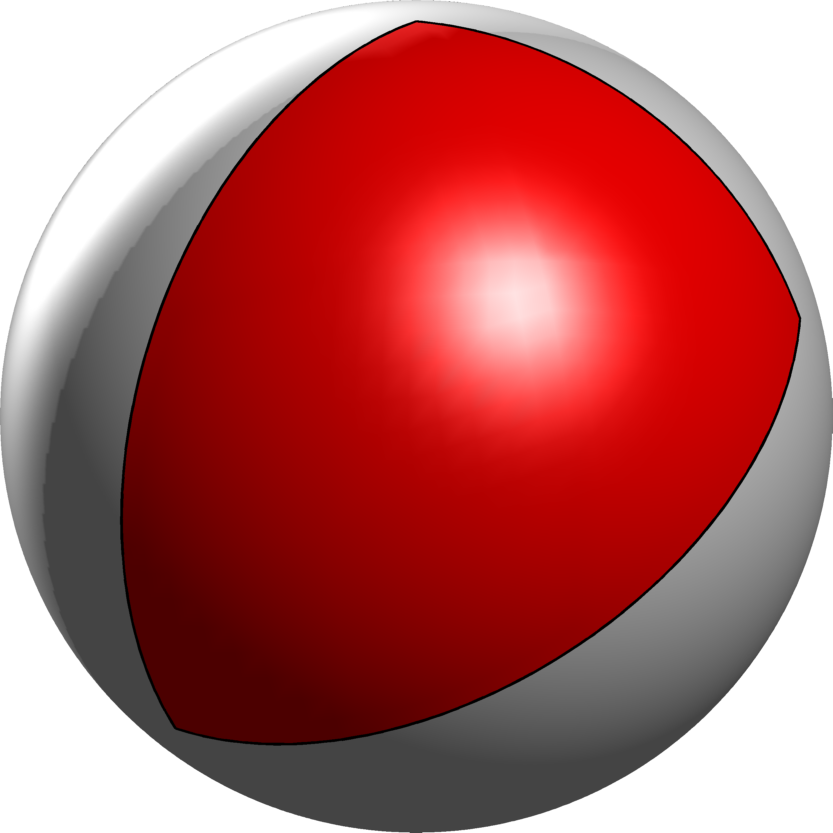}\qquad
\includegraphics[width=0.20\textwidth]{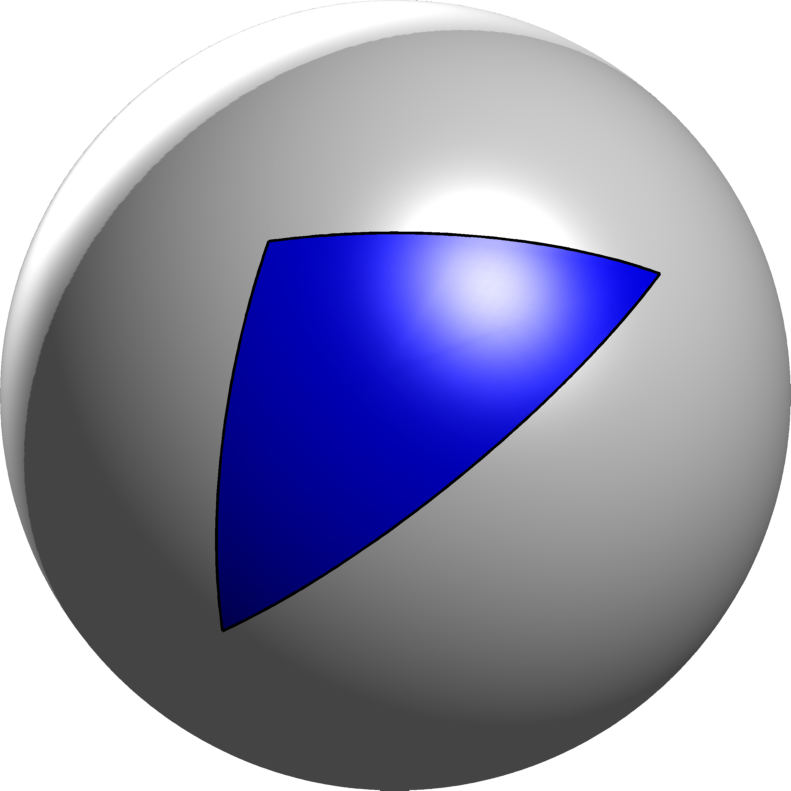}
\caption{Two triangles (color red) and their polar triangles (in blue), see Definition \ref{Def:polar_triangle}}
\label{fig:polar_triangles}
\end{figure}

Interestingly, polar cones already appear in \cite{GaRa-16} (resp.\ \cite{GaRa-14}) to compute the exponential decay of the survival probability of random walks (resp.\ the exponential decay and the critical exponent of the Brownian survival probability) in cones.

\subsection{Some properties of the principal eigenvalue}
\label{subsec:properties_eigenvalue}
Our main reference here is the book \cite{Da-88} of Dauge.

\subsubsection*{Monotonicity and regularity of the eigenvalues}
\begin{Lemma}[Lemma 18.5 in \cite{Da-88}]
\label{lem:eigenvalue_monotonic}
Let $T_1$ and $T_2$ be two simply connected domains on $\mathbb S^2$. If $T_1\subset T_2$ then
\begin{equation*}
     \lambda_1(T_1)\geq\lambda_1(T_2).
\end{equation*}
\end{Lemma}
In particular, as any spherical triangle is included in a half-sphere (whose principal eigenvalue equals $2$), one has the universal lower bound 
\begin{equation*}
     \lambda_1(T)\geq2 
\end{equation*}
for any spherical triangle $T$. (By \eqref{eq:DW_formula_exponent}, this implies that the critical exponent $\lambda$ should be bigger than $\frac{5}{2}$.)

Classical arguments in perturbation theory for operators \cite{kato} state that analytic perturbations of the operator induce analytic perturbations of the eigenvalues, see in particular \cite[Lem.~2.1]{HiJu-09} in our context. 
\begin{Lemma}
\label{lem:eigenvalue_analytic}
The function $\lambda_1(T)=\lambda_1(\alpha,\beta,\gamma)$ is analytic in the angles $\alpha,\beta,\gamma$.
\end{Lemma}

A consequence of Lemma \ref{lem:eigenvalue_analytic} is that a generic triangle has an irrational (and even transcendental) principal eigenvalue $\lambda_1$.

\begin{Lemma}
\label{lem:eigenvalue_infinity}
As one of the angles goes to $0$, $\lambda_1$ goes to infinity. 
\end{Lemma}
\begin{proof}
Lemma \ref{lem:eigenvalue_infinity} is a simple consequence of Lemma \ref{lem:eigenvalue_monotonic} and the fact that each spherical triangle can be included in any of the digons determined by its angles. Indeed, suppose the triangle $T$ has an angle equal to $\alpha$. Then $T$ is included in the digon $D_\alpha$ with angle $\alpha$ and
\[ \lambda_1(T) \geq \lambda_1(D_\alpha) = \frac{\pi}{\alpha}\left(\frac{\pi}{\alpha}+1\right).\]
We can notice immediately that if $\alpha \to 0$ then the first eigenvalue of $T$ goes to infinity.
\end{proof}

\subsubsection*{Revolution cones}
We now compute the spectrum of a revolution cone (or solid angle) in arbitrary dimension $d\geq 2$. Introduce some notation. We fix a half-axis $A$ in $\mathbb R^d$ and for any $x\neq0$ denote by $\theta(x)\in[0,\pi]$ the angle between the axes $A$ and $\vec{x}$. By definition, the revolution cone with apex angle $\zeta$ is (see Figure \ref{fig:spherical_cap})
\begin{equation}
\label{eq:definition_revolution_cone}
     K(\zeta)= \{x\in\mathbb R^d\setminus \{0\} : \theta(x)\in(0,\zeta)\}.
\end{equation}
Its section on the sphere is the circle $C(\zeta)=K(\zeta) \cap \mathbb S^2$.

\begin{Lemma}[Proposition 18.10 in \cite{Da-88}]
\label{lem:spectrum_revolution}
The spectrum of $C(\zeta)$ is the set of positive $\nu(\nu+d-2)$ for which there is $m\in\mathbb N$ such that $\mathsf P_{\nu}^m(\cos \zeta)=0$, where $\mathsf P_{\nu}^m$ denotes the $m$th Legendre function of the first kind.
\end{Lemma}
Notice that \cite[Prop.~18.10]{Da-88} computes the spectrum of the cone $K(\zeta)$, not of its section $C(\zeta)$. However the eigenvalues $\lambda_i(K)$ of a cone $K$ are directly related to the eigenvalues of its section $C=K\cap \mathbb S^2$, namely (see, e.g., \cite[18.3]{Da-88})
\begin{equation}
\label{eq:expression_lambda_i(K)}
     \lambda_i(K)=\sqrt{\lambda_i(C)+\left(1-\frac{d}{2}\right)^2}+\left(1-\frac{d}{2}\right).
\end{equation}

\subsubsection*{A few remarkable spherical triangles}

Consider triangles with angles
\begin{equation*}
     \left(\frac{\pi}{p},\frac{\pi}{q},\frac{\pi}{r}\right), \quad \text{with}\ p,q,r\in\mathbb N\setminus\{0,1\}.
\end{equation*}
As recalled in \cite{Be-83,Da-88}, the only possible triplets are
\begin{itemize}
     \item $(2,3,3)$ tetrahedral group;
     \item $(2,3,4)$ octahedral group;
     \item $(2,3,5)$ icosahedral group;
     \item $(2,2,r)$ dihedral group or order $2r\geq4$
\end{itemize}
Each triplet above corresponds to a tiling of the sphere. See Figures \ref{fig:tilings} and \ref{fig:some_further_tilings} for a few examples. Denote by $T_{(p,q,r)}$ the associated triangle when it exists.

\begin{Lemma}[Theorem 6 in \cite{Be-83}]
\label{lem:thm6_Be-83}
The eigenvalues of $T_{(p,q,r)}$ have the form $\nu_{(p,q,r)}(\nu_{(p,q,r)}+1)$, with ($\ell_1,\ell_2\in\mathbb N$)
\begin{itemize}
     \item $\nu_{(2,3,3)}=6+3\ell_1+4\ell_2$;
     \item $\nu_{(2,3,4)}=9+6\ell_1+6\ell_2$;
     \item $\nu_{(2,3,5)}=15+6\ell_1+10\ell_2$;
     \item $\nu_{(2,2,r)}=r+1+2\ell_1+r\ell_2$.
\end{itemize}
\end{Lemma}

\end{document}